\documentclass[a4paper,reqno,10pt]{amsart}
\usepackage{graphicx,verbatim, amsmath, amssymb, amsthm, amsfonts, epsfig, ifthen,mathtools,
caption,enumerate,mathrsfs}
\usepackage[all]{xy}
\xyoption{all}
\usepackage{hyperref}

\newtheorem{proposition}{Proposition}[section]
\newtheorem{theorem}[proposition]{Theorem}
\newtheorem{corollary}[proposition]{Corollary}
\newtheorem{lemma}[proposition]{Lemma}

\theoremstyle{definition}
\newtheorem{example}[proposition]{Example}
\newtheorem{definition}[proposition]{Definition}

\newtheorem{problem}[proposition]{Problem}

\theoremstyle{remark}
\newtheorem{remark}[proposition]{Remark}

\numberwithin{equation}{section}
\usepackage[usenames]{color}

\usepackage{color}

\newcommand{\margincolor}{red}
\definecolor{darkgreen}{rgb}{0,0.7,0}

\newcounter{margincounter}
\newcommand{\marginnum}{
\ifnum\value{margincounter}<10
\textcolor{\margincolor}{\begin{picture}(0,0)\put(2.2,2.4){\circle{9}}\end{picture}\footnotesize\arabic{margincounter}}
\else\ifnum\value{margincounter}<100
\textcolor{\margincolor}{\begin{picture}(0,0)\put(4.256,2.5){\circle{11}}\end{picture}\footnotesize\arabic{margincounter}}
\else
\textcolor{\margincolor}{\begin{picture}(0,0)\put(6.8,2.5){\circle{14}}\end{picture}\footnotesize\arabic{margincounter}}
\fi\fi
}

\newcommand{\newword}[1]{\emph{#1}}
\newcommand{\integers}{\mathbb Z}

\newcommand{\reals}{\mathbb R}
\newcommand{\complex}{\mathbb C}

\newcommand{\bic}{{\mathrm {biC}}}

\DeclareMathOperator{\D}{\mathsf{D}}
\DeclareMathOperator{\ind}{\mathsf{ind}}
\DeclareMathOperator{\Con}{\mathsf{Con}}
\DeclareMathOperator{\Com}{\mathsf{Con^c}}
\DeclareMathOperator{\Comb}{\mathsf{Con^{ca}}}
\DeclareMathOperator{\AlgCon}{\mathsf{AlgCon}}
\DeclareMathOperator{\jirr}{\mathsf{j-Irr}}
\DeclareMathOperator{\mirr}{\mathsf{m-Irr}}
\DeclareMathOperator{\cjirr}{\mathsf{j-Irr^c}}
\DeclareMathOperator{\cmirr}{\mathsf{m-Irr^c}}

\DeclareMathOperator{\brick}{\mathsf{brick}}

\DeclareMathOperator{\tors}{\mathsf{tors}}
\DeclareMathOperator{\ftors}{\mathsf{f-tors}}
\DeclareMathOperator{\torf}{\mathsf{torf}}

\DeclareMathOperator{\sttilt}{\mathsf{s}\tau\mathsf{-tilt}}
\DeclareMathOperator{\itrigid}{\mathsf{i}\tau\mathsf{-rigid}}

\DeclareMathOperator{\trigid}{\tau\mathsf{-rigid}}
\DeclareMathOperator{\trigidpair}{\tau\mathsf{-rigid-pair}}
\DeclareMathOperator{\itrigidpair}{\mathsf{i}\tau\mathsf{-rigid-pair}}
\DeclareMathOperator{\ttiltpair}{\tau\mathsf{-tilt-pair}}
\DeclareMathOperator{\sbrick}{\mathsf{sbrick}}

\DeclareMathOperator{\Sub}{\mathsf{Sub}}
\DeclareMathOperator{\Filt}{\mathsf{Filt}}

\DeclareMathOperator{\Ideals}{\mathsf{ideal}}
\DeclareMathOperator{\add}{\mathsf{add}}
\DeclareMathOperator{\Fac}{\mathsf{Fac}}

\let\mod\relax
\DeclareMathOperator{\mod}{\mathsf{mod}}

\DeclareMathOperator{\Hom}{\mathsf{Hom}}
\DeclareMathOperator{\End}{\mathsf{End}}

\DeclareMathOperator{\rad}{\mathsf{rad}}
\DeclareMathOperator{\Rad}{\mathsf{Rad}}

\DeclareMathOperator{\Ext}{\mathsf{Ext}}
\DeclareMathOperator{\Kernel}{\mathsf{Ker}}
\DeclareMathOperator{\Cokernel}{\mathsf{Cok}}
\DeclareMathOperator{\T}{\mathsf{T}}
\DeclareMathOperator{\Image}{\mathsf{Im}}
\DeclareMathOperator{\ann}{\mathsf{ann}}
\let\top\relax
\DeclareMathOperator{\top}{\mathsf{top}}

\DeclareMathOperator{\id}{\mathsf{id}}

\DeclareMathOperator{\Hasse}{\mathsf{Hasse}}
\DeclareMathOperator{\con}{\mathsf{con}}
\DeclareMathOperator{\com}{\mathsf{con^c}}
\DeclareMathOperator{\comb}{\mathsf{con^{ca}}}
\DeclareMathOperator{\simp}{\mathsf{S}}

\DeclareMathOperator{\wide}{\mathsf{wide}}
\DeclareMathOperator{\spanv}{\mathsf{span}}

\newcommand{\set}[1]{{\left\{ #1 \right\}}}
\newcommand{\pidown}{\pi_\downarrow}
\newcommand{\piup}{\pi^\uparrow}

\newcommand{\join}{\vee}
\newcommand{\meet}{\wedge}
\renewcommand{\Join}{\bigvee}
\newcommand{\Meet}{\bigwedge}
\newcommand{\joindn}{\join_\dn}
\newcommand{\meetdn}{\meet_\dn}
\newcommand{\Joindn}{\Join_\dn}
\newcommand{\Meetdn}{\Meet_\dn}
\newcommand{\xto}{\xrightarrow}
\newcommand{\cyc}{{\mathrm{cyc}}}
\newcommand{\op}{{\mathrm{op}}}

\newcommand{\sym}{\mathfrak{S}}

\newcommand{\E}{\mathcal{E}}
\newcommand{\CC}{{\mathcal C}}
\newcommand{\WW}{{\mathcal W}}
\newcommand{\TT}{{\mathcal T}}
\newcommand{\II}{{\mathcal I}}
\renewcommand{\AA}{{\mathcal A}}
\newcommand{\BB}{{\mathcal B}}
\renewcommand{\S}{{\mathcal S}}
\renewcommand{\SS}{\mathscr{S}}
\newcommand{\UU}{\mathcal{U}}
\newcommand{\VV}{\mathcal{V}}

\newcommand{\FF}{{\mathcal F}}
\newcommand{\X}{{\mathcal X}}

\newcommand{\UUU}{\mathscr{U}}

\newcommand{\forces}[1][]{\mathrel{\rightsquigarrow_{\mathrm{#1}}}}
\newcommand{\forceseq}[1][]{\mathrel{\leftrightsquigarrow_{\mathrm{#1}}}}

\usepackage{adjustbox}

\newenvironment{itemize}{\begin{enumerate}[$\bullet$]}{\end{enumerate}} 

\newcommand{\rs}[1]{\setcounter{enumi}{#1}}

\def\newboxedcommand#1#2 
{%
 \def\newboxedcommandlocala##1##2.{##2}%
 \edef\newboxedcommandlocalb{\expandafter\newboxedcommandlocala\string#1.}%
 \expandafter\newsavebox\csname\newboxedcommandlocalb savebox\endcsname%
 \expandafter\sbox\csname\newboxedcommandlocalb savebox\endcsname{#2}%
 \expandafter\newlength\csname\newboxedcommandlocalb largeurbox\endcsname%
 \expandafter\settowidth\csname\newboxedcommandlocalb largeurbox\endcsname{\usebox{\csname\newboxedcommandlocalb savebox\endcsname}}%
 \edef#1{\noexpand\begin{minipage}{\csname\newboxedcommandlocalb largeurbox\endcsname}\usebox{\csname\newboxedcommandlocalb savebox\endcsname}\noexpand\end{minipage}}%
}

\newsavebox\locboxinminipage
\newlength\locboxinminipagel
\newcommand{\boxinminipage}[1]
{%
 \sbox\locboxinminipage{#1}%
 \settowidth\locboxinminipagel{\usebox{\locboxinminipage}}%
 \begin{minipage}{\locboxinminipagel}\usebox{\locboxinminipage}\end{minipage}%
}

\newcommand{\R}{\mathbb{R}}
\newcommand{\dn}{\downarrow}

\newboxedcommand{\Sa}{\footnotesize$\xymatrix@C=.15cm@R=.25cm@!=0cm@!=0cm{1}$}
\newboxedcommand{\Sb}{\footnotesize$\xymatrix@C=.15cm@R=.25cm@!=0cm@!=0cm{2}$}
\newboxedcommand{\Sc}{\footnotesize$\xymatrix@C=.15cm@R=.25cm@!=0cm{3}$}
\newboxedcommand{\ab}{\footnotesize$\xymatrix@C=.15cm@R=.25cm@!=0cm{1 & \\ & 2}$}
\newboxedcommand{\bc}{\footnotesize$\xymatrix@C=.15cm@R=.25cm@!=0cm{2 & \\ & 3}$}
\newboxedcommand{\ba}{\footnotesize$\xymatrix@C=.15cm@R=.25cm@!=0cm{& 2  \\ 1 &}$}
\newboxedcommand{\cb}{\footnotesize$\xymatrix@C=.15cm@R=.25cm@!=0cm{& 3  \\ 2 &}$}
\newboxedcommand{\abc}{\footnotesize$\xymatrix@C=.15cm@R=.25cm@!=0cm{1 & & \\ & 2 & \\ & & 3}$}
\newboxedcommand{\cba}{\footnotesize$\xymatrix@C=.15cm@R=.25cm@!=0cm{& & 3  \\ & 2 & \\ 1 & & }$}
\newboxedcommand{\acb}{\footnotesize$\xymatrix@C=.15cm@R=.25cm@!=0cm{1 & & 3 \\ & 2  &}$}
\newboxedcommand{\bac}{\footnotesize$\xymatrix@C=.15cm@R=.25cm@!=0cm{& 2 & \\ 1 & & 3}$}
\newboxedcommand{\bacb}{\footnotesize $\xymatrix@C=.15cm@R=.25cm@!=0cm{& 2 & \\ 1 & & 3 \\ & 2 &}$}

\newboxedcommand{\Pun}{\footnotesize$\xymatrix@C=.15cm@R=.25cm@!=0cm@!=0cm{& & 1 \\ & 1 & & 2 \\2 & & & & 2}$}
\newboxedcommand{\Pdeux}{\footnotesize$\xymatrix@C=.15cm@R=.25cm@!=0cm@!=0cm{& & 2 \\ & 1 & & 2 \\ 1 & & & & 1}$}
\newboxedcommand{\Punm}{\footnotesize$\xymatrix@C=.15cm@R=.25cm@!=0cm@!=0cm{& & 2 \\ & 2 & & 1 & & 2 \\1 & & & & 1 & & 2 \\ & & & & & & & 1}$}
\newboxedcommand{\Pdeuxm}{\footnotesize$\xymatrix@C=.15cm@R=.25cm@!=0cm@!=0cm{& & 1 \\ & 1 & & 2 & & 1 \\ 2 & & & & 2 & & 1\\ & & & & & & & 2 }$}
\newboxedcommand{\Sddu}{\footnotesize$\xymatrix@C=.15cm@R=.25cm@!=0cm@!=0cm{2 \\ & 2 \\ & & 1}$}
\newboxedcommand{\Suud}{\footnotesize$\xymatrix@C=.15cm@R=.25cm@!=0cm@!=0cm{1 \\ & 1 \\ & & 2 }$}
\newboxedcommand{\Sduu}{\footnotesize$\xymatrix@C=.15cm@R=.25cm@!=0cm@!=0cm{2 \\ & 1 \\ & & 1}$}
\newboxedcommand{\Sudd}{\footnotesize$\xymatrix@C=.15cm@R=.25cm@!=0cm@!=0cm{1 \\ & 2 \\ & & 2 }$}
\newboxedcommand{\Sdu}{\footnotesize$\xymatrix@C=.15cm@R=.25cm@!=0cm@!=0cm{2 \\ & 1}$}
\newboxedcommand{\Sud}{\footnotesize$\xymatrix@C=.15cm@R=.25cm@!=0cm@!=0cm{1 \\ & 2 }$}
\newboxedcommand{\Sdd}{\footnotesize$\xymatrix@C=.15cm@R=.25cm@!=0cm@!=0cm{2 \\ & 2 }$}
\newboxedcommand{\Suu}{\footnotesize$\xymatrix@C=.15cm@R=.25cm@!=0cm@!=0cm{1 \\ & 1 }$}
\newboxedcommand{\Suudd}{\footnotesize$\xymatrix@C=.15cm@R=.25cm@!=0cm@!=0cm{& 1 \\2 & & 1 \\ & & & 2 }$}
\newboxedcommand{\Sudud}{\footnotesize$\xymatrix@C=.15cm@R=.25cm@!=0cm@!=0cm{1 \\& 2 & & 1 \\ & & 2 }$}

\newboxedcommand{\KPu}{\footnotesize$\xymatrix@C=.15cm@R=.25cm@!=0cm@!=0cm{& 2 & \\ 1 & & 1 }$}
\newboxedcommand{\KPd}{\footnotesize$\xymatrix@C=.15cm@R=.25cm@!=0cm@!=0cm{& 2 & & 2 &\\ 1 & & 1 & & 1}$}
\newboxedcommand{\KIu}{\footnotesize$\xymatrix@C=.15cm@R=.25cm@!=0cm@!=0cm{2 & & 2\\ & 1 & }$}
\newboxedcommand{\KId}{\footnotesize$\xymatrix@C=.15cm@R=.25cm@!=0cm@!=0cm{2 & & 2 & & 2\\ & 1 & & 1 & }$}

\newboxedcommand{\spR}{\footnotesize$\xymatrix@C=.15cm@R=.25cm@!=0cm@!=0cm{\reals}$}
\newboxedcommand{\spC}{\footnotesize$\xymatrix@C=.15cm@R=.25cm@!=0cm@!=0cm{\complex}$}
\newboxedcommand{\spRC}{\footnotesize$\xymatrix@C=.15cm@R=.25cm@!=0cm@!=0cm{\reals \\ \complex}$}
\newboxedcommand{\spRRC}{\footnotesize$\xymatrix@C=.15cm@R=.25cm@!=0cm@!=0cm{\reals & & \reals \\ & \complex}$}

\newcommand{\ta}[1]{\objectmargin{0cm}\xybox{(0,0)*+[F-]{#1}}}
\newcommand{\tb}[2]{\objectmargin{0cm}\xybox{(0,0)*+[F-]{\left. #1 \middle| #2 \right.}}}
\newcommand{\tc}[3]{\objectmargin{0cm}\xybox{(0,0)*+[F-]{\left. #1 \middle| #2 \middle| #3 \right.}}}
\newcommand{\cirl}[1]{\objectmargin{.0cm}\xybox{(0,0)*++[o][F-]{#1}}} 
\newcommand{\bcirl}[1]{\objectmargin{.03cm}\xybox{(0,0)*++[o][F-]{#1}}} 

\title[Lattices of torsion classes]{Lattice theory of torsion classes: Beyond {\Large$\tau$}-tilting theory}
\author[Demonet]{Laurent Demonet}
\address{L. Demonet: Graduate School of Mathematics, Nagoya University, Chikusa-ku, Nagoya, 464-8602 Japan}
\email{Laurent.Demonet@normalesup.org}
\urladdr{http://www.math.nagoya-u.ac.jp/~demonet/}
\author[Iyama]{Osamu Iyama}
\address{O. Iyama: Graduate School of Mathematical Sciences, University of Tokyo, 3-8-1
Komaba Meguro-ku Tokyo 153-8914, Japan}
\email{iyama@math.nagoya-u.ac.jp}
\urladdr{http://www.math.nagoya-u.ac.jp/~iyama/}
\author[Reading]{Nathan Reading}
\address{N. Reading:  Department of Mathematics, North Carolina State University, Raleigh, NC 27695-8205, USA}
\email{reading@math.ncsu.edu}
\urladdr{http://www4.ncsu.edu/~nreadin/}
\author[Reiten]{Idun Reiten}
\address{I. Reiten: Department of Mathematical Sciences, Norges teknisk-naturvitenskapelige universitet, 7491 Trondheim, Norway}
\author[Thomas]{Hugh Thomas}
\address{H. Thomas: D\'epartement de math\'ematiques, Universit\'e du Qu\'ebec \`a Montr\'eal,
C.P. 8888, succ. Centre-ville Montr\'eal (Qu\'ebec), H3C 3P8 Canada}
\email{hugh.ross.thomas@gmail.com}
\thanks{Laurent Demonet's work on this project was partially supported by JSPS Grant-in-Aid for Young Scientist (B) 26800008 and 17K14160. Osamu Iyama's work on this project was partially supported by JSPS Grant-in-Aid for Scientific Research (B) 24340004, (B) 16H03923, (C) 23540045 and (S) 15H05738.
Nathan Reading's work on this project was partially supported by the National Science Foundation under grant numbers DMS-1101568 and DMS-1500949.
Idun Reiten was supported by the FriNat grants 196600 and 231000 from the Research Council of Norway.
Hugh Thomas's work on this project was partially supported by an NSERC Discovery Grant and the Canada Research Chairs program.}


\begin{document}

\begin{abstract}
 The aim of this paper is to establish a lattice theoretical framework to study the partially ordered set $\tors A$ of torsion classes over a finite-dimensional algebra $A$.
 We show that $\tors A$ is a complete lattice which enjoys very strong properties, as \emph{bialgebraicity} and \emph{complete semidistributivity}.
 Thus its Hasse quiver carries the important part of its structure, and we introduce the brick labelling of its Hasse quiver and use it to study lattice congruences of $\tors A$. In particular, we give a representation-theoretical interpretation of the so-called \emph{forcing order}, and we prove that $\tors A$ is \emph{completely congruence uniform}. When $I$ is a two-sided ideal of $A$, $\tors (A/I)$ is a lattice quotient of $\tors A$ which is called an \emph{algebraic quotient}, and the corresponding lattice congruence is called an \emph{algebraic congruence}. The second part of this paper consists in studying algebraic congruences. We characterize the arrows of the Hasse quiver of $\tors A$ that are contracted by an algebraic congruence in terms of the brick labelling. In the third part, we study in detail the case of preprojective algebras $\Pi$, for which $\tors \Pi$ is the Weyl group endowed with the weak order. 
In particular, we give a new, more representation theoretical proof of the isomorphism between $\tors k Q$ and the Cambrian lattice when $Q$ is a Dynkin quiver. 
 We also prove that, in type $A$, the algebraic quotients of $\tors \Pi$ are exactly its Hasse-regular lattice quotients.
\end{abstract}
\maketitle

\setcounter{tocdepth}{2}
\tableofcontents

\section{Introduction}\label{intro}

The main object of study in this paper is the collection of torsion classes of a finite-dimensional algebra.  
Torsion classes are closely related to the study of derived categories and their $t$-structures.  
The recently developed $\tau$-tilting theory \cite{AIR,DIJ}, itself partly inspired by the cluster algebras of Fomin and Zelevinsky \cite{FZ1}, also provides insight into the structure of
torsion classes, but is generally forced to restrict attention to torsion classes which are functorially finite.  
By contrast, in this paper, we develop methods to understand the whole lattice of torsion classes.  
These methods also shed new light on certain lattices built from Weyl groups, such as the weak order and Cambrian lattices.

\subsection{Algebraic lattice congruences}
Let $A$ be a finite-dimensional algebra over an arbitrary field $k$ and let $\mod A$ be the category of finitely generated left $A$-modules. The main object of this paper is the complete lattice $\tors A$ of torsion classes of $\mod A$, ordered by inclusion. Recall that a \emph{torsion class} $\TT \subseteq \mod A$ is a full subcategory that is closed under extensions and factors.
Recall also that a complete lattice $L$ is a partially ordered set such that, for any subset $\S$, there is a unique largest element of $L$ smaller than all elements of $\S$, the \newword{meet} of $\S$, written $\Meet \S$, and a unique smallest element of $L$ larger than all elements of $\S$, the \newword{join} of $\S$, written $\Join \S$.

The starting point of this paper is the observation that each factor algebra of $A$ determines a \newword{complete lattice congruence} of $\tors A$
(an equivalence relation compatible with the complete lattice structure).
We describe the correspondence in terms of the lattices $\Ideals A$ and $\Com(\tors A)$. The lattice $\Ideals A$ is the set of two-sided ideals of $A$ ordered by inclusion. For a lattice $L$, the lattice $\Com L$ is the set of complete lattice congruences of $L$ ordered by refinement.

\begin{theorem}[Theorem \ref{eta op}]\label{eta op intro}
Let $A$ be a finite-dimensional $k$-algebra.
\begin{enumerate}[\rm(a)]
\item For any $I\in\Ideals A$, the map $\TT\mapsto\TT\cap\mod(A/I)$ is a surjective morphism of complete lattices from $\tors A$ to $\tors(A/I)$.
\end{enumerate}
Thus let $\Theta_I$ be the complete lattice congruence on $\tors A$ setting $\TT\equiv_{\Theta_I} \UU$ if and only if ${\TT \cap \mod (A/I)} = \UU \cap \mod (A/I)$.\begin{enumerate}[\rm(a)] \rs{1}
\item \label{eta op intro morph}
The map $\eta_A:\Ideals A\to\Com(\tors A)$ sending $I$ to $\Theta_I$ is a \emph{morphism of complete join-semilattices}: $\eta_A(\sum_{I \in \II} I) = \Join_{I \in \II} \eta_A(I)$ for any subset $\II \subseteq \Ideals A$.
\end{enumerate}
\end{theorem}

Theorem~\ref{eta op intro}\eqref{eta op intro morph} implies in particular that the map $\eta_A:I\mapsto\Theta_I$ is \emph{order-preserving}.
The map $\eta_A$ is typically not surjective.
We define an  \newword{algebraic congruence} of $\tors A$ to be a congruence of the form $\Theta_I$ for some $I\in\Ideals A$.
We write $\AlgCon A$ for the set of algebraic congruences of $\tors A$ (\emph{i.e.}, the image of $\eta_A$), partially ordered by refinement.
Similarly, an \newword{algebraic quotient} of $\tors A$ is the quotient of $\tors A$ modulo an algebraic congruence, so that $B \mapsto \tors B$ is a surjective map from factor algebras of $A$ to algebraic quotients of $\tors A$. Theorem \ref{eta op intro}(b) implies that $\AlgCon A$ is a complete lattice.

Recall that the \emph{Hasse quiver} $\Hasse P$ of a partially ordered set $P$ has vertex set $P$ and arrows $x\to y$ whenever $x>y$ and there is no $z$ such that $x>z>y$.

\begin{figure}
\begin{adjustbox}{rotate=90,center}\begin{minipage}{19.4cm}
\[
\begin{xy}
(0,0) *+{\tc{\ab}{\bc}{\cb}}="A",
(0,-20) *+{\tc{\ab}{\acb}{\cb}}="Ab",
(-12,-40) *+{\tc{\Sc}{\acb}{\cb}}="Aba",
(+12,-40) *+{\tc{\ab}{\acb}{\Sa}}="Abc",
(0,-60) *+{\tc{\Sc}{\acb}{\Sa}}="Abac",
(-40,-70) *+{\tb{\Sc}{\cb}}="Abab",
(+40,-70) *+{\tb{\ab}{\Sa}}="Abcb",
(0,-80) *+{\tb{\Sc}{\Sa}}="Abacb",
(-20,-100) *+{\ta{\Sc}}="Abacba",
(+20,-100) *+{\ta{\Sa}}="Abacbc",
(0,-120) *+{0}="Abacbac",
(-40,-30) *+{\tb{\bc}{\cb}}="Aa",
(+30,-20) *+{\tc{\ab}{\bc}{\Sb}}="Ac",
(+40,-40) *+{\tb{\ab}{\Sb}}="Acb",
(-15,-65) *+{\tb{\bc}{\Sb}}="Aac",
(0,-100) *+{\ta{\Sb}}="Aaca"
\ar"A";"Ab"|{\cirl{\Sb}}%
\ar"Ab";"Aba"|{\cirl{\ab}}%
\ar@{=>}"Ab";"Abc"|{\cirl{\cb}}%
\ar@{=>}"Aba";"Abac"|{\cirl{\cb}}%
\ar"Abc";"Abac"|{\cirl{\ab}}%
\ar@{=>}"Abac";"Abacb"|{\cirl{\acb}}%
\ar"Abacb";"Abacba"|{\cirl{\Sa}}%
\ar"Abacb";"Abacbc"|{\cirl{\Sc}}%
\ar"Abacba";"Abacbac"|{\cirl{\Sc}}%
\ar"Abacbc";"Abacbac"|{\cirl{\Sa}}%
\ar"Aba";"Abab"|{\cirl{\Sa}}%
\ar@{=>}"Abab";"Abacba"|{\cirl{\cb}}%
\ar"Abc";"Abcb"|{\cirl{\Sc}}%
\ar"Abcb";"Abacbc"|{\cirl{\ab}}%
\ar"A";"Aa"|{\cirl{\Sa}}%
\ar"Aa";"Abab"|{\cirl{\Sb}}%
\ar"A";"Ac"|{\cirl{\Sc}}%
\ar"Ac";"Acb"|{\cirl{\bc}}%
\ar"Acb";"Abcb"|{\cirl{\Sb}}%
\ar"Aa";"Aac"|(.45){\cirl{\Sc}}%
\ar@/_.8cm/"Ac";"Aac"|(.25){\cirl{\Sa}}%
\ar"Aac";"Aaca"|(.4){\cirl{\bc}}%
\ar"Aaca";"Abacbac"|{\cirl{\Sb}}%
\ar@/_.7cm/"Acb";"Aaca"|{\cirl{\Sa}}%
\end{xy}
\quad %
\begin{xy}
(0,0) *+{\tc{\ab}{\bc}{\Sc}}="A",
(0,-40) *+{\tc{\ab}{\Sa}{\Sc}}="Ab",
(-12,-70) *+{\tb{\Sa}{\Sc}}="Aba",
(-40,-90) *+{\ta{\Sc}}="Abab",
(+40,-70) *+{\tb{\ab}{\Sa}}="Abcb",
(+20,-100) *+{\ta{\Sa}}="Abacbc",
(0,-120) *+{0}="Abacbac",
(-40,-30) *+{\tb{\bc}{\Sc}}="Aa",
(+30,-20) *+{\tc{\ab}{\bc}{\Sb}}="Ac",
(+40,-40) *+{\tb{\ab}{\Sb}}="Acb",
(-20,-45) *+{\tb{\bc}{\Sb}}="Aac",
(-15,-90) *+{\ta{\Sb}}="Aaca"
\ar"A";"Ab"|{\cirl{\Sb}}%
\ar"Ab";"Aba"|{\cirl{\ab}}%
\ar"Aba";"Abacbc"|{\cirl{\Sc}}%
\ar"Abab";"Abacbac"|{\cirl{\Sc}}%
\ar"Abacbc";"Abacbac"|{\cirl{\Sa}}%
\ar"Aba";"Abab"|{\cirl{\Sa}}%
\ar"Ab";"Abcb"|{\cirl{\Sc}}%
\ar"Abcb";"Abacbc"|{\cirl{\ab}}%
\ar"A";"Aa"|{\cirl{\Sa}}%
\ar"Aa";"Abab"|{\cirl{\Sb}}%
\ar"A";"Ac"|{\cirl{\Sc}}%
\ar"Ac";"Acb"|{\cirl{\bc}}%
\ar"Acb";"Abcb"|{\cirl{\Sb}}%
\ar"Aa";"Aac"|(.45){\cirl{\Sc}}%
\ar@/_.8cm/"Ac";"Aac"|(.25){\cirl{\Sa}}%
\ar@/_.5cm/"Aac";"Aaca"|(.4){\cirl{\bc}}%
\ar"Aaca";"Abacbac"|{\cirl{\Sb}}%
\ar@/_.7cm/"Acb";"Aaca"|{\cirl{\Sa}}%
\end{xy}
\]
 \caption{Hasse quivers of the lattices $\tors \Lambda$ and $\tors \Lambda'$}\label{exlat}
 \end{minipage}
 \end{adjustbox}
\end{figure}

\begin{example} \label{basex}
 Consider the algebras
 \[\Lambda := \left. k \left(\boxinminipage{\xymatrix{1 \ar[r]^\alpha & 2 \ar@/^/[r]^\beta & 3 \ar@/^/[l]^{\beta^*}}}\right) \middle/ (\beta \alpha, \beta \beta^*, \beta^* \beta) \right. \quad \text{and} \quad \Lambda' := \Lambda / (\beta^*).\]

 We depict $\Hasse(\tors \Lambda)$ and $\Hasse(\tors \Lambda')$ in Figure~\ref{exlat}. It turns out that, in this case, each torsion class is of the form $\Fac T$ for some canonical module $T$, as explained in Section \ref{tautiltfiniteintro}, so we represent $\Fac T$ by the composition series of $T$. 
   In accordance with Theorem \ref{eta op intro}, $\tors \Lambda'$ is a
   lattice quotient of $\tors \Lambda$; the quotient map identifies torsion
   classes of $\Lambda$ connected by double arrows.
\end{example}

\subsection{Hasse quiver, brick labelling and forcing order}

We start our investigation by giving elementary lattice theoretical properties of $\tors A$. 
A complete lattice $L$ is called \emph{weakly atomic} if, whenever $x < y$ in $L$, $\Hasse[x, y]$ has at least one arrow. We recall in Section \ref{fbs} the definitions of  \emph{complete semidistributivity} and \emph{bialgebraicity}. We prove the following result.

\begin{theorem}[Theorem \ref{semidistributiveandbialg}] \label{semidistributiveandbialgintro}
 Let $A$ be a finite-dimensional algebra. The lattice $\tors A$ is bialgebraic, and therefore weakly atomic. Moreover, it is completely semidistributive. 
\end{theorem}
Note that the properties of $\tors A$ that are given in Theorem \ref{semidistributiveandbialgintro} are rare for complete lattices. 
They can be seen as a kind of discreteness of $\tors A$, even though it is usually infinite, and even uncountable. 

We now introduce a representation theoretical counterpart to the arrows of $\Hasse (\tors A)$. Recall that $S \in \mod A$ is called a \emph{brick} if any non-zero endomorphism of $S$ is invertible, \emph{i.e.} if $\End_A(S)$ is a division ring. It turns out that for each arrow $q : \TT \to \UU$ of $\Hasse(\tors A)$, there is a unique brick $S_q \in \TT$ satisfying $\Hom_A(U, S_q) = 0$ for any $U \in \UU$ (see Theorem~\ref{difftors2}). In order to relate lattice theory to representation theory, we label $q$ by $S_q$. Labels are written on arrows of Figure~\ref{exlat}. Notice that a brick usually labels more than one arrow.

For a complete lattice $L$, recall that $x \in L$ is \emph{completely join-irreducible} if it is non-zero and cannot be written non-trivially as the join of other elements. Equivalently, there is a unique arrow pointing from $x$ in $\Hasse L$.
A first lattice-theoretical interpretation of labels is that they naturally parametrize completely join-irreducible torsion classes.
\begin{theorem}[Theorem \ref{difftors2}(c)] \label{introdt}
 There is a bijection from completely join-irredu\-cible torsion classes $\TT$ to bricks of $A$ mapping $\TT$ to the label of the unique arrow pointing from $\TT$. 
There is a dual bijection from completely meet-irreducible torsion classes to isomorphism classes of bricks of $\mod A$.
\end{theorem}
Theorem \ref{introdt} generalizes a result of \cite{DIJ} about functorially finite torsion classes. It has been proven independently in \cite{BCZ}. 
Consider a complete lattice $L$.
A surjective complete lattice morphism $L \twoheadrightarrow L'$ determines a \emph{complete lattice congruence} $\Theta$ on $L$ (with congruence classes given by pre-images of elements of~$L'$). Given an arrow $q$ in $\Hasse L$, we say that $\Theta$ \emph{contracts} $q$ if the head and tail of $q$ are congruent modulo $\Theta$. If $L$ is finite, a lattice congruence on $L$ is completely determined by the set of arrows of $\Hasse L$ it contracts.  For  infinite $L$, this is not generally true.  However, consider the complete meet-sublattice $\Comb L \subseteq \Com L$ consisting of \emph{arrow-determined} complete congruences, \emph{i.e.} $\Theta \in \Com L$ such that $L/\Theta$ is weakly atomic (see Definition \ref{arrowdet} and Proposition \ref{arrdeteq}). An arrow-determined complete congruence is specified (among all such congruences) by the set of arrows it contracts.  

For two arrows $q$ and $q'$ of $\Hasse L$, we say that \emph{$q$ forces $q'$} and write $q \forces q'$ if any lattice congruence contracting $q$ also contracts $q'$. Clearly, $\forces$ is a preorder (a reflexive, transitive not-necessarily-antisymmetric relation) on $\Hasse_1(L)$. We call $\forces$ the \emph{forcing preorder}, and the corresponding equivalence relation is called \emph{forcing equivalence}. If $L$ is completely semidistributive and bialgebraic, \emph{e.g.} $L = \tors A$, then, for a subset $\S \subseteq \Hasse_1(L)$, there is an arrow-determined complete congruence contracting exactly $\S$ if and only if $\S$ is closed under forcing (Theorem \ref{mainiso}).

There is a natural map from the set of completely join-irreducible elements of $L$ to the forcing equivalence classes, mapping $x$ to the forcing equivalence class of the arrow pointing from~$x$. Dually, there is a natural map from the completely meet-irreducible elements of $L$ to the forcing equivalence classes. 
An important case for lattice theory occurs when these maps are actually bijective, in which case $L$ is called \emph{completely congruence uniform} (\emph{congruence uniform} if $L$ is finite).
A main theorem of this paper states that $\tors A$ is completely congruence uniform:
\begin{theorem}[Theorem~\ref{theorem uniform2}] \label{thintroconguni}
 Let $A$ be a finite-dimensional algebra.
 \begin{enumerate}[\rm(a)]
  \item Two arrows of $\Hasse(\tors A)$ are forcing equivalent if and only if they are labelled by isomorphic bricks. Hence there is a bijection between forcing equivalence classes of arrows of $\Hasse(\tors A)$ and isomorphism classes of bricks.
  \item The lattice $\tors A$ is completely congruence uniform.
 \end{enumerate}
\end{theorem}
In particular, by Theorem~\ref{thintroconguni}(a), the forcing preorder induces an order on the set $\brick A$ of bricks that we also denote by $\forces$ and call the \newword{forcing order}.

The labelling of $\Hasse(\tors A)$ by bricks sheds additional light on Theorem~\ref{eta op intro}.
Given $I\in\Ideals A$, recall from Theorem~\ref{eta op intro} that $\Theta_I$ is the complete lattice congruence on $\tors A$ corresponding to $\tors A \twoheadrightarrow \tors (A/I)$.
As $\tors (A/I)$ is weakly atomic, $\Theta_I$ is arrow-determined, so we can characterize $\Theta_I$ by the set of arrows it contracts. They are specified in the following theorem.
\begin{theorem}[Theorem~\ref{annihilate contract0}] \label{introanncont}
 Let $A$ be a finite-dimensional $k$-algebra and $I \in \Ideals A$. An arrow $q$ of $\Hasse(\tors A)$ is not contracted by $\Theta_I$ if and only if $S_q$ is annihilated by $I$, that is $S_q \in \mod (A/I)$. Moreover, the labelling of arrows that are not contracted by $\Theta_I$ is the same in $\Hasse(\tors A)$ and $\Hasse(\tors (A/I))$.

More generally, if $\UU \subseteq \TT$ are in $\tors A$, then $\TT \equiv_{\Theta_I} \UU$ if and only if there is no arrow of $\Hasse [\UU, \TT]$ whose label is in $\mod (A/I)$.  \end{theorem}

\begin{example}
 Theorem~\ref{introanncont} is illustrated in Figure~\ref{exlat} for algebras of Example~\ref{basex}. Indeed, the bricks of $\mod \Lambda$ that are not annihilated by $I = (\beta^*)$ are $\cb$ and $\acb$.
\end{example}

Given a brick $S$ of $A$, let $\ann S$ be the annihilator $\{a \in A \mid aS = 0\}$ which is a two-sided ideal of $A$.
As a corollary of Theorem~\ref{introanncont}, we get

\begin{corollary}[Corollary~\ref{Istar}] \label{Istar intro}
Consider a finite-dimensional $k$-algebra $A$ and write $I_0$ for $\bigcap_{S \in \brick A} \ann S$.
Then $\tors A$ and $\tors (A/I_0)$ are canonically isomorphic. Moreover, $I_0$ is the biggest ideal of $A$ with this property.
\end{corollary}

\subsection{Functorially finite torsion classes and \texorpdfstring{$\tau$-tilting}{tau-tilting} theory} \label{tautiltfiniteintro}
An important tool to study $\tors A$ consists of basic support $\tau$-tilting $A$-modules introduced by Adachi--Iyama--Reiten \cite{AIR}.
A module $T \in \mod A$ is \emph{$\tau$-rigid} if $\Hom_A(T, \tau T) = 0$ where $\tau$ is the Auslander-Reiten translation.
It is called \emph{$\tau$-tilting} if it is $\tau$-rigid and has $n$ non-isomorphic indecomposable summands where $n$ is the number of simple $A$-modules; in fact, this is equivalent to the natural maximality condition for $\tau$-rigid modules.
Finally, we say that $T$ is \emph{support $\tau$-tilting} if it is a $\tau$-tilting $(A/(e))$-module for some idempotent $e \in A$.

The set $\ftors A$ of \emph{functorially finite} torsion classes of $A$ is a subposet of $\tors A$. It is proven in \cite{AIR} that there is a bijection from the set $\sttilt A$ of isomorphism classes of basic support $\tau$-tilting $A$-modules to $\ftors A$.
The bijection sends $T \in \sttilt A$ to the category $\Fac T$ consisting of modules obtained as quotients of $T^\ell$ for any $\ell \in \integers_{\ge 0}$. It endows $\sttilt A$ with the structure of a partially ordered set.

By \cite[Theorem 1.3]{DIJ}, $\Hasse(\sttilt A) \cong \Hasse(\ftors A)$ is a full subquiver of $\Hasse(\tors A)$, and, by \cite{AIR}, arrows of $\Hasse(\sttilt A)$ are of the form $T \oplus X \to T \oplus X^*$ where $X$ is indecomposable, $X^*$ is indecomposable or zero and there is an exact sequence $X \xto{u} T' \to X^* \to 0$ where $u$ is a left minimal $(\add T)$-approximation.
The process of moving forwards or backwards along an arrow of $\Hasse(\sttilt A)$ is called a \emph{mutation}.
For such an arrow $T \oplus X \to T \oplus X^*$, the label of $q: \Fac (T \oplus X) \to \Fac(T \oplus X^*)$ is
 \[S_q \cong \frac{X}{\sum_{f \in \rad_A(T \oplus X, X)} \Image f}\] (see Proposition~\ref{theorem label2}).

Recall that by \cite[Theorem 2.7]{AIR} and \cite[Theorem 1.2]{DIJ}, $\# \tors A < \infty$ if and only if $\ftors A = \tors A$ if and only if $\# \ftors A = \# \sttilt A < \infty$. In this case, $A$ is called \emph{$\tau$-tilting finite}.
Hence, we get in Theorem~\ref{eta op} the following corollary of Theorem~\ref{eta op intro}(a).

\begin{corollary}\label{factor-closed}
The class of $\tau$-tilting finite algebras is closed under taking factor algebras.
\end{corollary}

For example, local algebras and representation-finite algebras are clearly $\tau$-tilting finite.
We refer to \cite{AAC,EJR,IZ,M} for more examples.

We suppose now that $A$ is $\tau$-tilting finite. An important ingredient for understanding $\Con(\tors A) = \Com(\tors A)$ in this case is that $\tors A$ has a property called \emph{polygonality} (Proposition~\ref{torsApole}), and therefore the forcing preorder can be easily described combinatorially (Proposition~\ref{forcing=polygonal}).
Using this ingredient, we give two algebraic characterizations of the forcing order on bricks.
In order to do so, we define a \emph{semibrick} as a set of bricks having no non-zero morphisms between distinct elements.
For a semibrick $E$, we define $\Filt E$ as the smallest full subcategory of $\mod A$ containing $E$ and closed under extensions.
Then the subcategory $\Filt E$ is wide, \emph{i.e.}, closed under extension, kernels and cokernels, \cite[Theorem 1.2]{ringelk}.
\begin{theorem}[Theorem~\ref{theorem forcing}]
 Let $A$ be a finite-dimensional algebra that is $\tau$-tilting finite. The forcing order $\forces$ on $\brick A$ is the transitive closure of the relation $\forces[f]$ defined by:
$S_1 \forces[f] S_2$ if there is a semibrick $\{S_1\} \cup E$ such that $S_2 \in \Filt(\{S_1\} \cup E) \setminus \Filt(E)$.

  The relation $\forces$ can also be defined as the transitive closure of the relation $\forces[pf]$ defined by:
$S_1 \forces[pf] S_2$ if there is a semibrick $\{S_1, S_1'\}$ such that $S_2 \in \Filt(\{S_1, S_1'\}) \setminus \{S'_1\}$.
\end{theorem}

\begin{example} Theorem \ref{thintroconguni}(a) implies that the set of arrows contracted in passing from the left hand side to the right hand side of Figure \ref{exlat} necessarily consists of all arrows labelled by some set of bricks, since each arrow labelled by a given brick forces all the other arrows labelled by that brick.  Further, this set must be closed under the relation $\forces$.  This is consistent with Figure \ref{exlat}, since $\smash{\cb}$ forces only $\smash{\acb}$, and $\smash{\acb}$ forces nothing.
\end{example}

We give another characterization of the forcing order on bricks under additional hypotheses on $A$.
The following theorem applies in particular to finite-dimensional hereditary algebras and preprojective algebras of Dynkin type.

\begin{theorem}[Theorem~\ref{forcing=doubleton}] \label{fordoub}
 Let $A$ be a finite-dimensional $k$-algebra that is $\tau$-tilting finite satisfying $\End_A(S) \cong k$ and $\Ext^1_A(S, S) = 0$ for all $S \in \brick A$. Then the forcing order $\forces$ on $\brick A$ is the transitive closure of the relation $\forces[d]$ defined by:
 $S_1 \forces[d] S_2$ if there exists a brick $S_1'$ such that
   \begin{itemize}\setlength{\itemindent}{10pt}
    \item[] $\dim \Ext^1_A(S_1, S_1') = 1$ and there is an exact sequence $0 \to S_1' \to S_2 \to S_1 \to 0$
    \item[or] $\dim \Ext^1_A(S_1', S_1) = 1$ and there is an exact sequence $0 \to S_1 \to S_2 \to S_1' \to 0$.
   \end{itemize}
\end{theorem}
The transitive closure of $\forces[d]$ was introduced in \cite{IRRT} as the \emph{doubleton extension order} and Theorem~\ref{fordoub} was proven for preprojective algebras of Dynkin type.

\subsection{Applications to preprojective algebras and Weyl groups}

Our remaining results concern the special case of a preprojective algebra $\Pi$ of Dynkin type
(see Section \ref{preproj-weyl} for background).
As mentioned above, $\Pi$ is $\tau$-tilting finite.
By a result of Mizuno (see Theorem \ref{Mizuno's theorem}), $\tors\Pi$ is isomorphic to the corresponding Weyl group $W$ endowed with the weak order.
The next goal is to characterize algebraic lattice congruences of $W$.

A partially ordered set $P$ is called \newword{Hasse-regular} if it has the property that each vertex of the Hasse quiver has the same degree (as an undirected graph). If $A$ is $\tau$-tilting finite, then $\tors A$ is necessarily Hasse-regular (see Corollary~\ref{Hasse-regular}).

As before, a join-irreducible element of a finite lattice $L$ is an element $j$ with exactly one arrow from~$j$ in $\Hasse L$.
 We say a lattice congruence on $L$ \emph{contracts} $j$ if it contracts the unique arrow from $j$.
A join-irreducible element $j\in L$ is called a \newword{double join-irreducible element} if the unique arrow from $j$ in $\Hasse L$ goes either to another join-irreducible element or to the bottom element of $L$.    
\begin{theorem} \label{double ji}
Let $W$ be a finite Weyl group of simply-laced type, and $\Pi$ the corresponding preprojective algebra.
Let $\Theta$ be a lattice congruence on $W \cong \tors \Pi$.
Consider the following three conditions on $\Theta$:
\begin{enumerate}[\rm(i)]
\item $\Theta$ is an algebraic congruence on $W$.
\item $W/\Theta$ is Hasse-regular.
\item There is a set $J$ of double join-irreducible elements such that $\Theta$ is the smallest congruence contracting every element of $J$.
\end{enumerate}
Then \textup{(i)} $\Rightarrow$ \textup{(ii)} $\Rightarrow$ \textup{(iii)}.
If $W$ is of type $A_n$, then all three conditions are equivalent.
\end{theorem}

It would be interesting to understand the algebraic quotients of $W$ for any Dynkin type. Unfortunately, (iii) $\Rightarrow$ (ii) and (iii) $\Rightarrow$ (i) are not true in type $D$ as shown in Example \ref{cexd4}.

The equivalence of (ii) and (iii) in type $A_n$ was also proved independently in \cite[Theorem~26]{HM}, which also characterizes double join-irreducible elements in terms of the noncrossing arc diagrams introduced in~\cite{arcs}.


In the following example, we show that in full generality algebraic congruences do not depend only on the lattice structure of $\tors A$:
\begin{example} \label{cexiiii}
 We consider the $k$-algebras
 \[\Lambda_1 := k \left(\xymatrix{ 1 \ar@(ul,dl)[]_u \ar[r]^x & 2 } \right) / (u^2) \quad \text{and} \quad \Lambda_2 := k \left(\xymatrix{ 1  \ar[r]^x & 2 \ar@(dr,ur)[]_v} \right) / (v^2).\]
 We also consider the $\reals$-algebra $\Lambda_3$ of type $B_2$, constructed as the tensor algebra of the species
 $\smash{\xymatrix{ \reals \ar@<-1pt>[r]^{\complex} & \complex }}$.
 The labelled Hasse quivers of their support $\tau$-tilting modules are depicted in Figure \ref{cexf}. It is an easy application of Theorem \ref{introanncont}, that an algebraic congruence on $\Lambda_1$ that contracts $q_1$ has to contract $q_2$ while the converse is not true. In the same way, an algebraic congruence on $\Lambda_2$ that contracts $q_2$ has to contract $q_1$ while the converse is not true. Finally, for $\Lambda_3$, an algebraic congruence contracts $q_1$ if and only if it contracts $q_2$. So the algebraic congruences of these three isomorphic lattices are not the same. Moreover, it also shows that (iii) in Theorem \ref{double ji} does not imply (i) in general and in fact that no such
combinatorial criterion can be equivalent to (i) in full generality.
 \begin{figure}
\makebox[\textwidth]{%
 $\boxinminipage{\xymatrix@C=1.5cm@R=2.5cm@!=0cm{
  & \tb{\Suudd}{\Sb} \ar[ddl]|-{\cirl{\Sa}} \ar[dr]|-{\cirl{\Sb}} \\
  & & \tb{\Suudd}{\Suud} \ar[d]|-{\cirl{\Sud}}^-{\quad q_1} \\
  \ta{\Sb} \ar[ddr]|-{\cirl{\Sb}} & & \tb{\Suu}{\Suud} \ar[d]|-{\cirl{\Suud}}^-{\quad\; q_2} \\
  & & \ta{\Suu} \ar[dl]|-{\cirl{\Sa}} \\
 & 0
 }} \quad \quad \boxinminipage{\xymatrix@C=1.5cm@R=2.5cm@!=0cm{
  & \tb{\Sudd}{\Sdd} \ar[ddl]|-{\cirl{\Sa}} \ar[dr]|-{\cirl{\Sb}} \\
  & & \tb{\Sudd}{\Sudud} \ar[d]|-{\cirl{\Sudd}}^-{\quad\; q_1} \\
  \ta{\Sdd} \ar[ddr]|-{\cirl{\Sb}} & & \tb{\Sa}{\Sudud} \ar[d]|-{\cirl{\Sud}}^-{\quad q_2} \\
  & & \ta{\Sa} \ar[dl]|-{\cirl{\Sa}} \\
 & 0
 }}
 \quad \quad \boxinminipage{\xymatrix@C=1.5cm@R=2.5cm@!=0cm{
  & \tb{\spRC}{\spC} \ar[ddl]|-{\cirl{\spR}} \ar[dr]|-{\cirl{\spC}} \\
  & & \tb{\spRC}{\spRRC} \ar[d]|-{\cirl{\spRC}}^-{\quad q_1} \\
  \ta{\spC} \ar[ddr]|-{\cirl{\spC}} & & \tb{\spR}{\spRRC} \ar[d]|-{\cirl{\spRRC}}^-{\quad\; q_2} \\
  & & \ta{\spR} \ar[dl]|-{\cirl{\spR}} \\
 & 0
 }}
 $
}
 \caption{$\Hasse(\sttilt \Lambda_1)$, $\Hasse(\sttilt \Lambda_2)$ and $\Hasse(\sttilt \Lambda_3)$} \label{cexf}
 \end{figure}
\end{example}

We now give a more explicit description of algebraic lattice quotients of $W$ in type~$A$.
Write $\Pi_{A_n}$ for the preprojective algebra of type $A_n$, and $W_{A_n}$ for the corresponding Weyl group, isomorphic to the symmetric group $\sym_{n+1}$.

We denote by $\UUU$ the set of the following objects, which we can naturally identify:
\begin{itemize}
\item Double join irreducible elements in $W_{A_n}$.
\item \emph{Non-revisiting paths}.   (These are paths in the quiver of $\Pi_{A_n}$ which visit each vertex at most once, including the trivial paths $e_i$.)
\item \emph{Uniserial} $\Pi_{A_n}$-modules. (These are $\Pi_{A_n}$-modules which have unique composition series.)
\end{itemize}
Then $\UUU$ forms a partially ordered set, setting $w \le w'$ if $w$ is a subpath of $w'$.
We denote by $\Ideals \UUU$ the set of order ideals of $\UUU$, which consists of subsets $\S \subset \UUU$ such that if $w \in \S$ and $w \le w'$ then $w' \in \S$.

\begin{theorem}\label{AlgQuot A}
 Let us consider the two-sided ideal $I_\cyc$ of $\Pi_{A_n}$ generated by all 2-cycles and $\overline{\Pi}_{A_n} := \Pi_{A_n} / I_\cyc$.  Then, writing $\eta$ for $\eta_{\overline{\Pi}_{A_n}}$,  the following hold.
 \begin{enumerate}[\rm(a)]
  \item The ideal $I_0$ defined in Corollary~\ref{Istar intro} coincides with $I_\cyc$.
  \item We have lattice isomorphisms
\begin{equation*}
\Ideals \UUU \xrightarrow{\sim}
\Ideals\overline{\Pi}_{A_n}
\xrightarrow{\sim}\AlgCon \Pi_{A_n}
\end{equation*}
given by $\S\mapsto \spanv_k \S$ and $I\mapsto\eta(I)$.
 \item If $I, J \in \Ideals \Pi_{A_n}$, we have
  \[\eta(I) = \eta(J) \Leftrightarrow I + I_\cyc = J + I_\cyc \Leftrightarrow I \cap \UUU = J \cap \UUU.\]
\end{enumerate}
\end{theorem}

Based on Theorem~\ref{AlgQuot A} and some general combinatorial results found in~\cite{regions2}, we give an explicit combinatorial description of arbitrary algebraic congruences and quotients in type $A$.
(See Theorems~\ref{type A contract} and~\ref{type A bottoms}.)

To conclude the paper, we apply our theory to preprojective algebras to obtain a new representation-theoretical
approach to some results about \emph{Cambrian lattices}.
We consider a preprojective algebra $\Pi$ of Dynkin type and the corresponding Weyl group $W$ endowed with the weak order.
We continue to identify the lattice $W$ with the lattice $\tors\Pi$ via Mizuno's isomorphism as mentioned above.
To each Coxeter element $c$, or equivalently to each orientation $Q_c$ of the Dynkin diagram, corresponds the so-called \emph{Cambrian congruence} $\Theta_c$ on $W$ (see Section~\ref{camblat}).
On the other hand, we can consider the natural surjective lattice morphism $W \cong \tors \Pi \twoheadrightarrow \tors kQ_c$.
Our first result about Cambrian lattices is the following one:

 \begin{theorem}[Theorem~\ref{kQ camb}]\label{kQ camb intro}
The Cambrian congruence $\Theta_c$ induces the surjective lattice morphism $\tors \Pi \twoheadrightarrow \tors kQ_c$.
 In particular, $\tors k Q_c$ is identified with the \emph{Cambrian lattice} $W/\Theta_c$.
 \end{theorem}

The identification of $\tors k Q_c$ with $W/\Theta_c$ in Theorem~\ref{kQ camb intro} was proved in \cite{IngTho} using combinatorial methods.
Our proof uses mostly representation theory, bypassing in particular the \emph{sortable elements} \cite{sort_camb} used in \cite{IngTho}.
We also give a new representation-theoretical argument for the following result, proven using sortable elements in~\cite{sort_camb}.
 \begin{theorem}[Theorem~\ref{camb sublat}] \label{cambsublatintro}
  The subset $\smash{\pidown^{\Theta_c} W}$ of $W$ consisting of smallest elements of each $\Theta_c$-equivalence class is a sublattice of $W$, canonically isomorphic to the Cambrian lattice $W/\Theta_c$.
 \end{theorem}
 As explained in the next section, it is a general result that $\smash{\pidown^{\Theta_c} W}$ is closed under joins. The strong part of Theorem~\ref{cambsublatintro} is that it is also closed under meets.

\section{Lattice congruences and forcing order}

\subsection{Preliminaries} \label{lat prelim}
We give some background material on lattices.
Much of this is in standard lattice-theory books such as \cite{Birkhoff,LTF}.
Some of the material given here follows an order-theoretic approach to lattice congruences described in \cite[Section~9-5]{regions}.

Let $L$ be a partially ordered set and let $x$ and $y$ be elements of $L$.
An element $z$ of $L$ is called the \newword{join} of $x$ and $y$ and denoted $x\join y$ if $z\ge x$ and $z\ge y$ and if, for every element $w$ with $w\ge x$ and $w\ge y$, we have $w\ge z$.
Thus the join of $x$ and $y$, if it exists, is the unique minimal common upper bound of $x$ and $y$.
Dually, the \newword{meet} $x\meet y$ of $x$ and $y$, if it exists, is the unique maximal common lower bound of $x$ and $y$.

A \newword{lattice} is a partially ordered set $L$ with the property that for every $x,y\in L$ the join $x\join y$ and the meet $x\meet y$ both exist.
The join $\join$ is an associative, commutative operation on $L$, and for any nonempty finite subset $\S=\set{x_1,\ldots,x_k}$ of $L$, the element $x_1\join\cdots\join x_k$ is the unique minimal common upper bound for the elements of $\S$.
We write $\Join \S$ for this minimal upper bound.
Similarly, $\meet$ is associative and we write $\Meet \S$ for $x_1\meet\cdots\meet x_k$, which is the unique maximal lower bound for $\S$.

If $\S$ is an infinite subset of $L$, then there need not exist a unique minimal upper bound for $\S$ in $L$, even when $L$ is a lattice.
(For example, consider the integers $\integers$ under their usual order and take $\S=\integers$.)
Similarly, $\S$ need not have a unique maximal lower bound.
A lattice $L$ is called \newword{complete} if every subset $\S$ of $L$ admits a unique minimal upper bound $\Join \S$ and a unique maximal lower bound $\Meet \S$. In this case $L$ has a minimum $0 := \Join \emptyset = \Meet L$ and a maximum $1 := \Meet \emptyset = \Join L$.

Recall that the \newword{Hasse quiver} $\Hasse L$ of an ordered set $L$ has set of vertices $L$ and an arrow $x \to y$ if and only if $x > y$ and for any $z \in L$, $x \ge z \ge y \Rightarrow x = z \text{ or } z = y$.
If $x\to y$ is an arrow in $\Hasse L$, then we say that $x$ \emph{covers} $y$.

We say that $j\in L$ is \newword{join-irreducible} if there does not exist a finite subset $\S\subseteq L$ such that $j=\Join \S$ and $j\notin \S$. We say that it is \newword{completely join-irreducible} if there does not exist a subset $\S \subseteq L$ such that $j = \Join \S$ and $j \notin \S$.
An element $j$ is completely join-irreducible if and only if there exists an element $j_*$ satisfying $\{x \in L \mid x < j\} = \{x \in L \mid x \leq j_*\}$. In particular, if $j$ is completely join-irreducible, then it covers exactly one element, $j_*$. If $L$ is finite, then the converse is true: if $j$ covers exactly one element, then it is join-irreducible.
In the same way, $m \in L$ is \newword{meet-irreducible} if every finite $\S\subseteq L$ with $m=\Meet \S$ has $m\in \S$. It is \newword{completely meet-irreducible} if every $\S \subseteq L$ with $m=\Meet \S$ has $m\in \S$.
If $m$ is completely meet-irreducible, then $m$ is covered by exactly one element $m^*$.
The converse is true if $L$ is finite.
We denote by $\jirr L$ ($\cjirr L$) and $\mirr L$ ($\cmirr L$) the sets of (completely) join-irreducible and (completely) meet-irreducible elements in $L$ respectively.

A map $\eta$ from a lattice $L_1$ to another lattice $L_2$ is called a \newword{morphism of lattices}
if $\eta(x\join y)=\eta(x)\join \eta(y)$ and $\eta(x\meet  y)=\eta(x)\meet \eta(y)$ for every $x,y\in L_1$.
If $\eta: L_1\to L_2$ is a morphism of lattices, then $\eta(\Join \S)=\Join\eta(\S)$ and $\eta(\Meet \S)=\Meet\eta(\S)$ for any \emph{finite} subset $\S$ of $L_1$.
However, the same property need not hold for \emph{infinite} subsets of $L_1$, even when $L_1$ and $L_2$ are both complete lattices.
A map $\eta$ from a complete lattice $L_1$ to a complete lattice $L_2$ is a \newword{morphism of complete lattices} if $\eta(\Join \S)=\Join\eta(\S)$ and $\eta(\Meet \S)=\Meet\eta(\S)$ for every subset $\S$ of $L_1$.
It is more typical in the lattice theory literature to say ``lattice homomorphism'' for a morphism of lattices and ``complete lattice homomorphism" for ``morphism of complete lattices'', but we adopt the more category-theoretical language here.

A \newword{join-semilattice} is a partially ordered set with a join operation, and a \newword{meet-semilattice} is a partially ordered set with a meet operation.
A map $\eta:L_1\to L_2$ is a \newword{morphism of join-semilattices} if $\eta(x\join y)=\eta(x)\join \eta(y)$ for every $x,y\in L_1$.
It is a \newword{morphism of meet-semilattices} if $\eta(x\meet  y)=\eta(x)\meet \eta(y)$ for every $x,y\in L_1$.
A join-semilattice or meet-semilattice can be complete or not in the sense of the previous paragraph, and we can similarly define a \newword{morphism of complete join-semilattices} or a \newword{morphism of complete meet-semilattices}.

A \newword{join-sublattice} (respectively, \newword{meet-sublattice}) of a join-semilattice (respectively, meet-semilattice) is a subset that is closed under the join (respectively, meet) operation, and a \newword{sublattice} of a lattice is a subset that is a join-sublattice and a meet-sublattice.
The image of a morphism $\eta:L_1\to L_2$ of lattices (respectively, join-semilattices, meet-semilattices) is a sublattice (respectively, join-sublattice, meet-sublattice) of $L_2$.

We recall the following general definition, and give some properties in the special case of complete lattices. 

\begin{definition} \label{defadjlat} 
Let $P$ and $Q$ be posets and let $a:P\to Q$ and $b:Q\to P$ be order-preserving maps.
We say that $(a,b)$ is an \newword{adjoint pair} if $p\in P$ and $q\in Q$
satisfy $a(p)\le q$ if and only if they satisfy $p\le b(q)$.
\end{definition}

\begin{proposition}\label{adjoint}
Assume that $(a,b)$ is an adjoint pair, and both $P$ and $Q$ are complete lattices. The following hold:
\begin{enumerate}[\rm(a)]
\item The map $a$ is a morphism of complete join-semilattices, and
$b$ is a morphism of complete meet-semilattices.
\item For any $p\in P$ and $q\in Q$, we have $p\le ba(p)$ and $ab(q)\le q$.
\end{enumerate}
\end{proposition}

\begin{proof}
(a) We show the assertion for $a$; the assertion for $b$ is dual.
Take any subset $\S \subseteq P$.
To prove $\Join  a(\S)=a(\Join  \S)$, it is enough
to show that $q\in Q$ satisfies $a(p)\le q$ for all $p\in \S$ if and only if
$a(\Join  \S)\le q$.
The condition $a(p)\le q$ for all $p\in \S$ is equivalent to
$p \le b(q)$ for all $p\in \S$.
This is equivalent to $\Join  \S\le b(q)$, which
is equivalent to $a(\Join  \S)\le q$.
Thus the assertion follows.

(b) Since $a(p)\le a(p)$, we have $p\le ba(p)$. Similarly we have $ab(q)\le q$.
\end{proof}

An equivalence relation $\equiv$ on a lattice $L$ is called a \newword{(lattice) congruence} if and only if it has the following property:
If $x_1$, $x_2$, $y_1$, and $y_2$ are elements of $L$ such that $x_1\equiv y_1$ and $x_2\equiv y_2$, then also $x_1\meet x_2\equiv y_1\meet y_2$ and $x_1\join x_2\equiv y_1\join y_2$.
Given a lattice congruence $\Theta$ on $L$, the \newword{quotient lattice} is $L/\Theta$, where the partial order is defined as follows: A $\Theta$-class $C_1$ is less than or equal to a $\Theta$-class $C_2$ in $L/\Theta$ if and only if there exists an element $x_1$ of $C_1$ and an element $x_2$ of $C_2$ such that $x_1\le x_2$ in $L$.
Equivalently, for any $x_1\in C_1$ and $x_2\in C_2$, the join of $C_1$ and $C_2$ is the congruence class containing $x_1\join x_2$ and the meet of $C_1$ and $C_2$ is the congruence class containing $x_1\meet x_2$.
We denote by $\Con L$ the set of congruences on $L$, partially ordered with the refinement order ($\Theta \le \Theta'$ if and only if $\forall x, x' \in L, x \equiv_\Theta y \Rightarrow x \equiv_{\Theta'} y$).
This is a lattice, and in fact it is a sublattice of the lattice of all equivalence relations (or equivalently the lattice of all set partitions of $L$).
It is also a distributive lattice, meaning that meet distributes over join and vice versa.

Similarly, we define a \newword{complete (lattice) congruence} on a complete lattice $L$ to be an equivalence relation $\equiv$ with the following property:
For any indexing set $I$ (not necessarily finite) and families $\set{x_i}_{i\in I}$ and $\set{y_i}_{i\in I}$ of elements of $L$, if $x_i\equiv y_i$ for all $i\in I$, then $\Meet\set{x_i \mid i\in I}\equiv\Meet\set{y_i \mid i\in I}$ and $\Join\set{x_i \mid i\in I}\equiv\Join\set{y_i \mid i\in I}$. We denote by $\Com L$ the lattice of complete congruences on $L$.

For a (complete) lattice congruence $\Theta$ on $L$, the map sending each element of $L$ to its congruence class is a morphism of (complete) lattices from $L$ to $L/\Theta$.
On the other hand, given a morphism of (complete) lattices $\eta:L_1\to L_2$, there is a (complete) lattice congruence $\Theta_\eta$ on $L_1$ defined by $x\equiv_{\Theta_\eta} y$
 if and only if $\eta(x)=\eta(y)$.
Moreover if $\eta$ is surjective, then $\eta$ induces an isomorphism $L_1/\Theta_\eta\to L_2$ of (complete) lattices.

Complete lattice congruences on a complete lattice $L$ are particularly well-behaved. For the remainder of this subsection, we assume that $L$ is complete and $\Theta$ is a complete congruence.
In particular, each congruence class is an interval in $L$.
(Given $x,y\in L$ with $x\le y$, the \newword{interval} $[x,y]$ in $L$ is the set $\set{z\in L \mid x\le z\le y}$.)
In particular, given an element $x$ of $L$, the congruence class of $x$ has a unique minimal element $\pidown x = \smash{\pidown^\Theta} x$ and a unique maximal element $\piup x = \smash{\piup_\Theta} x$.
The maps $\pidown$ and $\piup$ are order-preserving.
(See \cite[Proposition~9-5.2]{regions} and \cite[Exercise~9.42]{regions}.)
The finite case of the following proposition is \cite[Proposition~9-5.5]{regions}, and this version for complete lattices holds by essentially the same proof.
(See \cite[Exercise~9.42]{regions}.)

\begin{proposition} \label{realpidown}
Suppose $L$ is a complete lattice and $\Theta$ is a complete congruence.
The sets $\pidown L:=\set{\pidown x \mid x\in L}$ and $\piup L:=\set{\piup x \mid x\in L}$, endowed with the partial orders induced from the partial order on $L$, are complete lattices, both isomorphic to the quotient lattice $L/\Theta$.
\end{proposition}

The maps $\pidown:L\to \pidown L$ and $\piup:L\to \piup L$ are morphisms of complete lattices.
However, $\pidown L$ and $\piup L$ are not necessarily sublattices of $L$.
In general, $\pidown L$ is only a complete join-sublattice of $L$, and $\piup L$ is only a complete meet-sublattice.  
We will see in Section~\ref{camblat} that in an important example (when $\Theta$ is a Cambrian congruence), the subposets $\pidown L$ and $\piup L$ are sublattices of $L$.

For a complete lattice congruence $\Theta$ on $L$, consider the canonical projection $\eta:L\to L/\Theta$.
For $x,y\in L$, $\eta(x)\le\eta(y)$ if and only if $\pidown x\le y$ if and only if $x\le\piup y$.
In particular, the image of an interval in $L$ under $\eta$ is an interval in $L/\Theta$.
Specifically, for any $x\le y$, we have
$\eta([x,y])=[\eta(x),\eta(y)]$.

Given a complete congruence $\Theta$ on a complete lattice $L$, say that $\Theta$ \newword{contracts} an arrow $y \to x$ in $\Hasse L$ if $x\equiv_\Theta y$. As one might expect, arrows cannot be contracted independently. Rather, if an arrow is contracted by $\Theta$, then $\Theta$ is forced to contract other arrows as well.
To formalize this forcing, for each subset $\S \subseteq \Hasse_1(L)$, we define $\con(\S)$ to be the minimum lattice congruence that contracts all arrows in $\S$.
(We take $\con(\S)$ to be the meet of all congruences that contract all arrows in~$\S$.
Since $\Con L$ is a complete sublattice of the lattice of equivalence relations on $L$, this meet also contracts all arrows in~$\S$.)
We define the \emph{forcing equivalence} by
\[a\forceseq b\Longleftrightarrow \con(a)=\con(b).\]
We define the \emph{forcing preorder} on $\Hasse_1(L)$ by
\[a \forces b\Longleftrightarrow\con(a)\ge\con(b)\text{ in $\Con L$}.\]
This gives a partial order on the set $\Hasse_1(L)/\mathord{\forceseq}$ of forcing equivalence classes.

For a special class of finite lattices called polygonal lattices, the forcing preorder on arrows has a simple, local description, as we now explain.
A \emph{polygon} in a finite lattice $L$ is an interval $[x,y]$ such that $\{z\in L\mid x < z < y\}$
consists of two disjoint nonempty chains.
The lattice $L$ is \emph{polygonal} if the following two conditions hold:
First, if there are arrows $y_1\to x$ and $y_2\to x$ in the Hasse quiver for distinct elements $y_1$ and $y_2$, then $[x,y_1\join  y_2]$ is a polygon; and second, if there are arrows $y\to x_1$ and $y\to x_2$ in the Hasse quiver for distinct elements $x_1$ and $x_2$, then $[x_1\meet  x_2,y]$ is a polygon.

We define the \emph{polygonal preorder} $\forces[p]$ on arrows of $\Hasse L$.
In every polygon labelled as shown here,
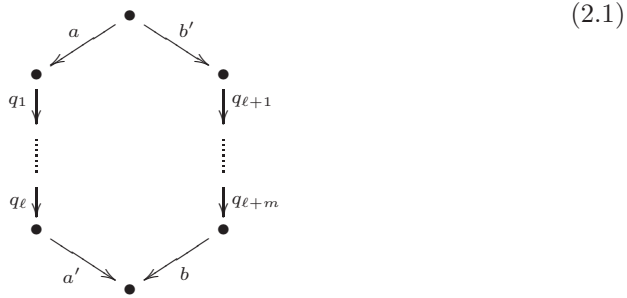
\begin{equation}\label{a polygon}
\xymatrix@R=1.2em{
&\bullet\ar[rd]^{b'} \ar[ld]_{a}\\
\bullet\ar[d]_{q_1}& &\bullet\ar[d]^{q_{\ell+1}}\\
 \ar@{..}[d] & & \ar@{..}[d] \\
\ar[d]_{q_\ell}&&\ar[d]^{q_{\ell+m}}\\
\bullet\ar[dr]_{a'} &&\bullet\ar[dl]^{b}\\
&\bullet
}\end{equation}
we have $a \forces[p] b \forces[p] a$ and $a \forces[p] q_i$ for all $i$.
We take the transitive closure to obtain the polygonal preorder.
The \emph{polygonal equivalence} is defined by\[a\forceseq_{\rm p}b\Longleftrightarrow a\forces[p]b\forces[p]a.\]
Clearly the polygonal preorder gives a partial order on $\Hasse_1(L)/\mathord{\forceseq_{\rm p}}$ which we call the \emph{polygonal order}.
We have the following general result.

\begin{proposition}\cite[Theorem~9-6.5]{regions}\label{forcing=polygonal}
Let $L$ be a finite polygonal lattice.
\begin{enumerate}[\rm(a)]
\item The forcing equivalence coincides with the polygonal equivalence.
\item The forcing order coincides with the polygonal order.
\end{enumerate}
\end{proposition}

Proposition~\ref{forcing=polygonal} lets us understand forcing among edges locally, in polygons.

\subsection{Bialgebraic completely semidistributive lattices} \label{fbs} 
The aim of this subsection is to introduce a well-behaved class of congruences that we call \emph{arrow-determined} on a lattice $L$ and to describe them in term of the Hasse quiver of $L$. It is known that, if $L$ is finite, $\Con L$ can be identified with a sublattice of the powerset of $\Hasse_1(L)$,
sending a congruence to the set of arrows it contracts, see for example \cite[9-5]{regions}. We generalize this result for complete lattices in Theorem \ref{mainiso}. To do so, we have to restrict our investigation to some well-behaved situations, since there are usually too many congruences.

Recall from the introduction that the lattice $L$ is \emph{weakly atomic} if $\Hasse[x, y]$ contains at least one arrow whenever $x<y$ in $L$.
We introduce the following notion, which is natural with respect to our problem.
\begin{definition} \label{arrowdet}
 For a congruence $\Theta$ on a lattice $L$, we say that $\Theta$ is \emph{arrow-determined} if for any ordered pair $x \leq y$ of $L$, $y \equiv_\Theta x$ if and only if all arrows of $\Hasse [x, y]$ are contracted by $\Theta$.
\end{definition}
 Notice that for a lattice $L$, the trivial congruence is arrow-determined if and only if $L$ is weakly atomic. More generally, we have the following characterization.
\begin{proposition} \label{arrdeteq}
  A complete lattice congruence $\Theta$ on a complete lattice $L$ is arrow-determined if and only if $L/\Theta$ is weakly atomic.
\end{proposition}

Before proving Proposition \ref{arrdeteq}, we introduce some notation that we use all along this subsection. When $\Theta$ is a congruence on a lattice $L$, we commonly identify $L/\Theta$ with $\pidown L$ as in Proposition~\ref{realpidown}. For $x \leq y$ in $\pidown L$, we denote by $[x, y]$ the interval of $L$ and by $[x,y]_\dn$ the interval of $\pidown L$. Similarly, we denote $\join$ and $\meet$ the lattice operations of $L$ and $\joindn$ and $\meetdn$ the lattice operations of $\pidown L$, even if $\join$ and $\joindn$ coincide as explained in Section \ref{lat prelim}. 

\begin{proof}[Proof of Proposition \ref{arrdeteq}]
First, suppose that $\Theta$ is arrow-determined.
Consider an ordered pair $x<y$ in $\pidown L$. 
As $\Theta$ is arrow-determined, there exists an arrow $v \to u$ in $\Hasse [x, y]$ that is not contracted by $\Theta$.
We claim that $\pidown v \to \pidown u$ is an arrow in $\Hasse[x,y]_\dn$.
Since $\pidown$ is order-preserving and since $v \to u$ is not contracted, we have $x\le\pidown u<\pidown v\le y$.
If there exists $z\in\pidown L$ with $\pidown u \leq z \leq \pidown v$, then $(u\join\pidown u)\le(u\join z)\le(u\join\pidown v)$.
We easily see that this simplifies to $u\le(u\join z)\le v$.
Since $v \to u$ is an arrow in $\Hasse L$, one of these inequalities must be an equality.
If $u=u\vee z$, then $z\leq u$, so $z=\pidown u$.  On the other hand,
if $u\vee z=v$, then observing that $\pidown u \vee \pidown z=z$,
the fact that $\pidown$ is a lattice morphism implies that $z=\pidown v$.  We have proved the claim, which implies that $L/\Theta$ is weakly atomic.

Conversely, suppose that $\pidown L$ is weakly atomic and let $[x, y]$ be an interval in~$L$.
Since $\pidown$ is order-preserving, we have $\pidown x\le\pidown y$.
If $x\equiv_\Theta y$, then all arrows of $\Hasse [x, y]$ are contracted by $\Theta$.
If $x \not\equiv_\Theta y$, then $\pidown x<\pidown y$.
We will exhibit an arrow in $\Hasse [x, y]$ that is not contracted by $\Theta$.
Since $\pidown L$ is weakly atomic, there exists an arrow $v \to u$ in $\Hasse [\pidown x,\pidown y]_\dn$.
Let $v'=v\join x$.
Since $x\equiv_\Theta\pidown x$, also $v'\equiv_\Theta v\join\pidown x=v$.
Let $u'=\Join\set{w\in L \mid w\le v',\,w\equiv_\Theta u}$.
Since $\Theta$ is a complete congruence, $u'\equiv_\Theta u$.
If $u'\equiv_\Theta v'$, then $u\equiv_\Theta u'\equiv_\Theta v'\equiv_\Theta v$, contradicting the fact that $v \to u$ is an arrow in $\Hasse [\pidown x,\pidown y]_\dn$.
By construction, $u'\le v'$, but in fact $u'<v'$ because $u'\not\equiv_\Theta v'$.
Since $v\le\pidown y\le y$ and $x\le y$, also $v'=v\join x\le y$.
Since $u\join x\le v\join x=v'$ and $u\join x\equiv_\Theta u\join\pidown x=u$, by the definition of $u'$, we see that $u\join x\le u'$, so $x\le u'$.

We have shown that $x\le u'<v'\le y$ and that $u'\not\equiv_\Theta v'$.
It remains only to show that $v'\to u'$ is an arrow in $\Hasse L$.
Suppose to the contrary that there exists $z$ with $u'<z<v'$.
Then since $\pidown$ is order preserving, $u\le\pidown z\le v$.
By the definition of~$u'$ we have $u'\not\equiv_\Theta z$, so $u<\pidown z$.
If $\pidown z=v$, then $z\ge v$, contradicting the fact that $z<v'$ and $v'=v\join x$.
Thus $u<\pidown z<v$, contradicting the fact that $v \to u$ is an arrow in $\Hasse [\pidown x,\pidown y]_\dn$.
By this contradiction, $v'\to u'$ is an arrow in $\Hasse [x, y]$ that is not contracted by $\Theta$, so we have verified that $\Theta$ is arrow-determined.
\end{proof}

An arrow-determined complete congruence $\Theta$ is completely specified by the set of arrows it contracts. 
Namely, $x \equiv_\Theta y$ if and only if all arrows of $\Hasse [x \meet y, x \join y]$ are contracted by $\Theta$.
We denote by $\Comb L$ the set of complete congruences over $L$ that are arrow-determined. Notice that, if $L$ is finite, we clearly have $\Con L = \Com L = \Comb L$. More generally, we obtain the following result.

\begin{proposition} \label{combmeetsl}  
The set $\Comb L$ is a complete meet-sublattice of $\Com L$, which is a complete meet-sublattice of $\Con L$.
In particular, $\Comb L$ and $\Com L$ are complete lattices.
\end{proposition}

\begin{proof}
The fact that $\Com L$ is a complete meet-sublattice of $\Con L$ is well-known and elementary.
Consider a family $(\Theta_i)_{i \in \II}$ of arrow-determined complete congruences, and denote $\Theta = \Meet_{i \in \II} \Theta_i$ (the meet is computed in $\Com L$ or equivalently in $\Con L$). Let $x \leq y$ in $L$ such that $x \not\equiv_\Theta y$. By definition, this means that there exists $i \in \II$ such that $x \not\equiv_{\Theta_i} y$. As $\Theta_i$ is arrow-determined, there exists an arrow $q: u \to v$ in $\Hasse [x, y]$ such that $u \not\equiv_{\Theta_i} v$. Again by the definition of $\Theta$, this implies $u \not\equiv_{\Theta} v$. In other words, we have proved that $\Theta$ is arrow-determined, hence $\Comb L$ is a complete meet-semilattice of $\Com L$.
\end{proof}

We now introduce a particularly well behaved class of complete lattices. We need several definitions and properties about complete lattices, that we recall briefly. For a more detailed introduction, we refer to \cite{AN} for completely semidistributive lattices and \cite{KL} for algebraic and co-algebraic lattices.  
The following definition appears in \cite{CH}, where the first bullet point is shown to be equivalent to $L$ being sectionally pseudocomplemented.

\begin{definition}
 A complete lattice $L$ is called \emph{completely semidistributive} if, for $x \in L$ and $\S \subseteq L$, the following hold:
\begin{itemize}
 \item If $x \meet  y=x\meet  z$ for all $y, z \in \S$, then $x\meet  \left(\Join \S\right) = x \meet  y$ for all $y \in \S$;
 \item If $x \join  y=x\join  z$ for all $y, z \in \S$, then $x\join  \left(\Meet \S\right) = x \join  y$ for all $y \in \S$.
\end{itemize}
\end{definition}

Recall also the following definitions.
\begin{definition}
 An element $x$ of a complete lattice $L$ is \emph{compact} if for any set $\S \subseteq L$ such that $x \leq \Join \S$, there exists a finite subset $\{x_1, x_2, \dots, x_n\} \subseteq \S$ such that $x \leq \Join_{i = 1}^n x_i$. Then $L$ is \emph{algebraic} if for any $x \in L$, there exists a set $\S$ of compact elements of $L$ such that $x = \Join \S$.

 Dually, $x$ is \emph{co-compact} if for any set $\S \subseteq L$ such that $x \geq \Meet \S$, there exists a finite subset $\{x_1, x_2, \dots, x_n\} \subseteq \S$ such that $x \geq \Meet_{i = 1}^n x_i$. Then $L$ is \emph{co-algebraic} if for any $x \in L$, there exists a set $\S$ of co-compact elements of $L$ such that $x = \Meet \S$.

 We say that $L$ is \emph{bialgebraic} if it is algebraic and co-algebraic.
\end{definition}

We now state the main results of this subsection.
The first result is known.
(See, for example \cite[2.2]{CD} or \cite[Theorem~3.6]{NationNotes}.)
\begin{theorem} \label{arrsep}
 Let $L$ be a complete lattice. If $L$ is algebraic or if $L$ is co-algebraic, then $L$ is weakly atomic.
\end{theorem}

Moreover, if we restrict to quotients of completely semidistributive and bialgebraic complete lattices, then the converse holds in the following sense. 
\begin{theorem} \label{eqardalg}
 Let $L$ be a complete lattice that is completely semidistributive and bialgebraic.
Then a complete congruence  $\Theta \in \Com L$ is arrow-determined if and only if $L/\Theta$ is weakly atomic if and only if $L/\Theta$ is bialgebraic.
\end{theorem}

We denote by $\Ideals (\Hasse_1(L))$ the set of subsets $ \S \subseteq \Hasse_1(L) $ such that for any $q \in \S$ and $q' \in \Hasse_1(L)$, if $q \forces q'$ then $q' \in \S$.
It is naturally a complete lattice with respect to inclusion (joins coincide with unions and meets coincide with intersections).

\begin{theorem} \label{mainiso}
Let $L$ be a complete lattice that is completely semidistributive and bialgebraic. Then $\Comb L$ is isomorphic to $\Ideals(\Hasse_1 (L))$, mapping a congruence to the set of arrows it contracts.
In particular, $\Comb L$ is distributive. 
\end{theorem}

Before proving the theorems, we give an example showing that, usually, $\Comb L$ is much smaller than $\Com L$, even when $L$ is bialgebraic, completely distributive and arrow separated. 

\begin{example} \label{exconga}
 Consider \[L := \{(-\infty, x) \mid x \in \R\} \cup \{(-\infty, x] \mid x \in \R\} \cup \{\emptyset, \R\}\] totally ordered by inclusion. It is a complete, completely distributive, lattice. The arrows of $\Hasse L$ are $(-\infty, x] \to (-\infty, x)$ for each $x \in \R$, so $L$ is weakly atomic.

 Moreover, $\cjirr L = \{(-\infty, x] \mid x \in \R\}$ and $\cmirr L = \{(-\infty, x) \mid x \in \R\}$. The set of compact elements of $L$ is $\cjirr L \cup \{\emptyset\}$ and the set of co-compact elements is $\cmirr L \cup \{\R\}$, hence we see easily that $L$ is bialgebraic. Consider the congruence $\Theta$ of $L$ identifying $(0,x)$ with $(0,x]$ for each $x$. It is a complete congruence. Moreover, $L/\Theta = \R \cup \{-\infty, +\infty\}$ so $\Hasse(L/\Theta)$ has no arrows,  hence is not weakly atomic. In particular, $\Theta \in \Com L \setminus \Comb L$. More precisely, the only $\Theta' \in \Comb L$ that is bigger than $\Theta$ identifies all elements of $L$. Moreover, for each complete congruence of $\R \cup \{-\infty, +\infty\}$, there is a corresponding complete congruence of $L$ between $\Theta$ and $\Theta'$.
\end{example}

We now give several lemmas.

\begin{lemma} \label{lemsd}
 Suppose that $L$ is completely semidistributive.
 \begin{enumerate}[\rm(a)]
  \item Any interval $[u, v]$ of $L$ is completely semidistributive.
  \item For any $\Theta \in \Com L$, $L/\Theta$ is completely semidistributive.
 \end{enumerate}
\end{lemma}

\begin{proof}
 (a) The property is immediate as $[u, v]$ is a complete sublattice of $L$.

 (b) 
Let $\S \subseteq \pidown L$ and $x \in \pidown L$ such that for all $y, z \in \S$, we have $x \joindn y = x \joindn z$, or equivalently $x \join y = x \join z$, since $\pidown L$ is a complete join-sublattice of $L$.
 As $L$ is completely semidistributive, we deduce that, $x \join \Meet \S = x \join z$ for all $z \in \S$. 
 Because $\pidown: L \to \pidown L$ is a morphism of complete lattices, we deduce $\pidown x \joindn \Meetdn \{\pidown y \mid y \in \S\} = \pidown x \joindn \pidown z$ for all $z \in \S$.
 As $x \in \pidown L$ and $\S \subseteq \pidown L$, we have $\pidown x = x$ and $\pidown z = z$ for any $z \in \S$ so $x \joindn \Meetdn \S = x \joindn z$.
 The dual argument, using $\piup L$, completes the proof.
\end{proof}

\begin{lemma} \label{lemalgint}\
 \begin{enumerate}[\rm (a)]
  \item If $x \in L$ is compact, then for any interval $[u, v]$ with $v \geq x$, we have $u \join x$ compact in $[u, v]$.
  \item If $x \in L$ is co-compact, then for any interval $[u, v]$ with $u \leq w$, we have $v \meet x$ co-compact in $[u, v]$.
  \item If $L$ is algebraic, any interval $[u, v]$ of $L$ is algebraic.
  \item If $L$ is co-algebraic, any interval $[u, v]$ of $L$ is co-algebraic.
 \end{enumerate}
\end{lemma}

\begin{proof}
 By symmetry, we prove (a) and (c).

 (a) Without loss of generality, suppose that $x \neq 0$. Let $\S \subseteq [u, v]$ such that $u \join x \leq \Join \S$. We have $x \leq \Join \S$ so there exists a (non-empty) finite subset $F \subseteq \S$ satisfying $x \leq \Join F$. Moreover, as $F \subseteq [u, v]$, $u \leq \Join F$ so $u \join x \leq \Join F$.

 (c) Let $x \in [u, v]$. As $L$ is algebraic, $x = \Join \S$ for a set $\S \subseteq L$ of compact elements. Then $x = \Join \{u \join y \mid y \in \S\}$. So, by (a), $x$ is a join of compact elements of $[u, v]$. So $[u, v]$ is algebraic.
\end{proof}

\begin{lemma} \label{cjitocompact}\
 \begin{enumerate}[\rm (a)]
  \item If $L$ is algebraic and $x \in L$ is completely join-irreducible, then $x$ is compact.
  \item If $L$ is co-algebraic and $x \in L$ is completely meet-irreducible, then $x$ is co-compact.
 \end{enumerate}
\end{lemma}

\begin{proof}
 Let us prove (a). As $L$ is algebraic, $x = \Join \S$ for a set $\S \subseteq L$ of compact elements. As $x$ is completely join-irreducible, $x \in \S$, so $x$ is compact.
\end{proof}

\begin{lemma} \label{lemalgquot}
  Let $L$ be a complete lattice and $\Theta \in \Com L$.
 \begin{enumerate}[\rm (a)]
  \item For $u \in \pidown L$, we have $u \in \cjirr L$ if and only if $u \in \cjirr (\pidown L)$.
  \item For $u \in \piup L$, we have $u \in \cmirr L$ if and only if $u \in \cmirr (\piup L)$.
  \item For $u \in \pidown L$, $u$ is compact in $L$ if and only if it is compact in $\pidown L$.
  \item For $u \in \piup L$, $u$ is co-compact in $L$ if and only if it is co-compact in $\piup L$.
 \end{enumerate}
\end{lemma}

\begin{proof}
 By symmetry, we prove (a) and (c).

 (a) Suppose that $u \in \cjirr L$. Let $\S \subseteq \pidown L$ such that $u = \Joindn \S$, so $u = \Join \S$, as $\pidown L$ is a complete join-sublattice of $L$. As $u \in \cjirr L$, we get $u \in \S$. So $u \in \cjirr (\pidown L)$. Conversely, suppose $u \in \cjirr (\pidown L)$ and consider $\S \subseteq L$ such that $u = \Join \S$. As $\pidown: L \to \pidown L$ is a morphism of complete lattice, we get $u = \pidown u = \Joindn \{\pidown x \mid x \in \S\}$. So $u = \pidown x$ for some $x \in \S$. In particular $u \leq x$. As $u = \Join \S  \geq x$, we get $u = x$. Finally, $u \in \cjirr L$.

 (c) Suppose that $u$ is compact in $L$. Let $\S \subseteq \pidown L$ such that $u \leq \Joindn \S = \Join \S$. As $u$ is compact in $L$, there exists a finite subset $F \subseteq \S$ such that $u \leq \Join F = \Joindn F$. So $u$ is compact in $\pidown L$. Conversely, suppose that $u$ is compact in $\pidown L$ and let $\S \subseteq L$ such that $u \leq \Join \S$. We get $u = \pidown u \leq \Joindn \{\pidown x \mid x \in \S\}$, so $u \leq \Joindn \{\pidown x \mid x \in F\}$ for some finite subset $F \subseteq \S$. For any $x \in F$, $\pidown x \leq x$ so $u \leq \Join \{\pidown x \mid x \in F\} \leq \Join F$. So $u$ is compact in $L$.
\end{proof}

The following lemma is known.
(See, for example, \cite[6.1]{CD} or \cite[Lemma~2.1]{AGT}.)

\begin{lemma} \label{algtocji}
 Let $L$ be a complete lattice.
 \begin{enumerate}[\rm(a)]
  \item If $L$ is co-algebraic, then any $x \in L$ is a join of elements of $\cjirr L$.
  \item If $L$ is algebraic, then any $x \in L$ is a meet of elements of $\cmirr L$.
 \end{enumerate}
\end{lemma}

\begin{proof}
 By symmetry, we prove (a). Consider $x' = \Join \{z \in \cjirr L \mid z \leq x\}$. It suffices to prove that $x' = x$. We have $x' \leq x$.
 Suppose that $x \not\leq x'$. As $L$ is co-algebraic, $x' = \Meet \S$ for some subset $\S \subseteq L$ of co-compact elements. Then $x \not\leq x'$ implies $x \not\leq y$ for some $y \in \S$.

 Let $E = \{z \in [0,x] \mid z \not\leq y\}$. As $x \in E$, $E$ is non-empty. If $\II \subseteq E$ is non-empty and totally ordered, then $\Meet \II \in E$. Indeed, if $\Meet \II \leq y$, as $y$ is co-compact, there exists a non-empty finite subset $F \subseteq \II$ such that $\Meet F \leq y$. As $F$ is non-empty, finite and totally ordered, $\Meet F \in F \subseteq E$, which contradicts $\Meet F \leq y$. Hence, by Zorn's Lemma, $E$ admits a minimal element $z$. We claim that $z$ is completely join-irreducible. Indeed, if $z = \Join \S'$ for some $\S' \subseteq L$, $z \not\leq y$ implies $z' \not\leq y$ for some $z' \in \S'$. As $z' \leq z$, by minimality of $z$ in $E$, we have $z' = z$.

 As $z \in \cjirr L$ and $z \leq x$, we have $z \leq x'$ by definition. As  $x' = \Meet \S$ and $y \in \S$, we get $z \leq y$, which contradicts $z \in E$.
\end{proof}

\begin{lemma} \label{exar}
 If $x \in L$ is compact and $x \neq 0$, then there exists an arrow in $\Hasse L$ starting at $x$.
\end{lemma}

\begin{proof} 
Consider the set of chains in $\{y \in L \mid y < x\}$, ordered by containment.
This partially ordered set satisfies the hypotheses of Zorn's Lemma, so there is a maximal totally ordered subset $\II$ of $\{y \in L \mid y < x\}$.
Let $z = \Join \II$. 
We have $z \leq x$, but since $x$ is compact, if $z=x$ then $x$ is a join of finitely many elements of $\II$.
However, since $\II$ is totally ordered, this join is strictly below $x$.
By this contradiction, we see that $z < x$.
If $u \in L$ satisfies $z \leq u < x$, we get that $\II \cup \{u\}$ is a totally ordered subset of $\{y \in L \mid y < x\}$, so by maximality of $\II$, $u \in \II$ so $u \leq z$ and $u = z$. Finally, there is an arrow $x \to z$ in $\Hasse L$.
\end{proof}

We now prove the first main result of this section.

\begin{proof}[Proof of Theorem~\ref{arrsep}] 
 By symmetry, we suppose that $L$ is algebraic. Hence, by Lemma \ref{lemalgint}(c), any interval $[u, v]$ of $L$ is algebraic.
 If $u < v$, since $v$ is a join of compact elements of $[u,v]$, there is an element $x\neq u$ in $[u, v]$ that is compact in $[u, v]$.
Thus, by Lemma \ref{exar}, $\Hasse [u, v]$ contains at least one arrow.
So $L$ is weakly atomic.
\end{proof}

The following proposition generalizes a known result for finite lattices (see for example \cite[Proposition~9-5.20(i)]{regions}).

\begin{proposition} \label{bijfeqjirr}  
 Suppose that $L$ is completely semidistributive. Let $x \to y$ be an arrow of $\Hasse L$. Then
 \begin{enumerate}[\rm (a)]
  \item There exists an arrow $j \to j_*$ in $\Hasse L$ forcing equivalent to $x \to y$ with $j$ completely join-irreducible and $j \leq x$, $j_* \leq y$ and $j \not\leq y$. 
  \item There exists an arrow $m^* \to m$ in $\Hasse L$ forcing equivalent to $x \to y$ with $m$ completely meet-irreducible and $m \geq y$, $m^* \geq x$ and $m \not\geq x$.
 \end{enumerate}
\end{proposition}

\begin{proof}
 By symmetry, we prove (a).
We consider the set $\S := \{z \in  L \mid y \join z = x\}=\set{z\in L\mid z\le x,z\not\le y}$.
It is not empty as $x \in \S$.
Let $j := \Meet \S$.
By complete semidistributivity, $y \join j = x$, so in particular $j \leq x$ and $j \not\leq y$.
If $j$ is not completely join-irreducible, then it is a join of elements strictly below $j$, but then since all of those elements are also less than $y$, their join is below $y$, contradicting $y\join j=x$.
We conclude that $j$ is completely join-irreducible and thus that there is a unique arrow $j\to j_*$ in $\Hasse L$. By definition of $j$, we have $j_* \leq y$.
If a congruence has $y \equiv x$ then also $y \meet j \equiv x \meet j = j$, so, as $y \meet j \leq j_* \leq j$, we get $j_* \equiv j$.
Conversely, if $j \equiv j_*$ then also $x = y \join j \equiv y \join j_* = y$.
\end{proof}

\begin{lemma} \label{exforcb}
 Suppose that $L$ is completely semidistributive, $x \to y$ is an arrow of $\Hasse L$ and $[u, v]$ is an interval of $L$.
 \begin{enumerate}[\rm(a)]
  \item If $L$ is algebraic, $x \leq v$ and $u \meet x \leq y$, then there exists an arrow $z \to t$ in $[u, v]$ that forces $x \to y$.
  \item If $L$ is co-algebraic, $y \geq u$ and $v \join y \geq x$, then there exists an arrow $z \to t$ in $[u, v]$ that forces $x \to y$.
 \end{enumerate}
\end{lemma}

\begin{proof}
By symmetry, we only prove (a).
By Lemmas \ref{lemsd} and \ref{lemalgint}, $[u \meet x, v]$ is completely semidistributive and algebraic, so we can suppose without loss of generality that $u \meet x = 0$ and $v = 1$.  
By Proposition \ref{bijfeqjirr}(a), there exists an arrow $j \to j_*$ in $\Hasse L$ that is forcing equivalent to $x \to y$ such that $j$ is completely join-irreducible and $j \leq x$.
As $L$ is algebraic, by Lemma \ref{cjitocompact}, $j$ is compact.
Hence, by Lemma \ref{lemalgint}, $z := u \join j$ is compact in $[u, 1]$.
As $u \meet x = 0$, $j \leq x$ and $j \neq 0$, we get that $j \not\leq u$ so $z > u$.
Hence, by Lemma \ref{exar}, there exists an arrow $z \to t$ in $\Hasse [u, 1]$.
By definition of $z$, we have $t \not\geq j$. As $j$ is completely join-irreducible, we deduce $t \meet j \leq j_* < j$.
For any congruence having $z \equiv t$, we have $j = z \meet j \equiv t \meet j$, so $j \equiv j_*$. Hence, $(z \to t) \forces (j \to j_*) \forceseq (x \to y)$.
\end{proof}

\begin{lemma} \label{trs}
 Suppose that $L$ is bialgebraic and completely semidistributive. Let $\II \in \Ideals(\Hasse_1 L)$. Consider two intervals
 $[u, v]$ and $[u', v']$ of $L$ such that
\[\Hasse_1 [u, v] \subseteq \II \quad \text{and} \quad \Hasse_1 [u', v'] \subseteq \II  \quad \text{and} \quad  [u, v] \cap [u', v'] \neq \emptyset.\] Then we have $\Hasse_1 [u \meet u', v \join v'] \subseteq \II$.
\end{lemma}

\begin{proof}
 By Lemmas \ref{lemsd} and \ref{lemalgint}, $[u \meet u', v \join v']$ is completely semidistributive and bialgebraic, so we can suppose without loss of generality that $u \meet u' = 0$ and $v \join v' = 1$, and prove that $\Hasse_1 L \subseteq \II$.

We suppose first that $v' = u$, so that $u' = 0$ and $v = 1$.
Consider an arrow $x \to y$ of $\Hasse L$.
By Lemma \ref{bijfeqjirr}, there exists an arrow $j \to j_*$ of $\Hasse L$ that is forcing equivalent to $x \to y$ with $j$ completely join-irreducible.
If $j \leq v' = u$, we have $(j \to j_*) \in \Hasse_1 [u', v'] \subseteq \II$.
Otherwise, $u \meet j < j$, so $u \meet j \leq j_*$ and by Lemma \ref{exforcb}(a), there is an arrow in $\Hasse_1 [u, 1] \subseteq \II$ that forces $j \to j_*$.
So $j \to j_*$ and $x \to y$ are in $\II$.

 Let us go back to the general case and fix $c \in [u, v] \cap [u', v']$. Consider an arrow $x \to y$ of $\Hasse [0, u]$. We have $x \leq u \leq c \leq v'$ and $u' \meet x \leq u' \meet u = 0 \leq y$, so by Lemma \ref{exforcb}(a), there is an arrow of $\Hasse [u', v'] \subseteq \II$ that forces $x \to y$, so $(x \to y) \in \II$. We proved that $\Hasse_1 [0, u] \subseteq \II$. Symmetrically, we get that $\Hasse_1 [v, 1] \subseteq \II$. So using the first case for the intervals $[0, u]$ and $[u, v]$, we deduce that $\Hasse_1 [0, v] \subseteq \II$. Using again the first case for $[0, v]$ and $[v, 1]$, we conclude $\Hasse_1 L \subseteq \II$.
\end{proof}

Let $\II \in \Ideals(\Hasse_1 L)$ and define the relation $\equiv$ on $L$ by $x \equiv y$ if and only if $\Hasse_1 [x \meet y, x \join y] \subseteq \II$.

\begin{lemma} \label{equivcong}
 Suppose that $L$ is completely semidistributive and bialgebraic. Then, the relation $\equiv$ is a complete congruence that is arrow-determined.
\end{lemma}

\begin{proof}
 First of all, it is clearly reflexive and symmetric. For the transitivity, suppose that $x \equiv y$ and $y \equiv z$. It means that \[\Hasse_1 [x \meet y, x \join y] \subseteq \II \quad \text{and} \quad \Hasse_1 [y \meet z, y \join z] \subseteq \II.\] As $y \in [x \meet y, x \join y] \cap [y \meet z, y \join z]$, by Lemma \ref{trs}, we get $\Hasse_1 [x \meet y \meet z, x \join y \join z] \subseteq \II$, so $\Hasse_1 [x \meet z, x \join z] \subseteq \II$, so $x \equiv z$. Therefore $\equiv$ is an equivalence relation.

We consider an index set $\S$ and two families $(x_i)_{i \in \S}$ and $(y_i)_{i \in \S}$ such that $x_i \equiv y_i$ for all $i \in \S$. Let $x = \Join_{i \in \S} x_i$ and $y = \Join_{i \in \S} y_i$ and let us prove that $x \equiv y$.

For each $i$, denote $u_i = x_i \meet y_i$ and $v_i = x_i \join y_i$, $u = \Join_{i \in \S} u_i$ and $v = \Join_{i \in \S} v_i$.
We have $[x \meet y, x \join y] \subseteq [u, v]$, so it suffices to prove that $\Hasse_1 [u, v] \subseteq \II$.  
Consider an arrow $m^* \to m$ of $\Hasse_1 [u, v]$ with $m$ completely meet-irreducible in $[u, v]$.
As $m \not\geq v$, there exists $i \in \S$ such that $m \not\geq v_i$.
Then $v_i \join m > m$.
As $v_i \join m \in [u, v]$ and $m$ is completely meet-irreducible in $[u, v]$, we deduce $v_i \join m \geq m^*$.
So, by Lemma \ref{exforcb}(b), $m^* \to m$ is forced by an arrow of $\Hasse [u_i, v_i]$.
By definition, we have $\Hasse_1 [u_i, v_i] \subseteq \II$ so $(m^* \to m) \in \II$.
By Proposition \ref{bijfeqjirr}, each arrow of $\Hasse [u, v]$ is forcing equivalent to an arrow $m^* \to m$ of $[u, v]$ with $m$ completely meet-irreducible in $[u, v]$.
Hence, $\Hasse_1 [u, v] \subseteq \II$.

 The proof that $\equiv$ is compatible with meets is dual.
 Finally, the fact that $\equiv$ is arrow-determined is a direct consequence of its definition.
\end{proof}

\begin{proof}[Proof of Theorem \ref{mainiso}]
There is a well defined, order-preserving map from $\Comb L$ to $\Ideals(\Hasse_1 L)$ mapping a congruence to the set of arrows it contracts. By definition of arrow-determined congruences, this map is injective, and by Lemma \ref{equivcong}, it is surjective.
The inverse map is order-preserving as well, so the map is an isomorphism of complete lattices.
Since $\Ideals(\Hasse_1 L)$ is closed under union and intersection, distributivity of $\Comb L$ follows. 
\end{proof}

Before proving Theorem \ref{eqardalg}, we need a last lemma.

\begin{lemma} \label{quotcomp}
 Consider a complete congruence $\Theta$ of $L$.  
 \begin{enumerate}[\rm(a)]
  \item Let $j \in \cjirr L$ such that $j \to j_*$ is not contracted by $\Theta$. Then its image $\overline{j}$ in $L/\Theta$ is completely join-irreducible. If $L$ is algebraic, then $\overline{j}$ is compact.
  \item Let $m\in \cmirr L$ such that $m^* \to m$  is not contracted by $\Theta$. Then its image $\overline{m}$ in $L/\Theta$ is completely meet-irreducible. If $L$ is co-algebraic, then $\overline{m}$ is co-compact.
 \end{enumerate}
\end{lemma}

\begin{proof}
 By symmetry, we prove (a).  
 As $j \in \cjirr L$ is not contracted, $j = \pidown j \in \pidown L$.
 So, by Lemma \ref{lemalgquot}(a), $j \in \cjirr(\pidown L)$.
 If $L$ is algebraic, by Lemma \ref{cjitocompact}(a), $j$ is compact in $L$ so by Lemma \ref{lemalgquot}(c), $j$ is compact in $\pidown L$.
\end{proof}

\begin{proof}[Proof of Theorem \ref{eqardalg}]
First of all, $\Theta$ is arrow-determined if and only if $L/\Theta$ is weakly atomic is Proposition \ref{arrdeteq}. Moreover, if $L/\Theta$ is bialgebraic, then $L/\Theta$ is weakly atomic by Theorem \ref{arrsep}. 

Conversely, suppose that $\Theta$ is arrow-determined.
Let $x \in L$ and \[E = \{j \in \cjirr L \mid j \leq x \text{ and $j \to j_*$ is not contracted by $\Theta$}\}\] and $x' = \Join E$.
We have $x' \leq x$. Suppose that $x' \not\equiv_\Theta x$.
As $\Theta$ is arrow-determined, there exists an arrow $y \to z$ in $\Hasse [x', x]$ that is not contracted by $\Theta$.
By Proposition \ref{bijfeqjirr}(a), there exists $j \to j_*$ in $\Hasse L$ that is forcing equivalent to $y \to z$ such that $j \in \cjirr L$ and $j \leq y \leq x$.
In particular, $j \to j_*$ is not contracted by $\Theta$ so $j \in E$.
As $j\not\leq x'$, this contradicts the definition of $x'$.
We proved that $x' \equiv_\Theta x$.

 Moreover, by Lemma \ref{quotcomp}(a), the images of the elements of $E$ are compact in $L/\Theta$, so any element of $L/\Theta$ is a join of compact elements. We proved that $L/\Theta$ is algebraic. Symmetrically, $L/\Theta$ is co-algebraic.
\end{proof}

We finish this subsection by noticing that, in our setting, the forcing order coincides with the \emph{complete forcing order} and the \emph{arrow-determined forcing order}.

For each subset $\S \subseteq \Hasse_1(L)$, we define $\com(\S)$ to be the minimum complete lattice congruence that contracts all elements of $\S$ and $\comb{\S}$ to be the minimal arrow-determined complete lattice congruence that contracts all elements of $\S$.
We define the \emph{complete forcing order} on $\Hasse_1(L)$ by
\[a \forces[c] b\Longleftrightarrow\com(a)\ge\com(b)\text{ in $\Com L$},\]
and the \emph{arrow-determined forcing order} on $\Hasse_1(L)$ by
\[a \forces[ca] b\Longleftrightarrow\comb(a)\ge\comb(b)\text{ in $\Comb L$}.\]

While it is elementary that for $\S \subseteq \Hasse_1(L)$ the congruences $\con(\S)$, $\com(\S)$ and $\comb(\S)$ are in general distinct, when $L$ is completely semidistributive and bialgebraic, we obtain the following proposition.

\begin{proposition} \label{concom}
 Suppose that $L$ is completely semidistributive and bialgebraic. The complete forcing order, the arrow-determined forcing order and the forcing order coincide.
\end{proposition}

\begin{proof}
 First of all, it is immediate that $q \forces q' \Rightarrow q \forces[c] q' \Rightarrow q \forces[ca] q'$. Conversely, let $q: x \to y$ and $q': x' \to y'$ be two arrows of $\Hasse L$ such that $q \forces[ca] q'$. It is immediate that the arrows contracted by $\con q$ form a forcing ideal $\II$. By Theorem \ref{mainiso}, there exists a complete arrow-determined congruence $\Theta$ contracting exactly the elements of $\II$. As $q \forces[ca] q'$, it means that $q' \in \II$ so $q \forces q'$.
\end{proof}

\subsection{Complete congruence uniformity}   

Continuing Section \ref{fbs}, we now generalize the notion of congruence uniformity to complete lattices. Let us consider a complete lattice $L$. We restrict our attention to an appropriate subset of the arrows of $\Hasse L$.
For $j \in \cjirr L$ and $\Theta \in \Com L$, we say that $\Theta$ \emph{contracts} $j$ if it contracts the arrow $j \to j_*$. For $m \in \cmirr L$, we say that $\Theta$ \emph{contracts} $m$ if it contracts $m^* \to m$. In the same way, for $j, j' \in \cjirr L$, we say that $j$ \emph{forces} $j'$ and we write $j \forces j'$ if $j \to j_*$ forces $j' \to j'_*$, and for $m, m' \in \cmirr L$, we say that $m$ \emph{forces} $m'$ and we write $m \forces m'$ if $m^* \to m$ forces $m'^* \to m'$. We denote by $\Ideals (\cjirr L)$ and $\Ideals (\cmirr L)$ the ideals of this relation.

\begin{definition} \label{ccongunifd}
 A complete lattice $L$ is \emph{completely congruence uniform} if the following conditions hold.
\begin{itemize}
\item Forcing is a partial order on $\cjirr L$ and
the map from $\Comb L$ to $\Ideals(\cjirr L)$ sending a congruence to the set of completely join-irreducible elements it contracts is a bijection.
\item Forcing is a partial order on $\cmirr L$ and the map from $\Comb L$ to $\Ideals(\cmirr L)$ sending a congruence to the set of completely meet-irreducible elements it contracts is a bijection.
\end{itemize}
\end{definition}

Notice that in Definition \ref{ccongunifd}, the two bijections $\Comb L \leftrightarrow \Ideals(\cjirr L)$ and $\Comb L \leftrightarrow \Ideals(\cmirr L)$ are automatically isomorphisms of complete lattices.

Notice that this definition is equivalent to the definition of a congruence uniform lattice when $L$ is finite. We now give easier criteria for congruence uniformity.

\begin{proposition} \label{congunifchar}
 For a complete lattice $L$, we have $\mathrm{(i)}\Rightarrow\mathrm{(ii)}\Rightarrow\mathrm{(iii)}$:
 \begin{enumerate}[\rm(i)]
  \item $L$ is completely congruence uniform.
  \item \begin{itemize}
  \item The map $\cjirr L\to\Comb L$ given by $j\mapsto \comb(j\to j_*)$ is a bijection between $\cjirr L$ and $\cjirr \Comb L$.
  \item The map $\cmirr L\to\Comb L$ given by $m\mapsto \comb(m^*\to m)$ is a bijection between $\cmirr L$ and $\cjirr \Comb L$.
  \end{itemize}
  \item \begin{itemize}
  \item The map $\cjirr L\to\Comb L$ given by $j\mapsto \comb(j\to j_*)$ is injective.
  \item The map $\cmirr L\to\Comb L$ given by $m\mapsto \comb(m^*\to m)$ is injective.
  \end{itemize} 
 \end{enumerate}
Moreover, if $L$ is completely semidistributive and bialgebraic, then $\mathrm{(iii)}\Rightarrow\mathrm{(i)}$.
\end{proposition}

\begin{proof}
 By symmetry, we consider the conditions about completely join-irreducible elements.

 (i) $\Rightarrow$ (ii) For $j \in \cjirr L$, we write $\II_j := \{j' \in \cjirr L \mid j \forces j'\}$. Consider a completely join-irreducible element $\II$ of $\Ideals(\cjirr L)$. We have $\II = \Join_{j \in \II} \II_j$, so, as $\II$ is completely join-irreducible, $\II = \II_j$ for some $j$. Conversely, it is immediate that, for any $j \in \cjirr L$, no proper subideal of $\II_j$ contains $j$ so $\II_j$ is completely join-irreducible. Then, the conclusion follows from the isomorphism $\Comb L \cong \Ideals(\cjirr L)$.

 (ii) $\Rightarrow$ (iii) This is immediate.

 We now suppose that $L$ is completely semidistributive and bialgebraic.

 (iii) $\Rightarrow$ (i) First of all, our assumption implies that the forcing on $\cjirr L$ is a partial order. Second, 
by Theorem \ref{mainiso}, there is an isomorphism from $\Ideals(\Hasse_1(L))$ to $\Comb L$, mapping $\II$ to $\Join_{q \in \II} \comb(q)$. Moreover, by Proposition \ref{bijfeqjirr}, in each forcing equivalence class of arrows of $\Hasse L$, there is an arrow $j \to j_*$ such that $j$ is completely join-irreducible in $L$.
Thus $\Comb L\cong\Ideals(\Hasse_1(L)) \cong \Ideals(\cjirr L)$. 
\end{proof}

Notice that, in general, (i), (ii) and (iii) of Proposition \ref{congunifchar} are not equivalent, as shown in the following example.

\begin{example}\

 (a) We take $L$ as in Example \ref{exconga}, and define $L' := L \cup \{\alpha\}$ where $\emptyset < \alpha < \R$, but for any $x \in \R$, $(-\infty, x)$ and $(-\infty, x]$ are not comparable with $\alpha$. We get easily that $\cjirr L' = \{(-\infty,x] \mid x \in \R\} \cup \{\alpha\}$ and $\cmirr L' = \{(-\infty,x) \mid x \in \R\} \cup \{\alpha\}$, and as in Example \ref{exconga}, $L'$ is complete and bialgebraic. 

 Moreover, if we consider the arrow $q_x: (-\infty,x] \to (-\infty,x)$ of $\Hasse L'$, it is immediate that $\comb(q_x)$ contracts only $q_x$. On the other hand, $\comb(\alpha \to \emptyset) = \comb(\R \to \alpha)$ identifies all elements of $L'$. So $L'$ satisfies (iii). But $\comb(\alpha \to \emptyset)$ is not completely join-irreducible in $\Comb L'$ as it is equal to $\Join_{x \in \R} \comb(q_x)$. So $L'$ does not satisfy (ii). It follows that $L'$
 is not completely semidistributive.

 (b) We now consider the lattice $L'' = L \cup \bar L$, where $\bar L$ is a copy of $L$, where we identify $\emptyset$ and $\bar \emptyset$ on the one hand, and $\R$ and $\bar \R$ on the other hand. Moreover, no other elements of $L$ and $\bar L$ are comparable. As before $L''$ is complete and bialgebraic. For each arrow $(x \to y) \in \Hasse_1(L'')$, we have $x \in \cjirr L''$,  $y \in \cmirr L''$ and $\comb(x \to y)$ contracts only $x \to y$. So $L''$ satisfies (ii).

 On the other hand, the forcing on $\cjirr L''$ is trivial, so $\Ideals(\cjirr L'') = 2^{\cjirr L''}$. We have a strict inclusion of ideals $\cjirr L \subsetneq \cjirr L''$, and \[\Join_{j \in \cjirr L} \comb(j \to j_*) = \Join_{j \in \cjirr L''} \comb(j \to j_*)\] is the maximum congruence, so (i) does not hold. As before, it implies that $L''$ is not completely semidistributive.
\end{example}

\section{Lattice of torsion classes}

\subsection{Elementary properties}\label{prelim-tor}
Let $k$ be a field.
We consider an associative, finite-dimensional $k$-algebra $A$ with an identity element.
We denote by $\mod A$ the category of finitely generated left $A$-modules. For $M\in\mod A$, we denote by $\add M$ the full subcategory of $\mod A$ consisting of direct summands of finite direct sums of copies of $M$.
For a class $\CC \subseteq \mod A$ , we define its orthogonal categories in $\mod A$ by
\begin{eqnarray*}
\CC^{\perp_A}=\CC^{\perp}&:=&\{X\in\mod A\mid\Hom_A(\CC,X)=0\},\\
{}^{\perp_A}\CC={}^{\perp}\CC&:=&\{X\in\mod A\mid\Hom_A(X,\CC)=0\}.
\end{eqnarray*}
We denote by $\Filt \CC$ the full subcategory of $A$-modules filtered by modules in $\CC$. Moreover, when $\CC_1, \dots, \CC_n \subseteq \mod A$, $\Filt(\CC_1, \dots, \CC_n) := \Filt(\CC_1 \cup \cdots \cup \CC_n)$.

We say that a full subcategory $\TT$ of $\mod A$ is a \newword{torsion class} (respectively, \newword{torsion-free class})
if it is closed under factor modules (respectively, submodules), isomorphisms and extensions.
For any subcategory $\CC$ of $\mod A$, ${}^\perp\CC$ is a torsion class and $\CC^\perp$ is a torsion-free class.
We denote by $\tors A$ (respectively, $\torf A$) the set of torsion classes (respectively, torsion-free classes) in $\mod A$.
The set $\tors A$ is closed under intersection, so it forms a complete lattice with respect to inclusion, with a unique maximal element $\mod A$ and a unique minimal element $\{0\}$ \cite[Proposition 1.3]{IRTT}.
The meet is intersection and the join $\Join \S$ of $\S \subset \tors A$ is the meet of all upper bounds of $\S$. Alternatively, $\Join \S$ is given explicitly as $\Filt (\Fac \S)$, the full subcategory of $\mod A$ consisting of modules that are filtered by quotients of modules in $\S$.
The set $\torf A$ is similarly a complete lattice with respect to inclusion.
For any subcategory $\X \subseteq \mod A$, there is a smallest torsion class $\T(\X)$ containing $\X$, namely the meet (\emph{i.e.}\ intersection) of all torsion classes containing $\X$.
We have anti-isomorphisms
\[(-)^\perp:\tors A\to\torf A\ \text{ and }\ {}^\perp(-):\torf A\to\tors A\]
of complete lattices.
A \emph{torsion pair} is a pair $(\TT,\FF)$ of a torsion class $\TT$ and a torsion-free class $\FF$ in $\mod A$ satisfying $\TT^{\perp}=\FF$ and $\TT={}^{\perp}\FF$.

We start by proving that the lattice $\tors A$ enjoys the properties investigated in Section \ref{fbs}. 

\begin{theorem} \label{semidistributiveandbialg}
 Let $A$ be a finite-dimensional algebra.
 \begin{enumerate}[\rm(a)]
  \item The lattice $\tors A$ is completely semidistributive.
  \item The lattice $\tors A$ is bialgebraic, and hence weakly atomic.
 \end{enumerate}
\end{theorem}

Notice that Theorem \ref{semidistributiveandbialg}(a) is a bit stronger than the semidistributivity proven in \cite[Theorem 4.5]{GM}.
We give the following proposition before proving Theorem \ref{semidistributiveandbialg}.

\begin{proposition} \label{compacttors}\
 \begin{enumerate}[\rm(a)]
  \item For $\TT \in \tors A$, $\TT$ is compact if and only if $\TT = \T(X)$ for some $X \in \mod A$.
  \item For $\TT \in \tors A$, $\TT$ is co-compact if and only if $\TT = {}^\perp X$ for some $X \in \mod A$.
 \end{enumerate}
\end{proposition}

\begin{proof}
 (a) Suppose that $\TT$ is compact. As $\TT = \Join_{X \in \TT} \T(X)$ holds, there exists $X_1, X_2, \dots, X_n \in \TT$ such that $\TT = \Join_{i = 1}^n \T(X_i) = \T(X_1 \oplus X_2 \oplus \cdots \oplus X_n)$.

 Conversely, suppose that $\TT = \T(X)$ for some $X \in \mod A$ and let $\S \subseteq \tors A$ such that $\TT \subseteq \Join \S$. We know that $\Join \S = \Filt \left(\bigcup \S\right)$. So $X$ is filtered by elements of $\bigcup \S$. On the other hand, $X$ is finite-dimensional, so $X$ is filtered by finitely many elements of $\bigcup \S$. Therefore, there exists $\{\TT_1, \TT_2, \dots, \TT_n\} \subseteq \S$ such that $X \in \Join_{i = 1}^n \TT_i$, so $\TT = \T(X) \subseteq \Join_{i = 1}^n \TT_i$. We proved that $\T(X)$ is compact.

 (b) The argument of (a) works analogously for torsion-free classes. Then, using the anti-isomorphism ${}^\perp - : \torf A \to \tors A$ leads us to the conclusion.
\end{proof}

We now prove Theorem \ref{semidistributiveandbialg}.

\begin{proof}[Proof of Theorem \ref{semidistributiveandbialg}]
 (a) By duality, we prove only the first condition. Let $\TT \in \tors A$ and $\S \subseteq \tors A$ satisfying $\TT \cap \UU = \TT \cap \VV$ for all $\UU, \VV \in \S$. It is enough to prove that $\TT \cap \Join \S \subseteq \TT \cap \UU$ for a fixed $\UU \in \S$.

 Recall that $\Join \S=\Filt(\bigcup \S)$ holds. Let $X \in\TT \cap \Join \S$. We prove by induction on $\dim X$ that $X \in \TT \cap \UU$. If $X = 0$, it is clear. Otherwise, there exist $\VV \in \S$ and a short exact sequence $0 \to V \to X \to Y \to 0$ with $0 \neq V \in \VV$
and $Y \in \Join \S$. As $\TT$ is a torsion class, $Y \in \TT$, hence by the induction hypothesis, $Y \in \TT \cap \UU = \TT \cap \VV$. As $\VV$ is a torsion class, $X \in \VV$. So $X \in \TT \cap \VV = \TT \cap \UU$.

 (b) For any $\TT \in \tors A$, $\TT = \Join_{X \in \TT} \T(X)$ and $\T(X)$ is compact by Proposition \ref{compacttors}(a), so $\tors A$ is algebraic. Dually, $\TT = \Meet_{X \in \TT^\perp} ({}^\perp X)$ is a meet of co-compact torsion classes by Proposition \ref{compacttors}(b), so $\tors A$ is co-algebraic.
By Theorem \ref{arrsep}, $\tors A$ is weakly atomic. 
\end{proof}

\subsection{Brick labelling} \label{bricklabel2}
 An important ingredient of this paper, which permits an understanding of the forcing preorder as well as an understanding of wide subcategories is the notion of \emph{brick labelling} of $\Hasse(\tors A)$. Note that we do not assume that $A$ is $\tau$-tilting finite in this section. Several of the results we give are generalizations of results that are already known in the $\tau$-tilting finite case.

 Recall that a \emph{brick} is an $A$-module whose endomorphism algebra is a division algebra.
 When $\UU \subseteq \TT$, we denote $\brick [\UU, \TT]$ the set of isomorphism classes of bricks in $\TT \cap \UU^\perp$.
This notation will be justified by Theorem \ref{dtcd}.

\begin{theorem} \label{difftors2}
 Let $\UU \subseteq \TT$ be two torsion classes in $\mod A$. The following hold:
 \begin{enumerate}[\rm(a)]
  \item We have $\TT = \UU$ if and only if $\brick [\UU, \TT] = \emptyset$.
  \item There is an arrow $q: \TT \to \UU$ in $\Hasse(\tors A)$ if and only if $\brick [\UU, \TT]$ contains exactly one element $S_q$. Moreover, $\TT \cap \UU^\perp = \Filt S_q$.
  \item There is a bijection $\cjirr(\tors A) \to \brick A$ that associates to $\TT'$ the brick $S_q$ for the only arrow $q$ starting at $\TT'$.
  \item There is a bijection $\cmirr(\tors A) \to \brick A$ that associates to $\UU'$ the brick $S_q$ for the only arrow $q$ ending at $\UU'$.
 \end{enumerate}
\end{theorem}

We will prove Theorem \ref{difftors2} at the end of this subsection. The bijections of Theorem \ref{difftors2}(c),(d) have also been established independently using different methods in \cite{BCZ}.

We will need a more general version of Theorem \ref{difftors2}(c),(d).

\begin{theorem} \label{dtcd}
 Let $\UU \subseteq \TT$ be torsion classes in $\mod A$. The following hold:
 \begin{enumerate}[\rm(a)]
  \item There is a bijection $\cjirr[\UU, \TT] \to \brick[\UU, \TT]$ mapping $\TT' \in \cjirr[\UU, \TT]$ to $S_q$ where $q:\TT' \to \UU'$ is the unique arrow of $\Hasse[\UU, \TT]$ starting at $\TT'$. Moreover, $\TT' = \UU \join \T(S_q)$ and $\UU' = \TT' \meet {}^\perp S_q$.
  \item There is a bijection $\cmirr[\UU, \TT] \to \brick[\UU, \TT]$ mapping $\UU' \in \cmirr[\UU, \TT]$ to $S_q$ where $q:\TT' \to \UU'$ is the unique arrow of $\Hasse[\UU, \TT]$ ending at $\UU'$. Moreover, $\UU' = \TT \meet {}^\perp S_q$ and $\TT' = \UU' \join \T(S_q)$.
 \end{enumerate}
\end{theorem}

We now define the \emph{brick labelling}.

\begin{definition} \label{definebrick2}
 Let $q: \TT \to \UU$ be an arrow of $\Hasse(\tors A)$. The \emph{label of $q$} is the unique brick $S_q$ in $\TT \cap \UU^\perp$, given in Theorem \ref{difftors2}(b).
\end{definition}

We give an example of brick labelling.

\begin{example} \label{kronecker2} 
 Let $k$ be an algebraically closed field and $Q$ the Kronecker quiver \[\xymatrix{2\ar@/^/[r]^{a}|{}="A" \ar@/_/[r]_{b}|{}="B"  &1 }  \] and $A=kQ$. For $(\lambda,\mu) \in k^2 \setminus \{(0,0)\}$, we consider the following brick in $\mod A$:
 \[
  S_{(\lambda:\mu)} = \left[\boxinminipage{\xymatrix{ 2 \ar@/^/[d]^-\mu \ar@/_/[d]_-\lambda \\ 1  }}\right]
 \]
 whose isomorphism class only depends of $(\lambda:\mu) \in \mathbb{P}^1(k)$. Then, for $\S\subseteq \mathbb{P}^1(k)$ non-empty, we define the torsion class $\TT(\S) = \Filt (\S \cup \{2\})$. We also define the torsion class $\TT(\emptyset) = \smash{\bigcap_{\S \neq \emptyset}} \TT(\S)$. Then $\TT: \smash{2^{\mathbb{P}^1(k)} }\to \tors A$ is an injective morphism of complete lattices from the power set of $\mathbb{P}^1(k)$ to $\tors A$. We denote by $\mathcal{R}$ its image. Then, using classical knowledge about the Auslander-Reiten quiver of $A$, the labelled Hasse quiver of $\tors A$ is given by
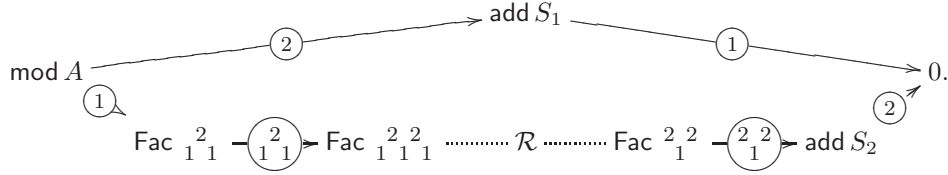
\begin{center}$
\xymatrix@C=.45cm@R=.3cm{&&&&\add S_1 \ar[drrrr]|{\cirl{\Sa}}\\
 \mod A \ar[dr]|-{\cirl{\Sa}} \ar[urrrr]|{\cirl{\Sb}} &&&&&&&& 0. \\
&\Fac \KPu \ar[rr]|-{\cirl{\KPu}}&& \Fac \KPd \ar@{.}[r] & \mathcal{R} \ar@{.}[r] & \Fac \KIu \ar[rr]|-{\cirl{\KIu}}&& \add S_2 \ar[ur]|{\cirl{\Sb}}}
$\end{center}
 Any arrow of $\Hasse \mathcal{R}$ has the form $q: \TT(\S) \to \TT(\S')$ for some $\S, \S' \subseteq \mathbb{P}^1(k)$ satisfying $\S \setminus \S' = \{(\lambda:\mu)\}$ for some $(\lambda: \mu) \in \mathbb{P}^1(k)$. The brick that labels 
 this arrow is $S_q = S_{(\lambda:\mu)}$. To be more explicit, if $P$ is an indecomposable preprojective module distinct from $S_1$, then $\Fac P$ contains all indecomposable modules except the ones that are to its left in the Auslander-Reiten quiver, if $I$ is indecomposable preinjective, then $\Fac I$ contains $I$ and indecomposable modules that are to its right in the Auslander-Reiten quiver. Finally, $\TT(\S)$ contains no preprojective modules, all preinjective modules and the tubes that are indexed by elements of $\S$.
\end{example}

In the rest of this subsection, we prove Theorem \ref{difftors2}.
We start by giving a relative version of \cite[Lemma 4.4]{DIJ}.

\begin{lemma}
  \label{simpleness} Let $\UU \in \tors A$ and $S$ be a brick in $\UU^\perp$. Then, the following statements hold.
  \begin{enumerate}[\rm(a)]
  \item Every morphism $f:X\to S$ in $\T(\UU, S)$ is either zero or surjective.
  \item If a brick $S'$ in $\UU^\perp$ satisfies $\T(\UU, S) =\T(\UU, S')$, then $S\simeq S'$.
  \end{enumerate}
\end{lemma}
\begin{proof}
  (a) We show that $f\neq0$ implies that $f$ is surjective. Since $X\in\T(\UU, S)=\Filt(\UU, \Fac S)$, there exists a filtration $0=X_0\subset X_1\subset\cdots\subset X_t=X$ satisfying $X_{i+1}/X_i\in\UU$ or $X_{i+1}/X_i\in\Fac S$ for any $i$.  We can assume $f(X_1) \neq 0$ by
  taking a maximal number $i$ satisfying $f(X_i)=0$ and replacing $X$ by $X/X_i$. Since $S \in \UU^\perp$, we get $X_1 \in \Fac S$, so there exists an epimorphism $g:S^{\oplus n}\to X_1$. Since $fg: S^{\oplus n} \to S$ is non-zero and $S$ is a brick, $fg$ must be a split epimorphism. Thus $f$ is surjective.

  (b) Since $S$ belongs to $\T(\UU, S') \cap \UU^\perp$, there exists a non-zero morphism ${f:S'\to S}$. This is surjective by (a), and therefore $\dim_k S' \geq \dim_k S$.  The same argument shows the opposite inequality, and therefore $f$ is an isomorphism.
\end{proof}

Next, we give the following elementary lemma:
\begin{lemma} \label{modgivesbrick}
 Let $\TT, \UU \in \tors A$ and $X \in \TT \cap \UU^\perp$ be non-zero. Then there exists $S \in \brick[\UU, \TT]$ that is the image of an endomorphism of $X$.
\end{lemma}

\begin{proof}
 We argue by induction on the dimension of $X$. If $X$ is a brick, it is immediate. Otherwise, $X$ admits a non-zero radical endomorphism $f = \iota \pi$ where $\pi$ is surjective and $\iota$ is injective. Then, $\Image f \in \Fac X \cap \Sub X \subseteq \TT \cap \UU^\perp$ and $0 < \dim \Image f < \dim X$, so by induction hypothesis, there is $S \in \brick[\UU, \TT]$ that is the image of an endomorphism $g$ of $\Image f$, hence of $\iota g \pi$.
\end{proof}

We deduce the following.

\begin{lemma} \label{gener}
 Let $\UU \subseteq \TT$ be two torsion classes. We have $\TT = \Filt (\UU \cup \brick [\UU, \TT])$.
\end{lemma}

\begin{proof}
 Let $X \in \TT$. We argue by induction on $\dim X$. If $X = 0$, the result is immediate. As $(\UU, \UU^\perp)$ is a torsion pair, there exists a short exact sequence
 \[0 \to U \to X \to U' \to 0\]
 with $U \in \UU$ and $U' \in \UU^\perp$. It suffices to prove $U' \in  \Filt (\UU \cup \brick [\UU, \TT])$. Suppose that $U' \neq 0$. We have $U' \in \TT \cap \UU^\perp$, so by Lemma \ref{modgivesbrick}, there exists a short exact sequence
 \[0 \to S \to U' \to Y \to 0\]
 with $S \in \brick [\UU, \TT]$. As $U' \in \TT$, we also have $Y \in \TT$ so by induction hypothesis, $Y \in \Filt (\UU \cup \brick [\UU, \TT])$, hence $U' \in  \Filt (\UU \cup \brick [\UU, \TT])$.
\end{proof}

We also deduce:

\begin{lemma} \label{ft2}
 Let $\UU \subseteq \TT$ be two torsion classes of $\mod A$. Then we have $\TT \cap \UU^\perp = \Filt \brick [\UU, \TT]$.
\end{lemma}

\begin{proof}
 The inclusion $\supseteq$ is immediate, hence we prove the other one. Let $X \in \TT \cap \UU^\perp$ be indecomposable. We argue by induction on $\dim X$. If $X$ is a brick, the result is immediate. Otherwise, using Lemma \ref{modgivesbrick}, there exists a brick $S \in \brick [\UU, \TT]$ that is a submodule and a quotient of $X$. Consider a short exact sequence
 \[0 \to S \to X \to Y \to 0. \]
 As $S$ is a brick, the short exact sequence does not split, and, as $\Hom_A(X, S) \neq 0$, we deduce $\Hom_A(Y, S) \neq 0$ so, as $S \in \UU^\perp$, $Y \notin \UU$.

 We fix a short exact sequence $0 \to U \to Y \to U' \to 0$ with $U \in \UU$ and $U' \in \UU^\perp$ and consider the following Cartesian diagram:
 \[\xymatrix{
  S \ar@{^{(}->}[r] \ar@{=}[d] & X' \ar@{^{(}->}[d] \ar@{->>}[r] & U \ar@{^{(}->}[d] \\
  S \ar@{^{(}->}[r] & X \ar@{->>}[d] \ar@{->>}[r] & Y \ar@{->>}[d] \\
  & U' \ar@{=}[r] & U'
 }\]
 As $X \in \UU^\perp$, we get $X' \in \UU^\perp$. As $S, U \in \TT$, we get $X' \in \TT$. We have clearly $U' \in \TT \cap \UU^\perp$. Moreover, as $S \subseteq X'$, we have $X' \neq 0$. As $Y \notin \UU$, $U' \neq 0$, so the dimension of each indecomposable summand of $X'$ and $U'$ is smaller than $\dim X$. This allows us to conclude by the induction hypothesis. \end{proof}

Then, we prove Theorem \ref{dtcd}.

\begin{proof}[Proof of Theorem \ref{dtcd}]
 (a) Let $\TT' = \UU \join \T(S)$. As $S \in \TT$ and $\UU \subseteq \TT$, it is immediate that $\UU \subseteq \TT' \subseteq \TT$. Let also $\UU' = \TT' \cap {}^\perp S$. As $\UU \subseteq {}^\perp S$ and $\TT' \not\subseteq {}^\perp S$, we have $\UU \subseteq \UU' \subsetneq \TT'$.

 If $\VV \subsetneq \TT'$, consider $X \in \VV$ and $f: X \to S$. As $X \in \TT' = \T(\UU, S)$, if $f \neq 0$ we get that $f$ is surjective by Lemma \ref{simpleness}(a), hence $S \in \T(X) \subseteq \VV$ so $\VV = \TT'$, which is a contradiction. So $\VV \subseteq {}^\perp S$, hence $\VV \subseteq \UU'$. We proved that $\TT'$ is completely join-irreducible in $[\UU, \TT]$ and that there is an arrow $\TT' \to \UU'$ in $\Hasse [\UU, \TT]$.

 Let us prove that the only element of $\brick [\UU', \TT']$ is $S$. It is clear that $S \in \brick [\UU', \TT']$. Consider $S' \in \brick [\UU', \TT']$. We have $\T(\UU, S') \subseteq \TT'$ and $\T(\UU, S') \not\subseteq \UU'$, so $\T(\UU, S') = \TT'$. By Lemma \ref{simpleness}(b), this implies that $S \cong S'$.

 Finally, we need to prove the uniqueness of the completely join-irreducible torsion class. Consider an arrow $\TT'' \to \UU''$ of $\Hasse [\UU, \TT]$ such that $\brick [\UU'', \TT''] = \{S\}$ and  $\TT''$ is completely join-irreducible. As $\UU \subseteq \TT''$ and $S \in \TT''$, we have $\TT' \subseteq \TT''$. As $S \in \TT'$ and $S \notin \UU''$, we have $\TT' \not\subseteq \UU''$, so $\TT'' = \TT' \join \UU''$ and, as $\TT''$ is join-irreducible, we obtain $\TT' = \TT''$. By uniqueness of the arrow starting at $\TT'$, we also have $\UU' = \UU''$.

 (b) It is dual to (a).
\end{proof}

Finally, we prove Theorem \ref{difftors2}.

\begin{proof}[Proof of Theorem \ref{difftors2}]
 (a) The case $\TT = \UU$ is trivial. Suppose that $\TT \neq \UU$ and let $X \in \TT \cap \UU^\perp$ be non-zero. By Lemma \ref{modgivesbrick}, $X$ admits a quotient that is in $\brick [\UU, \TT]$.

 (b) First of all, if $\# \brick [\UU, \TT] = 1$ and $\UU \subseteq \VV \subseteq \TT$, by (a), we have $\UU = \VV$ or $\TT = \VV$ hence there is an arrow $\TT \to \UU$ in $\Hasse(\tors A)$. Conversely, if $\TT \to \UU$ is an arrow of $\Hasse(\tors A)$, by (a) there exists $S \in \brick [\UU, \TT]$. Then, by Theorem \ref{dtcd}(a), there exists an arrow $\TT' \to \UU'$ in $\Hasse [\UU, \TT]$ such that $\brick[\UU', \TT'] = \{S\}$. So $\brick [\UU, \TT] = \{S\}$. Finally, by Lemma \ref{ft2}, $\TT \cap \UU^\perp = \Filt S_q$.

 (c) and (d) It is Theorem \ref{dtcd} for $\UU = 0$ and $\TT = \mod A$.
\end{proof}

\subsection{Complete congruence uniformity of the lattice of torsion classes} 
The main result of this subsection is the following.

\begin{theorem} \label{theorem uniform2}
Let $A$ be a finite-dimensional algebra.
\begin{enumerate}[\rm(a)]
\item Two arrows of $\Hasse(\tors A)$ are labelled by the same brick if and only if they are forcing equivalent.
\item The lattice $\tors A$ is completely congruence uniform.
\item The brick labelling coincides with the join-irreducible labelling and with the meet-irreducible labelling under the bijections of \textup{Theorem \ref{difftors2}(c),(d)}.
\end{enumerate}
\end{theorem}

In particular, we get:
\begin{corollary} \label{fincongunif}
 Let $A$ be a finite-dimensional algebra that is $\tau$-tilting finite. Then $\tors A$ is congruence uniform.
\end{corollary}

Note that the congruence uniformity of $\tors A$ was known only for certain special classes of algebras: preprojective algebras, via the weak order (see Section \ref{pre weak sec}), and certain gentle algebras \cite{PPP}.
By Theorem \ref{theorem uniform2}, the forcing preorder $\forces$ can be considered as a partial order on $\brick A$. In particular, we get the following description of $\Comb(\tors A)$.

\begin{corollary} 
 The complete lattices $\Comb(\tors A)$ and $\Ideals(\brick A)$ are isomorphic, where $\Ideals(\brick A)$ consists of the sets of bricks that are closed under forcing.
\end{corollary}

\begin{proof}
 This is a consequence of Theorem \ref{theorem uniform2}.
\end{proof}

We need some preparation to prove Theorem \ref{theorem uniform2}. We associate to each brick a certain complete congruence. Let $S \in \brick A$. We define the relation $\equiv_S$ in the following way. For $\TT, \UU \in \tors A$, we put $\TT \equiv_S \UU$ if every $X \in (\TT \join \UU) \cap (\TT \meet \UU)^\perp$ admits $S$ as a subfactor.

\begin{proposition} \label{congbrk}
 The relation $\equiv_S$ is a complete congruence.
\end{proposition}

\begin{proof}
 For simplicity, we write $\equiv$ instead of $\equiv_S$. This relation is clearly symmetric and reflexive.

 Let us prove that it is transitive. Suppose that $\TT \equiv \UU$ and $\UU \equiv \VV$. Let $X \in (\TT \join \VV) \cap (\TT \meet \VV)^\perp$ be non-zero. We consider a short exact sequence $0 \to U \to X \to U' \to 0$ with $U \in \UU$ and $U' \in \UU^\perp$. Suppose that $U \neq 0$. We have $U \in (\TT \meet \VV)^\perp = \TT^\perp \join \VV^\perp$ so $U$ admits a non-zero quotient $U''$ that is in $\TT^\perp$ or in $\VV^\perp$. By symmetry, we suppose that $U'' \in \TT^\perp$. So $U'' \in (\TT \join \UU) \cap (\TT \meet \UU)^\perp$, hence it admits $S$ as a subfactor because $\TT \equiv \UU$, so $X$ admits $S$ as a subfactor. If $U' \neq 0$, it admits a submodule $U''$ that is in $\TT$ or $\VV$ and we conclude as before.

 Consider now two families $(\UU_i)_{i \in \II}$ and $(\TT_i)_{i \in \II}$ of torsion classes satisfying $\UU_i \equiv \TT_i$ for all $i \in \II$. Let $\UU = \Join_{i \in \II} \UU_i$ and $\TT = \Join_{i \in \II} \TT_i$. We will prove that $\UU \equiv \TT$. Let $X \in (\TT \join \UU) \cap (\TT \meet \UU)^\perp$ be non-zero. As $X \in \TT \join \UU = \Join_{i \in \II} (\TT_i \join \UU_i)$, there exists a non-zero submodule $X'$ of $X$ and $i_0 \in \II$ such that $X' \in \TT_{i_0} \join \UU_{i_0}$. As $X' \in (\TT \meet \UU)^\perp$ and $\Join_{i \in \II} (\TT_i \meet \UU_i) \subseteq \TT \meet \UU$, we have $X' \in \Meet_{i \in \II} (\TT_i \meet \UU_i)^\perp$ so $X' \in (\TT_{i_0} \meet \UU_{i_0})^\perp$. Hence, as $\TT_{i_0} \equiv \UU_{i_0}$, $X'$ admits $S$ as a factor module, so $X$ does. In the same way, we prove that $\Meet_{i \in \II} \UU_i \equiv \Meet_{i \in \II} \TT_i$ so $\equiv$ is a complete congruence.
\end{proof}

\begin{proposition} \label{feq}
 Two arrows $q$ and $q'$ of $\Hasse(\tors A)$ are forcing equivalent if and only if $S_q \cong S_{q'}$.
\end{proposition}

\begin{proof}
 We denote $q: \TT \to \UU$ and $q': \TT' \to \UU'$.

 First, suppose that $S := S_q \cong S_{q'}$. Let $\equiv$ be a congruence satisfying $\TT \equiv \UU$. We have $S \in \TT \meet \TT'$ and $S \notin \UU'$ so $\UU' \subsetneq (\TT \meet \TT') \join \UU' \subseteq \TT'$. As $q'$ is an arrow, we deduce $(\TT \meet \TT') \join \UU' = \TT'$. We have $\UU \meet \TT' \subseteq {}^\perp S$ and $\UU' \subseteq {}^\perp S$ so $(\UU \meet \TT') \join \UU' \subseteq {}^\perp S$ and therefore $\UU' \subseteq (\UU \meet \TT') \join \UU' \subsetneq \TT'$, so, as before, $(\UU \meet \TT') \join \UU' = \UU'$. As $\UU \equiv \TT$ and $\equiv$ is a lattice congruence, we deduce $\TT' = (\TT \meet \TT') \join \UU' \equiv (\UU \meet \TT') \join \UU' = \UU'$.

 Suppose now that $q$ and $q'$ are forcing equivalent. The congruence $\equiv_{S_q}$ defined above contracts $q$, so it contracts $q'$. Hence, $S_q$ is a subfactor of $S_{q'}$. Conversely, $S_{q'}$ is a subfactor of $S_q$. Then, $S_{q'} \cong S_q$.
\end{proof}

We can finally prove Theorem \ref{theorem uniform2}.

\begin{proof}[Proof of Theorem \ref{theorem uniform2}]
 (a) This is Proposition \ref{feq}.

 (b) By (a) and Theorem \ref{difftors2}(c), the forcing equivalence classes in $\Hasse_1(\tors A)$ correspond bijectively with $\brick A\cong\cjirr(\tors A)$. With this and its dual, we conclude using Proposition \ref{congunifchar}(ii)$\Rightarrow$(i) together with Theorem \ref{semidistributiveandbialg} that $\tors A$ is completely congruence uniform.

(c) This is immediate.
\end{proof}

 We end this subsection by giving the following elementary observation about the forcing order on bricks.

 \begin{proposition} \label{topbottomlatsimpleforce}\
  \begin{enumerate}[\rm (a)]
   \item The arrows incident to $0$ in $\Hasse(\tors A)$ are $\Filt S \to 0$ for each simple $A$-module $S$. The label of $\Filt S \to 0$ is $S$.
   \item The arrows incident to $\mod A$ in $\Hasse(\tors A)$ are $\mod A \to {}^\perp S$ for each simple $A$-module $S$. The label of $\mod A \to {}^\perp S$ is $S$.
   \item The maximal elements for the forcing order on $\brick A$ are simple $A$-modules.
   \item For a simple $A$-module $S$ and a brick $S'$, $S \forces S'$ if and only if $S$ is a subfactor of $S'$. 
  \end{enumerate}
 \end{proposition}

 \begin{proof}
  (a) Any $\TT \in \tors A \setminus \{0\}$ contains $\Filt S$ for a simple module $S$. Indeed, let $X \in \TT$ be non-zero and $S$ be a simple factor module of $X$. Then $S \in \TT$, so $\Filt S \subseteq \TT$. As $\Filt S \cap \Filt S' = 0$ for $S \neq S'$, the result follows. It is immediate that $S$ labels $\Filt S \to 0$.

  (b) By the dual of (a), arrows incident to $0$ in $\torf A$ are $\Filt S \to 0$ so, as $ {}^\perp(-):\torf A\to\tors A $ is an anti-isomorphism, arrows incidents to $\mod A$ in $\tors A$ are $\mod A \to {}^\perp (\Filt S) = {}^\perp S$.

  (d) Let $S$ be a simple $A$-module, $e$ be the corresponding primitive idempotent and $B := A/(e)$. As a very special case of Theorem \ref{eta op}(a), $\pi: \tors A \twoheadrightarrow \tors B$, $\TT \mapsto \TT \cap \mod B$ is a surjective morphism of complete lattices. Moreover, in this case, $\pi$ splits, as any $\UU \in \tors B$ is also a torsion class in $\mod A \supseteq \mod B$, identifying $\tors B$ with a sublattice of $\tors A$.

  By (a), $q_S: \Filt S \to 0$ is an arrow of $\Hasse(\tors A)$. We have, in $\tors A$, $\Filt S \join \mod B = \mod A$, so the lattice congruence $\Theta$ corresponding to $\pi$ is $\con(q_S)$.

  Let $S' \in \brick A$. By Theorem \ref{difftors2}(c), there is an arrow $q: \T(S') \to \UU$ with $S_q \cong S'$ and $\T(S') \in \cjirr(\tors A)$.
  If $S$ is not a subfactor of $S'$, then $S' \in \mod B$, so that $q$ is an arrow of $\Hasse(\tors B)$. So $q$ is not contracted by $\pi$. Therefore $S'$ is not forced by $S$.
  If $S$ is a subfactor of $S'$, $\pi(\T(S')) \neq \T(S')$ so $q$ has to be contracted by $\pi$, hence by $\con(q_S)$. It implies that $S$ forces $S'$.

  (c) As any non-simple brick admits a strict simple subfactor, any maximal brick has to be simple by (d). Moreover, by (d) again, a simple module cannot force another simple module, so all simple modules are maximal.
 \end{proof}

\section{Functorially finite torsion classes}

\subsection{Reminders on \texorpdfstring{$\tau$-tilting}{tau-tilting} theory} 
We recall that a torsion class $\TT$ of $\mod A$ is \newword{functorially finite}
if there exists $M\in\mod A$ such that $\TT=\Fac M$, where $\Fac M$ is the full
subcategory of $\mod A$ consisting of factor modules of finite direct sums of
copies of $M$.
We denote by $\ftors A$ the set of all functorially finite torsion classes in $\mod A$.

If $X \in \mod A$, we denote by $|X|$ the number of non-isomorphic indecomposable direct summands of $X$. We say that $X$ is \emph{basic} if it has no direct summand of the form $Y \oplus Y$ for an indecomposable $A$-module $Y$.

There is a bijection between $\ftors A$ and a certain class of $A$-modules.
Recall that $M\in\mod A$ is \newword{$\tau$-rigid} if $\Hom_A(M,\tau M)=0$ where $\tau$ is the Auslander-Reiten translation.
We say that $M\in\mod A$ is \newword{$\tau$-tilting} if it is $\tau$-rigid and $|M|=|A|$
holds.
We say that $M\in\mod A$ is \newword{support $\tau$-tilting} if there exists an idempotent
$e$ of $A$ such that $M$ is a $\tau$-tilting $(A/(e))$-module.
We denote by $\sttilt A$ the set of isomorphism classes of basic support
$\tau$-tilting $A$-modules, by $\trigid A$ the set of isomorphism classes of
basic $\tau$-rigid $A$-modules, and by $\itrigid A$ the set of isomorphism classes of
indecomposable $\tau$-rigid $A$-modules.
By \cite[Theorem 2.7]{AIR}, we have a surjection
\begin{equation}\label{tau-rigid gives tors}
\Fac:\trigid A\to\ftors A
\end{equation}
given by $M\mapsto\Fac M$, which induces a bijection
\begin{equation}\label{tau-tilting gives tors}
\Fac:\sttilt A\xrightarrow{\sim}\ftors A.
\end{equation}

Sometimes, we use the following characterization of vanishing of $\Hom_A(X, \tau Y)$:
\begin{proposition}[{\cite[Proposition 5.8]{AS}}] \label{characanntau}
  Let $X$ and $Y$ be two $A$-modules. Then $\Hom_A(X, \tau Y) = 0$ if and only if $\Ext^1_A(Y, X') = 0$ for all $X' \in \Fac X$.
\end{proposition}

We also introduce the notion of a \emph{$\tau$-rigid pair}. A $\tau$-rigid pair over $A$ is a pair $(M, P)$ where $M$ is a $\tau$-rigid $A$-module and $P$ is a projective $A$-module satisfying $\Hom_A(P, M) = 0$. We say that $(M, P)$ is \emph{basic} if both $M$ and $P$ are. We denote by $\trigidpair A$ the set of isomorphism classes of basic $\tau$-rigid pairs over $A$ and by $\itrigidpair A$ the subset of $\trigidpair A$ consisting of indecomposable ones (\emph{i.e.} $(M,0)$ with $M$ indecomposable or $(0,P)$ with $P$ indecomposable). We identify $M \in \trigid A$ with $(M, 0) \in \trigidpair A$. We say that a $\tau$-rigid pair $(M,P)$ is \emph{$\tau$-tilting} if, in addition, we have $|M| + |P| = |A|$.
We denote by $\ttiltpair A$ the set of isomorphism classes of basic $\tau$-tilting pairs. We have a bijection $\ttiltpair A \to \sttilt A$ mapping $(M, P)$ to $M$. Finally, for $(M, P) \in \trigidpair A$, we denote by $\ttiltpair_{(M,P)} A$ the set of isomorphism classes of basic $\tau$-tilting pairs over $A$ having $(M,P)$ as a direct summand.

We recall that the order on $\ttiltpair A \cong \sttilt A$ induced by the bijection \eqref{tau-tilting gives tors} is characterized in the following way.
\begin{lemma}[{\cite[Lemma 2.25]{AIR}}] \label{ordertautilt}
 For $(T, P), (U, Q) \in \ttiltpair A$, we have the inequality $(T, P) \ge (U, Q)$ if and only if $\Hom_A(U, \tau T) = 0$ and $\Hom_A(P, U) = 0$.
\end{lemma}
Moreover, $\sttilt A \cong \ttiltpair A$ is endowed with a \emph{mutation}, exchanging two pairs $(T_1, P_1)$ and $(T_2, P_2)$, described in Theorem \ref{muttautilt}. We call $(T, P) \in \trigidpair A$ \emph{almost $\tau$-tilting} if $|T|+|P|=|A|-1$.

\begin{theorem} \label{muttautilt}\
 \begin{enumerate}[\rm(a)]
  \item \cite[Theorem 2.18]{AIR} If $(T, P)$ is an almost $\tau$-tilting pair, then $\ttiltpair_{(T,P)} A$ has exactly two elements $(T_1, P_1)$ and $(T_2, P_2)$.
  \item \cite[Theorem 2.33]{AIR} The Hasse quiver of $\ttiltpair A$ has an arrow linking $(T_1, P_1)$ and $(T_2, P_2)$ of \textup{(a)} and all arrows occur in this way.
 \end{enumerate}
\end{theorem}

Note that a version of Theorem \ref{muttautilt}(a) was proved in \cite[Proposition 5.7]{DF} for $2$-term silting complexes.

Any $\tau$-rigid pair has two canonical completions, as shown below.
\begin{theorem}[{\cite[Theorem 2.10]{AIR}}] \label{bongartz}
 If $(X, Q) \in \trigidpair A$, then the subposet $\ttiltpair_{(X,Q)} A$ of $\ttiltpair A$ is an interval $[(X^-, Q^-), (X^+, Q)]$. Moreover, they are characterized by the identities $\Fac X^+ = {}^\perp(\tau X) \cap Q^\perp$ and $\Fac X^- = \Fac X$.
\end{theorem}

In Theorem \ref{bongartz}, we call $(X^-, Q^-)$ the \emph{co-Bongartz completion of $(X, Q)$} and $(X^+, Q)$ the \emph{Bongartz completion of $(X, Q)$}. Additionally, we observe the following.

\begin{lemma} \label{bongcompmin1}
 Let $(T, P) \in \ttiltpair A$ and $X$ be the minimal direct summand of $T$ such that $\Fac T = \Fac X$. Then $(T, P)$ is the Bongartz completion of $(T/X, P)$.
\end{lemma}

\begin{proof}
  First of all, it is immediate that $(T, P) \leq ((T/X)^+, P)$. By \cite[Theorem 1.3]{DIJ}, if $(T,P)$ was not the Bongartz completion of $(T/X, P)$, there would be an arrow $(T',P') \to(T,P)$ in $\Hasse (\ttiltpair A)$ such that $T/X \in \add T'$ and $P \in \add P'$
 So $P' \cong P$ and we can decompose $T' \cong M \oplus U$ and $T \cong M \oplus V$ with $U$ and $V$ indecomposable. We have $\Fac T = \Fac M$. So, as $X$ is the minimal direct summand of $T$ such that $\Fac T = \Fac X$, $\add X$ does not contain $V$. So $V$ is a direct summand of $T/X$, hence $M \oplus U$ does not have $T/X$ as a direct summand. It is a contradiction.
\end{proof}

We recall that $A$ is \newword{$\tau$-tilting finite} if there are only finitely many
indecomposable $\tau$-rigid $A$-modules. We get the following straightforward corollary of Theorem \ref{muttautilt}.

\begin{corollary} \label{Hasse-regular}
Let $A$ be a finite-dimensional $k$-algebra.
Then $\ftors A$ is Hasse-regular.
In particular, if $A$ is $\tau$-tilting finite, then $\tors A$ is Hasse-regular.
\end{corollary}

The following characterizations are shown in
\cite{DIJ} and \cite{IRTT}:
\begin{theorem}[{\cite{DIJ,IRTT}}]\label{tors is complete lattice}
The following conditions are equivalent.
\begin{enumerate}[\rm(i)]
\item $A$ is $\tau$-tilting finite.
\item $\ftors A$ is a finite set.
\item $\ftors A$ is a complete lattice.
\item $\ftors A=\tors A$.
\end{enumerate}
\end{theorem}

On the other hand, it is a much more subtle condition for $A$ that $\ftors A$ is a lattice.
It is shown in \cite[Theorem 0.3]{IRTT} that for a path algebra $kQ$ of a connected
acyclic quiver $Q$,
$\ftors(kQ)$ is a lattice if and only if $Q$
is either a Dynkin quiver or has at most $2$ vertices.

We have the following description of join-irreducible elements in $\tors A$.

\begin{theorem}[{\cite[Theorem 2.7 and following discussion]{IRRT}}]\label{join-irreducible in sttilt}
Let $A$ be a finite-dimensional $k$-algebra.
\begin{enumerate}[\rm(a)]
 \item If $A$ is $\tau$-tilting finite, then the map $M\mapsto\Fac M$ of \eqref{tau-rigid gives tors} restricts to a bijection
\[\Fac:\itrigid A\xrightarrow{\sim}\jirr(\tors A).\]
 \item More generally, the map $M\mapsto\Fac M$ restricts to a bijection
\[\Fac:\itrigid A\xrightarrow{\sim} \ftors A \cap \cjirr(\tors A).\]
\end{enumerate}
\end{theorem}

We finish this subsection by interpretations of the brick labelling in terms of $\tau$-tilting modules. It has been defined by Asai for functorially finite torsion classes. By \cite[Theorem 1.3]{DIJ}, $\Hasse(\ftors A)$ is a full subquiver of $\Hasse(\tors A)$. Then the brick labelling of arrows of $\Hasse(\ftors A)$ has the following description.

\begin{proposition}[\cite{A}] \label{theorem label2}
 Let $q: \TT \to \UU$ be an arrow of $\Hasse(\ftors A)$. Then
  \[S_q \cong \frac{X}{\Rad_A(T,X)\cdot T} \]
 where the basic support $\tau$-tilting modules $T$ and $U$ corresponding to $\TT$ and $\UU$ via the bijection $\Fac$ are decomposed as $T=X\oplus M$ and $U=Y\oplus M$ for $X$ an 
 indecomposable $A$-module and $Y$ an $A$-module which is indecomposable or zero.
\end{proposition}

We recall also bijections arising when $A$ is $\tau$-tilting finite. A set $\{S_i\}_{i\in I}$ of bricks (or its direct sum) is called a \emph{semibrick} if $\Hom_A(S_i,S_j)=0$ for any $i\neq j$. We denote by $\brick A$ the set of isomorphism classes of bricks of $A$, and by $\sbrick A$ the set of isomorphism classes of semibricks.

\begin{proposition}\label{brick-rigid}
Let $A$ be a finite-dimensional $k$-algebra.
\begin{enumerate}[\rm(a)]
\item \cite{DIJ} There is an injection $\itrigid A\to \brick A$ sending $M$ to $M/\rad_{\End_A(M)}M$.
\item \cite{A} There is an injection $\sttilt A\to \sbrick A$ sending $M$ to $M/\rad_{\End_A(M)}M$.
\end{enumerate}
 Moreover, if $A$ is $\tau$-tilting finite, these maps are bijections.
\end{proposition}

Notice that Theorem \ref{difftors2}(c) given before extends Proposition \ref{brick-rigid}(a), using Proposition \ref{theorem label2} and Theorem~\ref{join-irreducible in sttilt}.

\subsection{Wide subcategories}  \label{wide2}

Let $A$ be a finite-dimensional $k$-algebra. This subsection deals with combinatorial interpretation of wide subcategories in $\tors A$ in terms of bricks. Recall that a full subcategory $\WW \subseteq \mod A$ is \emph{wide} if it is stable by extension, kernel and cokernel. In particular it is an abelian category. We denote by $\wide A$ the set of wide subcategories of $\mod A$.

Before going further, we recall the following relation between semibricks and wide subcategories of $\mod A$. 
 \begin{proposition}[\cite{ringelk}] \label{bijsbw}
  There is a bijection
  \[\Filt: \sbrick A \to \wide A\]
  mapping a semibrick $\S$ to the full subcategory of $A$-modules that are filtered by bricks of $\S$. The inverse bijection associates to $\WW \in \wide A$ the set of its simple objects.
 \end{proposition}

For $(N, Q) \in \trigidpair A$, as in Theorem \ref{bongartz}, we denote by $(N^+, Q^+)$ and $(N^-, Q)$ the Bongartz and co-Bongartz completions of $(N, Q)$. Then, we consider the torsion classes
\[\UU(N, Q) := \Fac N^- = \Fac N \quad \text{and} \quad \quad \TT(N, Q) := \Fac N^+ = {}^{\perp} (\tau N) \cap Q^\perp ,\]
and the full subcategory
\[\WW{(N, Q)} :=  \TT(N, Q) \cap  \UU(N, Q)^\perp = {}^{\perp} (\tau N) \cap Q^\perp \cap N^\perp.\]

Our starting point is the following, which is mostly in the article of Jasso about $\tau$-tilting reduction \cite{J}.
\begin{theorem} \label{ttwideaa}
 Let $A$ be a finite-dimensional algebra and $(N, Q) \in \trigidpair A$.
 \begin{enumerate}[\rm (a)]
  \item The subcategory $\WW{(N, Q)}$ is a wide subcategory of $\mod A$.
  \item Let $C_{N, Q} := \End_A(N^+)/[N]$ where $[N]$ is the ideal consisting of endomorphisms that factor through $\add N$. Then there is an equivalence of categories \[F_{N,Q} : \WW(N,Q) \xrightarrow{\sim} \mod C_{N, Q},\] mapping $X$ to $\Hom_A(N^+, X)$.
  \item There is an isomorphism of lattices \[[\UU(N, Q), \TT(N, Q)]  \xrightarrow[\sim]{\psi_{N, Q}} \tors C_{N,Q}\] with $\psi_{N, Q}(\VV) = F_{N, Q}(\VV \cap \WW(N,Q))$.
  \item If $A$ is $\tau$-tilting finite, then there is an isomorphism of lattices \[\ttiltpair_{(N,Q)} A \xrightarrow[\sim]{\Fac} [\UU(N, Q), \TT(N, Q)].\]
 \end{enumerate}
\end{theorem}

\begin{proof}
 (a) As $\TT(N, Q)$ and $\UU(N, Q)^\perp$ are stable by extensions, so is $\WW = \WW(N, Q)$. Let $f: X \to Y$ be a morphism in $\WW$. Let us prove that $\Kernel f \in \WW$. We have $\Kernel f \in \Sub X \subseteq Q^\perp \cap N^\perp$, so we need to prove that $\Kernel f \in {}^\perp(\tau N)$. By applying $\Hom_A(-, \tau N)$ to the short exact sequence $0 \to \Kernel f \to X \to \Image f \to 0$, we get an exact sequence
  \[0 = \Hom_A(X, \tau N) \to \Hom_A(\Kernel f, \tau N) \to \Ext^1_A(\Image f, \tau N).\]
 By Auslander-Reiten duality, \[\Ext^1_A(\Image f, \tau N) = \underline{\Hom}_A(N, \Image f) \subseteq \underline{\Hom}_A(N, Y) = 0,\] so we get $\Kernel f \in \WW$. Dually, we have $\Cokernel f \in \WW$.

 (b) Jasso proved the result when $Q = 0$ in \cite[Theorem 1.4]{J}. For the general case, denote $A' := A/(e)$ where $e$ is the idempotent corresponding to $Q$. Then as a full subcategory of $\mod A$, we have $\mod A' = Q^\perp$. Therefore, the result of Jasso for $(N, 0) \in \trigidpair A'$ implies the general result.

 (c) As in the proof of (b), this is a consequence of \cite[Theorem 1.5]{J} which establishes the bijection when $Q = 0$.

 (d) This is a consequence of \eqref{tau-tilting gives tors} and Theorem \ref{bongartz}.
\end{proof}

We deduce from Theorem \ref{ttwideaa} and \eqref{tau-tilting gives tors} the compatibility of the brick labelling with the $\tau$-tilting reduction which is described in Theorem \ref{ttwideaa}(b).
We keep the notation of Theorem \ref{ttwideaa}.

\begin{proposition} \label{bcomptred}
 For $(N, Q) \in \trigidpair A$, consider an arrow $q: \TT \to \UU$ in $\Hasse[\UU(N, Q), \TT(N, Q)]$ and the corresponding arrow $\bar q: \psi_{N, Q}(\TT) \to \psi_{N, Q}(\UU)$ in $\Hasse(\tors C_{N, Q})$. Then we have $S_q \in \WW(N, Q)$ and $S_{\bar q} = F_{N, Q}(S_q)$.
\end{proposition}

\begin{proof}
 By definition, $S_q \in \TT \cap \UU^\perp$. As $\UU(N, Q) \subseteq \UU \subseteq \TT \subseteq \TT(N, Q)$, we get $S_q \in \TT(N, Q) \cap \UU(N, Q)^\perp = \WW(N, Q)$. Also $S_q \in (\TT \cap \WW(N, Q)) \cap (\UU \cap \WW(N, Q))^\perp$, so, as $F_{N, Q}$ is an equivalence of categories, $F_{N, Q}(S_q) \in \psi_{N, Q}(\TT)  \cap \psi_{N, Q}(\UU)^\perp$. Therefore $F_{N, Q}(S_q)$ is the label of $\bar q$.
\end{proof}

\begin{definition} \label{lpolyt}
 A subset of $\tors A$ of the form $[\UU(N, Q), \TT(N, Q)]$ for some basic $\tau$-rigid pair $(N,Q)$ is called a \emph{polytope} or an \emph{$\ell$-polytope} where $\ell := n - |N|-|Q|$.
\end{definition}

\begin{remark}
 Suppose that $A$ is $\tau$-tilting finite. Then $1$-polytopes correspond to arrows of $\Hasse(\tors A)$. Moreover, an $\ell$-polytope is $\ell$-Hasse-regular, using the isomorphism of lattices $\ttiltpair_{(N, Q)} A \cong [\UU(N, Q), \TT(N, Q)]$ and Theorem \ref{muttautilt}. So $2$-polytopes are polygons in the sense of Section \ref{lat prelim}. We prove in Proposition~\ref{torsApole} that the converse holds and $\tors A$ is polygonal.
\end{remark}

We have the following result about bricks in polytopes of $\tors A$.
\begin{theorem} \label{theorem filtorth}
 Let $A$ be a finite-dimensional algebra. Let $(N, Q)$ be a basic $\tau$-rigid pair and $[\UU, \TT] = [\UU(N, Q), \TT(N, Q)]$ be the corresponding polytope of $\tors A$. Let $\WW = \WW(N, Q) \in \wide A$.
 \begin{enumerate}[\rm(a)]
  \item The set $\S$ of  simple objects of $\WW$ is a semibrick of $A$ satisfying $\WW = \Filt \S$.
  \item Bricks in $\WW$ are exactly the labels of arrows of $\Hasse [\UU, \TT] \subseteq \Hasse(\tors A)$.
  \item The semibrick $\S$ consists of labels of arrows incident to $\UU$ in $\Hasse [\UU, \TT]$.
  \item The semibrick $\S$ consists of labels of arrows incident to $\TT$ in $\Hasse [\UU, \TT]$.
 \end{enumerate}
 We denote $\S$ by $\simp[\UU,\TT]$.
\end{theorem}

\begin{proof}
 First, we consider the case where $(N, Q) = (0, 0)$, hence $\UU = 0$ and $\TT = \mod A$. In this case, $\S$ consists of simple $A$-modules, and (a) is immediate as $A$ is a finite-dimensional algebra and $\WW = \mod A$. By Theorem \ref{difftors2}(c), all bricks appear as labels of arrows in $\Hasse(\tors A)$, so (b) holds.
 Proposition \ref{topbottomlatsimpleforce} implies (c) and (d).

 For a general $(N, Q)$, let $\S$ be the set of simple objects in the abelian category $\WW$. As each object of $\WW$ has finite length, we have $\Filt \S = \WW$. Moreover, Proposition \ref{bcomptred} tells us that the isomorphism
  \[\psi_{N, Q}: [\UU, \TT] \cong \tors C_{N, Q}\]
 is compatible with the brick labelling, via the equivalence $F_{N, Q}: \WW \to \mod C_{N, Q}$. Then, the conclusion for $(N, Q)$ follows the results for $(0,0) \in \trigidpair C_{N, Q}$.
\end{proof}

We give the following description of semibricks in terms of arrows in $\Hasse(\tors A)$.
 \begin{proposition} \label{bongcompmin2}
  Let $A$ a finite-dimensional algebra. Let $(T, P) \in \ttiltpair A$ and consider the smallest direct summand $X$ of $T$ such that $\Fac X = \Fac T$.
  \begin{enumerate}[\rm (a)]
   \item For an indecomposable direct summand $(N, Q)$ of $(T, P)$, the mutation of $(T, P)$ at $(N, Q)$ is smaller than $(T, P)$ if and only if $N$ is a direct summand of $X$ and $Q = 0$.
   \item We have $T/\rad_{\End_A(T)}(T) = X/\rad_{\End_A(X)}(X)$. Moreover, $T/\rad_{\End_A(T)}(T)$ is the direct sum of the labels of arrows of $\Hasse(\tors A)$ starting at $\Fac T$.
  \end{enumerate}
 \end{proposition}

 \begin{proof}
  (a) By Lemma \ref{bongcompmin1}, $(T, P)$ is the Bongartz completion of $(T/X, P)$. In particular, if $N \in \add X$ and $Q = 0$, the mutation of $(T, P)$ at $(N, Q)$ contains $(T/X, P)$ as a direct summand, hence is smaller than $(T, P)$. If $N \notin \add X$ or $Q \neq 0$, then $\Fac T/N \supseteq \Fac X = \Fac T$, hence the mutation of $(T, P)$ at $(N, Q)$ is bigger than $(T, P)$.

  (b) We have
  \[\frac{T}{\rad_{\End_A(T)}(T)} = \frac{X}{\Rad_A(T, X) \cdot T} \oplus \frac{T/X}{\Rad_A(T, T/X) \cdot T}.\]
  As $\Fac T = \Fac X$ and $\add X \cap \add(T/X) = 0$, there is a radical surjective map $\pi$ from $X^\ell$ to $T/X$ for some integer $\ell$, hence $\Rad_A(T, T/X) \cdot T \supseteq \Image \pi = T/X$, so the second term vanishes:
   \[\frac{T}{\rad_{\End_A(T)}(T)} = \frac{X}{\Rad_A(T, X) \cdot T}.\]
   This is the direct sum of the labels of arrows starting at $(T, P)$, by (a) and Proposition \ref{theorem label2}. We have proved the second part of the claim.

  We have \[\Rad_A(T, X) \cdot T = \Rad_A(X, X) \cdot X + \Rad_A(T/X, X) \cdot (T/X).\] For any $f: T/X \to X$, the image of $f$ coincides with the image of $f\pi: X^\ell \to X$ which is radical, as $\pi$ is. So $\Rad_A(T/X, X) \cdot (T/X) \subseteq \Rad_A(X, X) \cdot X$, and $\Rad_A(T, X) \cdot T = \Rad_A(X, X) \cdot X$, hence
  \begin{align*} \frac{T}{\rad_{\End_A(T)}(T)} &= \frac{X}{\Rad_A(X, X) \cdot X} = \frac{X}{\rad_{\End_A(X)}(X)}. \qedhere \end{align*}
 \end{proof}

We deduce the following bijection between $\ttiltpair A$ and $\wide A$ when $A$ is $\tau$-tilting finite.

\begin{theorem} \label{ttwideb}
 Let $A$ be a finite-dimensional algebra that is $\tau$-tilting finite. Then there is a bijection
  \[\ttiltpair A \xto{\sim} \wide A\]
mapping a pair $(T, P)$ to $\WW{(T/X, P)}$
where $X$ is the minimal summand of $T$ satisfying $\Fac X = \Fac T$.
\end{theorem}

\begin{proof}
 By Propositions \ref{bijsbw} and \ref{brick-rigid}(b), there are bijections
 \[\ttiltpair A \to \sbrick A \xto{\Filt} \wide A\]
 where the first map maps $(T, P)$ to $T/\rad_{\End_A(T)}(T)$. So it suffices to prove $\Filt L = \WW(T/X, P)$ where $L := T/\rad_{\End_A(T)}(T)$.

 By Lemma \ref{bongcompmin1}, $(T, P)$ is the Bongartz completion of $(T/X, P)$ so the maximum of the polytope $I := \ttiltpair_{(T/X, P)} A$. By Lemma \ref{bongcompmin2}(a), all arrows starting at $(T, P)$ in $\Hasse(\ttiltpair A)$ are in $\Hasse I$, and by Lemma \ref{bongcompmin2}(b), they are labelled by the indecomposable direct summands of $L$. So, by Theorem~\ref{theorem filtorth}(a)(d), the indecomposable direct summands of $L$ are the simple objects of $\WW(T/X, P)$ . Therefore, $\Filt L = \WW(T/X, P)$.
\end{proof}

We give more details about Theorem \ref{theorem filtorth}(c)(d):

\begin{proposition} \label{proposition polygons}
 Let $[\UU, \TT]$ be an $\ell$-polytope in $\tors  A$. Then there exist indexings
  \begin{itemize}
   \item $\alpha_i: \TT \to \TT_i$, $1 \leq i \leq \ell$ of arrows pointing from $\TT$ in $[\UU, \TT]$,
   \item $\beta_i: \UU_i \to \UU$, $1 \leq i \leq \ell$ of arrows pointing toward $\UU$ in $[\UU, \TT]$,
  \end{itemize}
 such that the following hold:
 \begin{enumerate}[\rm (a)]
  \item We have $\UU = \Meet_{i = 1}^\ell \TT_i$;
  \item We have $\TT = \Join_{i = 1}^\ell \UU_i$;
  \item For $i, j \in \{1, \dots, \ell\}$, $\TT_i  \not\supseteq \UU_j$ if and only if $i = j$;
  \item For any $i \in \{1, \dots, \ell\}$, the same brick labels $\alpha_i$ and $\beta_i$.
 \end{enumerate}
\end{proposition}

\begin{proof}
 As in the proof of Theorem \ref{theorem filtorth}, we only have to consider the case where $\UU = 0$ and $\TT = \mod A$. Let $\{S_1, S_2, \dots, S_\ell\}$ be the set of isomorphism classes of simple $A$-modules. Then, using Proposition \ref{topbottomlatsimpleforce}, putting $\alpha_i:  \mod A \to {}^\perp S_i =: \TT_i$ and $\beta_i: \UU_i := \Filt S_i \to 0$, the assertions follow.
\end{proof}

We give also an alternative way to construct polytopes, which is a kind of converse to Proposition \ref{proposition polygons}.

\begin{proposition} \label{polyg2}\
 \begin{enumerate}[\rm (a)]
  \item Let $\TT \in \ftors A$. Consider $\ell$ distinct arrows $\alpha_i: \TT \to \TT_i$ of $\Hasse(\tors A)$. Let $\UU := \Meet_{i = 1}^\ell \TT_i$. Then $[\UU, \TT]$ is an $\ell$-polytope.
  \item Let $\UU \in \ftors A$. Consider $\ell$ distinct arrows $\beta_i: \UU_i \to \UU$ of $\Hasse(\tors A)$. Let $\TT :=  \Join_{i = 1}^\ell \UU_i$. Then $[\UU, \TT]$ is an $\ell$-polytope.
 \end{enumerate}
\end{proposition}

\begin{proof}
 By duality, we prove only (b). By \cite[Theorem 1.3]{DIJ}, all $\UU_i$ are in $\ftors A$. Thanks to Theorem \ref{muttautilt}, the basic $\tau$-tilting pairs corresponding to $\UU$, $\UU_1$, $\UU_2$, \dots, $\UU_\ell$ admit a maximal common direct summand $(N, Q)$
with $|N|+|Q| = |A| - \ell$. All $\UU_i$'s and $\UU$ appear in the $\ell$-polytope $[\UU(N, Q), \TT(N, Q)]$, so $\TT = \Join_{i=1}^\ell \UU_i \in [\UU(N, Q), \TT(N, Q)]$, and $[\UU, \TT] \subseteq [\UU(N, Q), \TT(N, Q)]$.

As the $\beta_i$ are $\ell$ arrows pointing toward $\UU$ in $\Hasse([\UU(N, Q), \TT(N, Q)] \cap \ftors A)$ and $[\UU(N, Q), \TT(N, Q)] \cap \ftors A$ is $\ell$-Hasse-regular, we have $\UU = \UU(N, Q)$ and the $\beta_i$ are all arrows pointing toward $\UU$. Hence, by Proposition \ref{proposition polygons}(b), we have $\TT(N, Q) = \TT$.
\end{proof}

Additionally, we show that if $A$ is $\tau$-tilting finite, then $\tors A$ is polygonal as defined in Section \ref{lat prelim}.

\begin{proposition} \label{torsApole}
Let $A$ be a finite-dimensional algebra that is $\tau$-tilting finite. The following hold:
 \begin{enumerate}[\rm (a)]
  \item The lattice $\tors A$ is polygonal. The polygons of $\tors A$ are precisely the $2$-polytopes.
  \item Let $[\UU, \TT]$ be a polygon of $\tors A$ and $S$ be a brick in $\TT \cap \UU^\perp$. Then:
 \begin{itemize}
  \item if $S \in \simp[\UU,\TT]$, then $S$ labels exactly two arrows of $[\UU, \TT]$;
  \item if $S \notin \simp[\UU,\TT]$, then $S$ labels exactly one arrow of $[\UU, \TT]$.
 \end{itemize}
 \end{enumerate}
\end{proposition}

\begin{proof}
 (a) Let $\TT_1\to\UU$ and $\TT_2\to\UU$ be distinct arrows of $\Hasse(\tors A)$. By Proposition \ref{polyg2}, $[\UU,\TT_1\join \TT_2]$ is a $2$-polytope, hence a polygon. The other condition for polygonality is proved dually.  As polygons are of the form $[\UU,\TT_1\join \TT_2]$ for some distinct arrows $\TT_1\to\UU$ and $\TT_2\to\UU$, we have also proved that polygons are $2$-polytopes.

 (b) By Theorem \ref{theorem filtorth}(b)(c)(d), $S$ labels at least two arrows if $S \in \simp[\UU,\TT]$ and $S$ labels at least one arrow otherwise.

 If two distinct arrows $q_1: \TT_1 \to \UU_1$ and $q_2: \TT_2 \to \UU_2$ belong to the same path of $\Hasse[\UU, \TT]$, we can suppose without loss of generality that $\TT_2 \subseteq \UU_1$. Then the label of $q_1$ is in $\UU_1^\perp$ and the label of $q_2$ is in $\TT_2$, so these labels are distinct as $\UU_1^\perp \cap \TT_2 = 0$. As a polygon has two maximal paths, $S$ labels at most two arrows of $[\UU, \TT]$.

 Consider the arrows $\TT \to \VV_1$ and $\TT \to \VV_2$ in $\Hasse[\UU, \TT]$. If two distinct arrows $q_1: \TT_1 \to \UU_1$ and $q_2: \TT_2 \to \UU_2$ belong to different paths of $\Hasse([\UU, \TT] \setminus \{\TT\})$, we can suppose without loss of generality that $\TT_1 \subseteq \VV_1$ and $\TT_2 \subseteq \VV_2$. So the label of $q_1$ is in $\VV_1 \cap \UU^\perp$ and the label of $q_2$ is in $\VV_2 \cap \UU^\perp$. As $\VV_1 \cap \VV_2 = \UU$, the labels of $q_1$ and $q_2$ have to be distinct. Combining this assertion with the first one, we have proved that all labels of arrows of $\Hasse([\UU, \TT] \setminus \{\TT\})$ have to be distinct. So, if $S \notin \simp[\UU,\TT]$, $S$ cannot label two arrows of $\Hasse [\UU, \TT]$.
\end{proof}

\subsection{Algebraic characterizations of the forcing order}

The aim of this subsection is to describe the forcing order on bricks in terms of representation theory.
We start with a characterization which holds for any finite-dimensional algebra that is $\tau$-tilting finite.

\begin{definition}
Define the \emph{filtration order} $\forces[f]$ on $\brick A$ as the transitive closure of the following.
\begin{itemize}
\item $S_1 \forces[f] S_2$ if there is a semibrick $\{S_1\} \cup E$ such that $S_2 \in \Filt (\{S_1\} \cup E) \setminus \Filt E$.
\end{itemize}
Define the \emph{pair filtration order} $\forces[pf]$ on $\brick A$ as the transitive closure of the following.
\begin{itemize}
\item $S_1 \forces[pf] S_2$ if there is a semibrick $\{S_1,S'_1\}$ such that $S_2 \in \Filt\{S_1,S'_1\} \setminus \{S'_1\}$.
\end{itemize}
\end{definition}

We have the following first main result in this subsection.

\begin{theorem} \label{theorem forcing}
Let $A$ be a finite-dimensional $k$-algebra that is $\tau$-tilting finite. The forcing order $\forces$, the filtration order $\forces[f]$ and the pair filtration order $\forces[pf]$ coincide. In particular, for $S_1, S_2 \in \brick A$, if $S_1 \forces S_2$, then $S_1$ is a subfactor of $S_2$.
\end{theorem}

We start with a lemma:

\begin{lemma} \label{forcdep} Let $A$ be a finite-dimensional $k$-algebra that is $\tau$-tilting finite.
 If $E \cup \{S_1\} \in \sbrick A$ and $S_2 \in \Filt (E \cup \{S_1\}) \setminus \Filt E$ is a brick then $S_1 \forces S_2$.
\end{lemma}

\begin{proof}
 By Theorem \ref{ttwideb}, there exists $(N, Q) \in \trigidpair A$ such that $\WW(N, Q) = \Filt (E \cup \{S_1\})$. Let us denote $\UU = \UU(N, Q)$ and $\TT = \TT(N, Q)$. By Theorem~\ref{theorem filtorth}(b), $S_2$ labels an arrow $\TT_1 \to \TT_2$ of $[\UU, \TT]$.

 Let us prove by induction on $\TT_2 \in [\UU, \TT]$ that $S_1\forces S_2$. First of all, if $\TT_2 = \UU$, then $S_2 \in E \cup \{S_1\}$ holds by Theorem \ref{theorem filtorth}(c). By our assumption, $S_2 \notin \Filt E$, so $S_2 = S_1$.

 Otherwise, suppose that $\TT_2 \supsetneq \UU$. Then there exists an arrow $\TT_2 \to \TT_3$ in $[\UU, \TT]$. Taking the common summands of the $\tau$-tilting pairs corresponding to $\TT_1$, $\TT_2$ and $\TT_3$, we obtain a 2-polytope $[\UU',\TT'] \subseteq [\UU,\TT]$, hence a polygon, containing the arrows $\TT_1 \to \TT_2$ and $\TT_2 \to \TT_3$.
 If $\TT' = \TT_1$, then $S_2$ labels an arrow ending at $\UU' \subsetneq \TT_2$ by Theorem~\ref{theorem filtorth}(c)(d). By the induction hypothesis, we have $S_1 \forces S_2$.

 It remains to consider the case $\TT' \neq \TT_1$. Let $\simp[\UU',\TT']=\{S_3,S_4\}$.
 Then $S_3\forces S_2$ and $S_4\forces S_2$ hold by Proposition \ref{forcing=polygonal}.
 Moreover, as $\Filt(S_3, S_4) \ni S_2 \notin \Filt E$, either $S_3\notin \Filt E$ or $S_4\notin \Filt E$ holds.
 Since $S_3$ and $S_4$ label arrows ending at $\UU' \subsetneq \TT_2$, either $S_1\forces S_3$ or $S_1\forces S_4$ holds by the induction hypothesis.
 So $S_1\forces S_2$ holds.
\end{proof}

\begin{proof}[Proof of Theorem \ref{theorem forcing}]
Proposition~\ref{torsApole}(a) says that $\tors A$ is polygonal, so by Proposition~\ref{forcing=polygonal}, the forcing order coincides with the polygonal order.

Clearly $S_1\forces[pf]S_2$ implies $S_1\forces[f]S_2$.
Thanks to Lemma \ref{forcdep} and by transitivity of $\forces$, $S_1\forces[f]S_2$ implies $S_1\forces S_2$.

We show that $S_1\forces S_2$ implies $S_1\forces[pf]S_2$. As $\forces$ and $\forces[p]$ on $\Hasse_1(\tors A)$ coincide, where $\forces[p]$ is the polygonal forcing, it suffices to consider the case where
there are arrows $q_1$ and $q_2$ labelled by $S_1$ and $S_2$ such that $q_1 \forces[p] q_2$ in a polygon $[\UU,\TT]$ of $\tors A$. If $q_1 \forceseq[p] q_2$, we have $S_1 \cong S_2$, so $S_1\forces[pf]S_2$. So we assume that $q_1$ and $q_2$ are not forcing equivalent. Then $q_1$ is an arrow of $\Hasse[\UU, \TT]$ incident to $\TT$ or $\UU$ and $q_2$ is an arrow of $\Hasse[\UU, \TT]$ that is not incident to $\TT$ or $\UU$.
 By Theorem~\ref{theorem filtorth}, the semibrick $\simp[\UU,\TT]$ is of the form $\{S_1,S'_1\}$ and $S_2$ belongs to $\Filt \{S_1, S_1'\}$. Then, by Proposition~\ref{torsApole}(b), $S_2 \not\cong S_1'$, so $S_1 \forces[pf] S_2$ holds.

The last statement is clear since $\forces$ coincides with $\forces[f]$.
\end{proof}

We define a convenient concept.
\begin{definition}
 An $A$-module is \emph{multiplicity free} if it has no repetition in its composition series.
\end{definition}

Notice that a multiplicity free indecomposable module is a brick. Before giving more specific characterizations, we give the following elementary observation.

\begin{lemma} \label{lemma extbricks}
Let $\{S, S'\} \in \sbrick A$ with $S \not\cong S'$
and consider a non-split short exact sequence $0 \to S \to X \to S' \to 0$. Then $X$ is a brick.
\end{lemma}

\begin{proof}
 Recall that $S$ and $S'$ are two non-isomorphic simple objects in the abelian category $\WW := \Filt(S, S')$. By assumption, $X$ has length $2$ in $\WW$ and is not isomorphic to $S\oplus S'$. Thus $X$ is indecomposable and multiplicity free in $\WW$, and hence it is a brick.
\end{proof}

For multiplicity free bricks, the forcing order is described in a very simple way.

\begin{corollary} \label{elemforc}
 Let $A$ be a finite-dimensional algebra that is $\tau$-tilting finite. Consider $X, Y \in \brick A$, such that $Y$ is multiplicity free. Then $X\forces Y$ if and only if $X$ is a subfactor of $Y$.
\end{corollary}

\begin{proof}
By Theorem~\ref{theorem forcing}, it suffices to prove `if' part.
Suppose that $X$ is a subfactor of $Y$. We show by induction on $\dim Y - \dim X$ that $X$ forces $Y$. This is clear if $\dim Y = \dim X$.
Suppose $\dim Y > \dim X$. Then there exists a subfactor $X'$ of $Y$, a simple $A$-module $S$ and a non-split short exact sequence of one of the following forms:
\[\xi: 0 \to S \to X' \to X \to 0\ \text{ or }\ \xi': 0 \to X \to X' \to S \to 0.\]
As $X'$ is a subfactor of $Y$, it is also multiplicity free and hence $\{X,S\}$ is a semibrick.
By Lemma~\ref{lemma extbricks}, $X'$ is a brick and by Theorem~\ref{theorem forcing}, $X$ forces $X'$. On the other hand, by induction hypothesis, $X'$ forces $Y$. Therefore $X$ forces $Y$.
\end{proof}

In \cite{IRRT}, the forcing order is shown to be equivalent to the \emph{doubleton extension order} when $A$ is a preprojective algebras of Dynkin type. We end this section by proving this for a much more general class of algebras.
\begin{definition}[\cite{IRRT}] \label{dbl}
 The \emph{doubleton extension order} on $\brick A$ is the transitive closure $\forces[d]$ of the relation defined by:
 $S_1 \forces[d] S_2$ if there exists a brick $S_1'$ such that
   \begin{itemize}
    \item[] $\dim \Ext^1_A(S_1, S_1') = 1$ and there is an exact sequence $0 \to S_1' \to S_2 \to S_1 \to 0$;
    \item[or] $\dim \Ext^1_A(S_1', S_1) = 1$ and there is an exact sequence $0 \to S_1 \to S_2 \to S_1' \to 0$.
   \end{itemize}
\end{definition}

We will consider bricks having the following stronger property.
\begin{definition}
 A brick $S \in \mod A$ is called a \emph{stone} if $\Ext^1_A(S, S) = 0$. It is called a \emph{$k$-stone} if additionally $\End_A(S) \cong k$.
\end{definition}

We give the following characterization that is the second main theorem of this subsection.
\begin{theorem}\label{forcing=doubleton}
 Let $A$ be a finite-dimensional $k$-algebra that is $\tau$-tilting finite such that all bricks of $\mod A$ are $k$-stones. Then the forcing order $\forces$ on $\brick A$ coincides with the doubleton extension order $\forces[d]$.
\end{theorem}

From now on, until the end of this subsection, we suppose that $A$ is $\tau$-tilting finite and all bricks are $k$-stones. We start with the following observation.

\begin{lemma} \label{lemma basicstones}
Let $\{S, S'\}$ be a semibrick of $A$. Then $\dim \Ext^1_A(S', S)$ is $0$ or $1$.
 In the latter case, the non-split short exact sequence $0 \to S \to X \to S' \to 0$ satisfies:
 \begin{enumerate}[\rm(a)]
  \item $\Hom_A(X, S) = \Hom_A(S', X) = \Ext^1_A(X, S) = \Ext^1_A(S', X) = 0$;
  \item $\dim \Hom_A(S, X) = \dim \Hom_A(X, S') = 1$;
  \item $\dim \Ext^1_A(S, X)$ is either $0$ or $1$;
  \item $\dim \Ext^1_A(X, S')$ is either $0$ or $1$.
 \end{enumerate}
\end{lemma}

\begin{proof}
 We suppose that $\Ext^1_A(S', S) \neq 0$. Let us consider a non-split short exact sequence
 \[\xi: 0 \to S \to X \to S' \to 0.\]
 According to Lemma~\ref{lemma extbricks}, $X$ is also a brick and hence a $k$-stone. Applying $\Hom_A(-, S')$ to $\xi$ gives $\dim \Hom_A(X, S') = 1$. Applying $\Hom_A(-, S)$ to $\xi$ gives the exact sequence
 \begin{equation}
  0 \to \Hom_A(X, S) \to \Hom_A(S, S) \to \Ext^1_A(S', S) \to \Ext^1_A(X, S) \to 0 \label{exseq1}
 \end{equation}
 Because $\xi$ does not split and $\End_A(S) \cong k$, we obtain $\Hom_A(X, S) = 0$. Then, applying $\Hom_A(X, -)$ to $\xi$ yields the exact sequence
 \begin{align*}0= & \Hom_A(X,S) \to \Hom_A(X, X) \to \Hom_A(X, S') \\ \to & \Ext^1_A(X, S) \to \Ext^1_A(X,X) = 0. \end{align*}
Since $\dim \Hom_A(X, S') = 1$, $\Hom_A(X, X) \to \Hom_A(X, S')$ is surjective, therefore $\Ext^1_A(X, S) = 0$. Again by \eqref{exseq1}, we get $\dim \Ext^1_A(S', S) = 1$. The dual reasoning implies that $\dim \Hom_A(S, X) = 1$ and $\Hom_A(S', X) = \Ext^1_A(S', X) = 0$.

 We proved the first part of the Lemma, (a) and (b). For (c), applying $\Hom_A(S,-)$ to $\xi$ gives the exact sequence
 \[0 = \Ext^1_A(S, S) \to \Ext^1_A(S, X) \to \Ext^1_A(S, S').\]
 Exchanging the role of $S$ and $S'$, we have already proven that $\dim \Ext^1_A(S, S')$ is $0$ or $1$, so (c) holds. Finally, (d) is dual to (c).
\end{proof}

We deduce a description of $\Filt(S_0, S_1)$:

\begin{lemma} \label{lemma formfiltstones}
 Let $\{S_0,S_1\}$ be a semibrick with $S_0 \not\cong S_1$. Then we have an equivalence of categories $\Filt(S_0, S_1) \cong \mod (kQ/I)$ where
 \[Q = \left(\boxinminipage{\xymatrix{\bullet \ar@/^/[r]^{} & \bullet \ar@/^/[l]^{} }}\right)\]
 and $I$ is an ideal satisfying $(Q_1^N) \subseteq I \subseteq (Q_1)$ for $N$ big enough, where $(Q_1^\ell)$ is the two-sided ideal generated by paths of length $\ell$. Moreover, $S_0$ and $S_1$ correspond to the simple $kQ/I$ modules.
\end{lemma}

\begin{proof}
 We start with the case where $A$ is basic with two isomorphism classes of simple $A$-modules $S_0$ and $S_1$. Denote $E = A/\rad A$. As $A$-modules, we have $E \cong S_0 \oplus S_1$. So, as $\mod E \subseteq \mod A$ is fully faithful and $S_0$ and $S_1$ are $k$-stones, we get $E \cong \End_E(E) \cong \End_A(E) \cong k \times k$ as $k$-algebras. So $A$ is \emph{elementary} in the sense of \cite[Section III.1]{ARS}. We consider the $E$-bimodule $F := \rad A / \rad^2 A$. By \cite[Theorem III.1.9(b)]{ARS}, there is a surjective morphism $\phi: T_E (F) \twoheadrightarrow A$ where $T_E (F)$ is the tensor algebra $\bigoplus_{n \geq 0} F^{\otimes_E n}$ and $(F^N) \subseteq \Kernel \phi \subseteq (F^2)$. Let $e_0$ and $e_1$ be orthogonal primitive idempotents corresponding to $S_0$ and $S_1$ respectively. For $i, j \in \{0,1\}$, $\dim e_i F e_j$ is the multiplicity of $S_i$ as a direct summand of the $E$-module $F e_j$, that is, by \cite[Proposition III.1.15(a)]{ARS}, $\dim \Ext^1_A(S_j, S_i)$. As $S_0$ and $S_1$ are stones, $\Ext^1_A(S_0, S_0) = 0 = \Ext^1_A(S_1, S_1)$. Moreover, by Lemma~\ref{lemma basicstones}, $\dim \Ext^1_A(S_0, S_1) \le 1$, and $\dim \Ext^1_A(S_1, S_0) \le 1$, so we deduce from the above discussion that $T_E (F)$ is a quotient of $kQ$. The result follows in this case.

 Consider now the general case. We know that $\WW := \Filt(S_0, S_1)$ is a wide subcategory. Moreover, as $\mod A$ is $\tau$-tilting finite, using Theorem~\ref{theorem filtorth}, $\WW$ is functorially finite. Hence, it is easy that a minimal left $\WW$-approximation $P$ of $A$ is a progenerator of $\WW$. 
 So, by Morita theory, $\WW \cong \mod B$ for $B = \End_A(P)$, which is a basic finite-dimensional $k$-algebra and satisfies the assumptions of the previous paragraph. The conclusion follows.
\end{proof}

From the above, we deduce the following characterization of polygons in $\tors A$:

\begin{proposition} \label{proposition shapepoly}
 Suppose that $[\UU, \TT]$ is a polygon of $\tors A$, and let $\{S_0, S_1\} = \simp[\UU,\TT]$.
Depending on $(\dim \Ext^1_A(S_1, S_0),\dim \Ext^1_A(S_0, S_1))$, the polygon $[\UU, \TT]$ is labelled in the following way, where $X_i$ is the non-trivial extension of $S_{1-i}$ by $S_{i}$:
 \begin{center} \begin{tabular}{ccccccc}
  $\xymatrix@C=.705cm@R=.35cm@!=0cm{
    & \TT \ar[dddl]_{S_0} \ar[dddr]^{{S_1}} \\ \\ \\
   \bullet \ar[dddr]_{{S_1}} & & \bullet \ar[dddl]^{S_0} \\ \\ \\
   & \UU
  }$ & &
  $\xymatrix@C=.705cm@R=.35cm@!=0cm{
    & \TT \ar[ddl]_{S_0} \ar[dddr]^{{S_1}} \\ \\
   \bullet \ar[dd]_{X_0} \\
    & & \bullet \ar[dddl]^{S_0} \\
    \bullet \ar[ddr]_{{S_1}}   \\ \\
   & \UU
  }$ & &
  $\xymatrix@C=.705cm@R=.35cm@!=0cm{
    & \TT \ar[dddl]_{S_0} \ar[ddr]^{{S_1}} \\ \\
    & & \bullet \ar[dd]^{X_1} \\
    \bullet \ar[dddr]_{{S_1}} \\ & & \bullet \ar[ddl]^{{S_0}}  \\ \\
   & \UU
  }$ & &
  $\xymatrix@C=.705cm@R=.35cm@!=0cm{
    & \TT \ar[ddl]_{S_0} \ar[ddr]^{{S_1}} \\ \\
   \bullet \ar[dd]_{X_0} & & \bullet \ar[dd]^{X_1} \\ \\
    \bullet \ar[ddr]_{{S_1}} & & \bullet \ar[ddl]^{{S_0}}  \\ \\
   & \UU
  }$ \\
  $(0,0)$ & & $(1,0)$ & & $(0,1)$ & & $(1,1)$.
 \end{tabular}\end{center}
\end{proposition}

\begin{proof}
 By Lemma~\ref{lemma formfiltstones}, $\WW = \Filt(S_0, S_1) \cong \mod (kQ/I)$ where
 \[Q = \left(\boxinminipage{\xymatrix{\bullet \ar@/^/[r]^{} & \bullet \ar@/^/[l]^{} }}\right)\]
 and $I$ is an ideal satisfying $(Q_1^N) \subseteq I \subseteq (Q_1)$ for $N$ big enough. According to Proposition \ref{bcomptred}, the labels in $\Hasse [\UU, \TT]$ coincide with the labels in $\mod (kQ/I)$ via the equivalence above. Therefore, we can suppose that $A = kQ/I$. Then the computation of $\Hasse (\sttilt A)$ is straightforward as $A$ is a Nakayama algebra with two simple modules.
\end{proof}

We deduce the following proposition.
\begin{proposition} \label{prpstone}
 For $S_1, S_2 \in\brick A$, the following are equivalent:
 \begin{enumerate}[\rm (i)]
  \item There exists a semibrick $\{S_1,S'_1\}$ such that $S_2 \in \Filt\{S_1,S'_1\} \setminus \{S'_1\}$;
  \item $S_1 \cong S_2$ or there exists a brick $S_1' \in \mod A$ such that one of the following situations occurs:
  \begin{itemize}
   \item $\dim \Ext^1_A(S_1, S_1') = 1$ and there is an extension $0 \to S_1' \to S_2 \to S_1 \to 0$;
   \item $\dim \Ext^1_A(S_1', S_1) = 1$ and there is an extension $0 \to S_1 \to S_2 \to S_1' \to 0$.
  \end{itemize}
 \end{enumerate}
 Moreover, in \textup{(ii)}, $\{S_1, S_1'\}$ is automatically a semibrick.
\end{proposition}

\begin{proof}
 (i) $\Rightarrow$ (ii). If $S_1 \not\cong S_2$, this is an immediate consequence of Proposition~\ref{proposition shapepoly}.

 (ii) $\Rightarrow$ (i). Suppose that there exists a short exact sequence \[\xi: 0 \to S_1 \to S_2 \to S_1' \to 0.\] Applying $\Hom_A(S_1', -)$ to $\xi$ gives the long exact sequence
 \begin{align}0 &\to \Hom_A(S_1', S_1) \to \Hom_A(S_1', S_2) \to \Hom_A(S_1', S_1') \nonumber \\ &\to \Ext^1_A(S_1', S_1) \to \Ext^1_A(S_1', S_2) \to \Ext^1_A(S_1', S_1') = 0 \label{lex4}.\end{align}
 Therefore, as $\dim \Ext^1_A(S_1', S_1) = 1$ and $\xi$ does not split, we get $\Ext^1_A(S_1', S_2) = 0$. Then, applying $\Hom_A(-, S_2)$ to $\xi$ gives the exact sequence
  \[0 \to \Hom_A(S_1', S_2) \to \Hom_A(S_2, S_2) \to \Hom_A(S_1, S_2) \to \Ext^1_A(S_1', S_2) = 0,\]
 so $\Hom_A(S_1', S_2) = 0$ and $\dim \Hom_A(S_1, S_2) = 1$. Using \eqref{lex4} again, we obtain $\Hom_A(S_1', S_1) = 0$. Applying $\Hom_A(S_1, -)$ to $\xi$ gives $\Hom_A(S_1, S_1') = 0$. So $S_1$ and $S_1'$ are orthogonal, and we have the assertion.
\end{proof}

\begin{proof}[Proof of Theorem~\ref{forcing=doubleton}]
 By Proposition~\ref{prpstone}, we get that $\forces[pf]$ and $\forces[d]$ coincide. As, by Theorem~\ref{theorem forcing}, $\forces$ and $\forces[pf]$ coincide, the result follows.
\end{proof}

 The following useful observation will be used in Section \ref{camblat}.
 \begin{proposition}\label{splitbrick}
  Let $A$ be a finite-dimensional $k$-algebra that is $\tau$-tilting finite such that all bricks of $\mod A$ are $k$-stones. Then for $S \in \brick A$ that is not simple, there is a semibrick $\{S_1, S_2\}$ such that $\dim \Ext^1_A(S_2, S_1) = 1$ and a short exact sequence $0 \to S_1 \to S \to S_2 \to 0$.
 \end{proposition}

 \begin{proof}
  By Proposition \ref{topbottomlatsimpleforce}, there is a simple $A$-module $S_0$ such that $S_0 \forces S$. By Theorem \ref{forcing=doubleton}, $S_0 \forces[d] S$. As $S_0 \not\cong S$, by definition of the doubleton extension order, there exist two bricks $S_1$ and $S_2$ with $\dim \Ext^1_A(S_2, S_1) = 1$ and a short exact sequence $0 \to S_1 \to S \to S_2 \to 0$. By Proposition \ref{prpstone}, $\{S_1, S_2\}$ is a semibrick.
 \end{proof}

\section{Algebraic lattice congruences on torsion classes}

\subsection{General results on morphisms of algebras}\label{gen-res}

Let $\AA$ be an abelian category. A full subcategory $\TT$ of $\AA$ is a \emph{torsion class} in $\AA$ if it is closed under factor objects and extensions. Dually we define a \emph{torsion-free class} in $\AA$. The classes $\tors \AA$ of torsion classes and $\torf \AA$ of torsion-free classes in $\AA$ are ordered by inclusion.

The following observation is a starting point of this section.

\begin{proposition}\label{AB tors}
Let $\AA$ and $\BB$ be abelian categories.
\begin{enumerate}[\rm(a)]
\item Let $F:\AA\to\BB$ be a right exact functor. Then we have order-preserving maps $F^*:\tors\BB\to\tors\AA$ and $F_*:\torf\BB\to\torf\AA$ given by
\[F^*(\TT):=\{X\in\AA\mid F(X)\in\TT\}\ \text{ and }\ F_*(\FF):=F^*({}^{\perp_{\BB}}\FF)^{\perp_{\AA}}.\]

\item Let $G:\BB\to\AA$ be a left exact functor. Then we have order-preserving maps $G^*:\torf\AA\to\torf\BB$ and $G_*:\tors\AA\to\tors\BB$ given by
\[G^*(\FF):=\{X\in\BB\mid G(X)\in\FF\}\ \text{ and }\ G_*(\TT):={}^{\perp_{\BB}}G^*(\TT^{\perp_{\AA}}).\]
\end{enumerate}
\end{proposition}

\begin{proof}
(a) Fix $\TT\in\tors\BB$. Let $0\to X\xto{\iota} Y\to Z\to0$ be an exact sequence in $\AA$.
Then $F(X)\to F(Y)\to F(Z)\to0$ is an exact sequence in $\BB$.
If $Y\in F^*(\TT)$, then $F(Y)\in\TT$, so $F(Z)\in\TT$. Thus $Z\in F^*(\TT)$.
Similarly, if $X,Z\in F^*(\TT)$, then $F(X),F(Z)\in\TT$ and hence $\Image F(\iota) \in \TT$ so $F(Y)\in\TT$. Thus $Y\in F^*(\TT)$. Clearly $F^*$ is order-preserving.

Let $\FF\in\torf\BB$. Clearly $F_*(\FF)$ is a torsion class in $\AA$ since it is defined by $(-)^{\perp_{\AA}}$.
Since ${}^{\perp_{\BB}}(-):\torf\BB\to\tors\BB$
and $(-)^{\perp_{\AA}}:\tors\AA\to\torf\AA$
are order-reversing, $F_*$ is also order-preserving.

(b) This is dual to (a).
\end{proof}

A \emph{torsion pair} is a pair $(\TT,\FF)$ consisting of a torsion class $\TT$ in $\AA$ and a torsion-free class $\FF$ in $\AA$ such that $\Hom_{\AA}(\TT, \FF) = 0$, and for any $X\in\AA$, there exists a short exact sequence $0\to T\to X\to F\to0$ with $T\in\TT$ and $F\in\FF$.
In this case, we have $\TT={}^{\perp_{\AA}}\FF$ and $\FF=\TT^{\perp_{\AA}}$.

We say that $\AA$ has \emph{enough torsion-free classes} if for any torsion class $\TT$ in $\AA$, there exists a torsion-free class $\FF$ in $\AA$ such that $(\TT,\FF)$ is a torsion pair.
Dually we define for $\AA$ to have \emph{enough torsion classes} in an obvious way. Finally, we say that $\AA$ has \emph{enough torsion pairs} if is has enough torsion classes and enough torsion-free classes.

In Definition~\ref{defadjlat}, we gave the concept of adjoint pairs of order-preserving maps. Any adjoint pair of functors induces an adjoint pair of order-preserving maps.

\begin{proposition}\label{AB adjoint}
Let $\AA$ and $\BB$ be abelian categories, and $(F:\AA\to\BB,\,\,G:\BB\to\AA)$ be an adjoint pair of functors.
\begin{enumerate}[\rm(a)]
\item If $\BB$ has enough torsion pairs, then $(G_*:\tors\AA\to\tors\BB,\,\,F^*:\tors\BB\to\tors\AA)$ is an adjoint pair.
\item If $\AA$ has enough torsion pairs, then $(F_*:\torf\BB\to\torf\AA,\,\,G^*:\torf\AA\to\torf\BB)$ is an adjoint pair.
\end{enumerate}
\end{proposition}

\begin{proof}
(a) For $\S\in\tors\BB$, we take a torsion pair $(\S,\FF)$ in $\BB$.
Then, for $\TT\in\tors\AA$, we have $\TT\subseteq F^*(\S)$ if and only if $F(\TT)\subseteq\S$ if and only if $\Hom_{\BB}(F(\TT),\FF)=0$. Since $(F,G)$ is an adjoint pair, this is equivalent to $\Hom_{\AA}(\TT,G(\FF))=0$.
This holds if and only if $G^*(\TT^{\perp_{\AA}})
\supseteq\FF$. As $\BB$ has enough torsion classes, we have $\FF = ({}^{\perp_{\BB}} \FF)^{\perp_{\BB}}$ and $G^*(\TT^{\perp_{\AA}}) = ({}^{\perp_{\BB}} G^*(\TT^{\perp_{\AA}}))^{\perp_{\BB}}$. Therefore, $G^*(\TT^{\perp_{\AA}}) \supseteq\FF$ if and only if $G_*(\TT)={}^{\perp_{\BB}} G^*(\TT^{\perp_{\AA}}) \subseteq {}^{\perp_{\BB}} \FF = \S$.
We have shown that $(G_*,F^*)$ is an adjoint pair.

(b) For opposite categories $\AA^{\op}$ and $\BB^{\op}$, we have an adjoint pair $(G^{\op}:\BB^{\op}\to\AA^{\op},\ F^{\op}:\AA^{\op}\to\BB^{\op})$. By (a), this gives rise to an adjoint pair $(F^{\op}_*:\tors\BB^{\op}\to\tors\AA^{\op},\ G^{\op*}:\tors\AA^{\op}\to\tors\BB^{\op})$. Using natural identifications $\tors\BB^{\op}=\torf\BB$ and $\tors\AA^{\op}=\torf\AA$, the desired assertion follows.
\end{proof}

We apply these observations to morphisms of algebras. It is immediate that, for a finite-dimensional $k$-algebra $A$, $\mod A$ has enough torsion classes and enough torsion-free classes. In the rest of this subsection, let $\phi:A\to B$ be a morphism of finite-dimensional $k$-algebras. We denote the associated restriction functor by
\[E:={}_A(-):\mod B\to\mod A,\]
which is an exact functor.
Moreover we have a right exact functor $E_\lambda$ and a left exact functor $E_\rho$ given by
\[E_\lambda:=B\otimes_A-:\mod A\to\mod B\ \text{ and }\ E_\rho:=\Hom_A(B,-):\mod A\to\mod B,\]
which give rise to adjoint pairs $(E_\lambda,E)$ and $(E,E_\rho)$.
For $\TT\in\tors A$ and $\S\in\tors B$, we define:
\begin{align*}
\phi_-(\TT) &:= E^*(\TT) = \{Y\in\mod B\mid {}_AY\in\TT\},\\
\phi_+(\TT) &:= E_*(\TT) = {}^{\perp_B}\{Y\in\mod B\mid {}_AY\in\TT^{\perp_A}\},\\
\phi^+(\S) &:= E_\lambda^*(\S) = \{X\in\mod A\mid B\otimes_AX\in\S\},\\
\phi^-(\S) &:= E_{\rho*}(\S) = {}^{\perp_A}\{X\in\mod A\mid\Hom_A(B,X)\in\S^{\perp_B}\}.
\end{align*}
We summarize the following basic properties.

\begin{theorem}\label{relating the maps}\quad
\begin{enumerate}[\rm(a)]
\item \label{lowers}
$\phi_-$ and $\phi_+$ are order-preserving maps $\tors A\to\tors B$.
\item \label{uppers}
$\phi^+$ and $\phi^-$ are order-preserving maps $\tors B\to\tors A$.
\item \label{adj+}
$(\phi_+:\tors A\to\tors B,\,\,\phi^+:\tors B\to\tors A)$ is an adjoint pair.
\item \label{adj-}
$(\phi^-:\tors B\to\tors A,\,\,\phi_-:\tors A\to\tors B)$ is an adjoint pair.
\item \label{meet semi homs}
The maps $\phi_-:\tors A\to\tors B$ and $\phi^+:\tors B\to\tors A$ are morphisms of complete meet-semilattices.
\item \label{join semi homs}
The maps $\phi_+:\tors A\to\tors B$ and $\phi^-:\tors B\to\tors A$ are morphisms of complete join-semilattices.
\item \label{subs}
For any $\TT\in\tors A$, we have $\phi_-(\TT)\subseteq\phi_+(\TT)$.
\end{enumerate}
\end{theorem}

\begin{proof}
\eqref{lowers}\eqref{uppers} These are shown in Proposition~\ref{AB tors}.

\eqref{adj+}\eqref{adj-} These are shown in Proposition~\ref{AB adjoint}.

\eqref{meet semi homs}\eqref{join semi homs} These follow from \eqref{adj+}, \eqref{adj-} and Propositions~\ref{adjoint}.

\eqref{subs} Let $Y \in \phi_-(\TT)$, \emph{i.e.} ${}_A Y\in\TT$. Then we have, for all $Z\in\mod B$ satisfying ${}_AZ\in\TT^{\perp_A}$, $\Hom_B(Y,Z)\subseteq\Hom_A({}_A Y,{}_A Z)=0$. Therefore \begin{align*}Y&\in{}^{\perp_B}\{Z\in\mod B\mid {}_A Z\in\TT^{\perp_A}\} = \phi_+(\TT).\qedhere\end{align*}
\end{proof}

We will observe in Example~\ref{not lattice morphism}(a) below that, contrary to what one might have expected given
Theorem~\ref{relating the maps}\eqref{subs}, $\phi^-(\S)\subseteq\phi^+(\S)$ does not necessarily hold for $\S\in\tors B$ in general.
We give a sufficient condition for this to hold.

Recall that we call a morphism $\phi:A\to B$ of finite-dimensional $k$-algebras an \emph{epimorphism} if it satisfies the following three equivalent
conditions \cite[Propositions 1.1, 1.2]{Sto}, \cite[Proposition 1.1]{Si}, see also \cite{Ste}:
\begin{itemize}
 \item $\phi$ is an epimorphism in the category of rings;
 \item $B \otimes_A B \cong B$ through multiplication;
 \item the functor ${}_A(-):\mod B\to \mod A$ is fully faithful.
\end{itemize}
Note that, while a surjective morphism of rings is an epimorphism, the converse is far from being true, e.g. the following inclusion is a ring epimorphism:
\[\phi: \begin{bmatrix} k & k \\ 0 & k \end{bmatrix} \hookrightarrow \begin{bmatrix} k & k \\ k & k \end{bmatrix}.\]
For ring epimorphisms, we have the following property.

\begin{proposition}
Let $\phi:A\to B$ be an epimorphism of finite-dimensional $k$-algebras.
For any $\S\in\tors B$, we have $\phi^-(\S)\subseteq\phi^+(\S)$.
\end{proposition}

\begin{proof}
Let $\S\in\tors B$ and $\FF:=\S^{\perp_B}$.
Since $\phi$ is an epimorphism, we have $\Hom_A(B,{}_A(-))=\Hom_B(B,-)=\id_{\mod B}$ and hence
\[{}_A\FF\subseteq\{Y\in\mod A\mid\Hom_A(B,Y)\in\FF\}.\]
Assume $X\in\phi^-(\S)$, that is, $\Hom_A(X,Y)=0$ holds for any $Y\in\mod A$ satisfying $\Hom_A(B,Y)\in\FF$.
Thus $\Hom_B(B\otimes_AX,\FF)=\Hom_A(X,{}_A\FF)=0$ holds. Therefore $B\otimes_AX\in\S$ and we have $X\in\phi^+(\S)$.
\end{proof}

The following example shows that $\phi^+$ and $\phi_-$ are not necessarily morphisms of join-semilattices, and $\phi^-$ and $\phi_+$ are not necessarily morphisms of meet-semilattices.

\begin{example}\label{not lattice morphism}\
\begin{enumerate}[\rm(a)]
\item Let $A=k$ and $B$ be an arbitrary finite-dimensional $k$-algebra with $n \geq 2$ non-isomorphic simple modules $S_1$, $S_2$, \dots, $S_n$. For any $\S\in\tors B$, it is easy to check that
\[\phi^+(\S)=\left\{\begin{array}{ll}
0& \text{if } \S\neq\mod B\\
\mod A&\text{if } \S=\mod B
\end{array}\right.
\ \text{ and }\
\phi^-(\S)=\left\{\begin{array}{ll}
0&\text{if } \S=0\\
\mod A&\text{if }\S\neq0.
\end{array}
\right.\]
For all $i = 1, \dots, n$, $\Filt S_i \neq \mod B$, while $\Join_i \Filt S_i = \mod B$ so $\phi^+$ is not a morphism of join-semilattices. In the same way, for all $i = 1, \dots, n$, ${}^\perp S_i \neq 0$, while $\Meet_i {}^\perp S_i = 0$ so $\phi^-$ is not a morphism of meet-semilattices.

\item Let $A$ be a finite-dimensional algebra with $n \geq 2$ non-isomorphic simple modules $S_1$, $S_2$, \dots, $S_n$.  We consider an embedding $\phi: A \hookrightarrow B$ where $B$ is a matrix algebra $B$, which is simple. The only torsion classes in $\mod B$ are $0$ and $\mod B$. For $\TT \in \tors A$, by Theorem \ref{relating the maps}\eqref{adj+}, we have $\phi_+(\TT) = 0$ if and only of $\TT \subseteq \phi^+(0)$ if and only if $B\otimes_A\TT=0$.
So we have 
\[\phi_+(\TT)=\left\{\begin{array}{ll}
0&\text{if } B\otimes_A\TT=0 \\ 
\mod B&\text{if }  B\otimes_A\TT\neq0 \\ 
\end{array}
\ \text{ and }\
\phi_-(\TT)=\left\{\begin{array}{ll}
0& \text{if }  {}_AB\notin\TT\\ 
\mod B&\text{if } {}_AB\in\TT. 
\end{array}\right.
\right.\]
For any $i = 1, \dots, n$, as ${}^\perp S_i$ contains the projective cover $P_j$ of $S_j$ for $j \neq i$. As $B \otimes_A P_j \neq 0$, 
we have $\phi_+({}^\perp S_i) = \mod B$. On the other hand, $\Meet_i {}^\perp S_i = 0$, so $\phi_+$ is not a morphism of meet-semilattices.
Since $B$ is sincere as an $A$-module and $n\geq 2$, we have ${}_AB\notin\Filt S_i$, 
so $\phi_-(\Filt S_i) = 0$.
On the other hand, $\Join_i \Filt S_i = \mod A$, so $\phi_-$ is not a morphism of join-semilattices.
\end{enumerate}
\end{example}

Let us fix a surjective morphism $\phi:A\to B$ of finite-dimensional $k$-algebras.
In this case, the functor ${}_A(-):\mod B\to\mod A$ is fully faithful, and we can regard
$\mod B$ as a full subcategory of $\mod A$ consisting of $A$-modules annihilated by $\Kernel \phi$.
Then $\mod B$ is closed under submodules and factor modules in $\mod A$. We have
\[(B\otimes_A-)\circ{}_A(-)=\id_{\mod B}=\Hom_A(B,-)\circ{}_A(-).\]
For a subcategory $\CC$ of $\mod A$, we consider the subcategory
\[\overline{\CC}:=\CC\cap\mod B \subseteq \mod B.\]
We get the following basic properties.
\begin{proposition}\label{tau-rigid gives tau-rigid}\
\begin{enumerate}[\rm(a)]
\item If $X$ is a $\tau$-rigid $A$-module, then $B\otimes_AX$ is a $\tau$-rigid $B$-module.
\item There is a commutative diagram
\[\xymatrix@C=2cm{
\trigid A\ar[r]^-{\Fac-}\ar[d]_{B\otimes_A-}&\ftors A\ar[d]^{\overline{(-)}}\\
\trigid B\ar[r]^-{\Fac-}&\ftors B.}\]
\end{enumerate}
\end{proposition}

\begin{proof}
(a) Let $P_1 \rightarrow P_0 \rightarrow X \rightarrow 0$ be a
minimal projective presentation of $X$ in $\mod A$.  Then $X$ is
$\tau$-rigid if and only if the induced map $\Hom(P_0,X)\rightarrow
\Hom(P_1,X)$ is surjective, see \cite[Proposition 2.4]{AIR}.  Using this,
it is easy to see that if $X$ is a $\tau$-rigid $A$-module, then
$B \otimes_A X$ is a $\tau$-rigid $B$-module.

(b) Suppose $X$ is a $\tau$-rigid $A$-module.
By (a), $B\otimes_A X$ is a $\tau$-rigid $B$-module.  It is
clear that $\Fac (B\otimes_A X) \subseteq \overline{\Fac X}$.
On the other hand, if $Y\in \overline{\Fac X}$, the surjective map
$X^r \twoheadrightarrow Y$ factors through $B\otimes_A X^r$, showing that
$Y$ is in $\Fac (B\otimes_A X)$.
\end{proof}

\begin{proposition} \label{surj nice}
Let $\phi:A\to B$ be a surjective morphism of finite-dimensional $k$-algebras.
\begin{enumerate}[\rm(a)]
\item \label{torsion pair}
If $(\TT,\FF)$ is a torsion pair in $\mod A$, then $(\overline{\TT},\overline{\FF})$ is a torsion pair in $\mod B$.
\item \label{+-over}
$\phi_+=\overline{(-)}=\phi_-$.
\item \label{split}
$\overline{(-)}\circ\phi^+=\id_{\tors B}=\overline{(-)}\circ\phi^-$.
\item \label{surj morph}
$\overline{(-)}: \tors A \to \tors B$ is a surjective morphism of complete lattices.
\item \label{maxmin}
For any $\S\in\tors B$, the set $\{\TT\in\tors A\mid\overline{\TT}=\S\}$ is the interval $[\phi^-(\S), \phi^+(\S)]$ in $\tors A$. Therefore $\smash{\pidown = \phi^-}$ and $\smash{\piup = \phi^+}$.
\end{enumerate}
\end{proposition}

\begin{proof}
\eqref{torsion pair} Since $\overline{\TT}\subseteq\TT$ and $\overline{\FF}\subseteq\FF$, we have $\Hom_B(\overline{\TT},\overline{\FF})=0$.
For any $X\in\mod B$, take an exact sequence $0\to T\to X\to F\to0$ with $T\in\TT$ and $F\in\FF$.
Since $\mod B$ is closed under submodules and factor modules in $\mod A$, we have
$T\in\overline{\TT}$ and $F\in\overline{\FF}$. Thus the assertion follows.

\eqref{+-over} The equation $\phi_-(\TT)=\overline{\TT}$ is clear.
Let $\FF:=\TT^{\perp_A}$.
Then $\{Y\in\mod B\mid{}_AY\in\FF\}=\overline{\FF}$ holds.
Thus $\phi_+(\TT)={}^{\perp_B}\overline{\FF}=\overline{\TT}$ holds by \eqref{torsion pair}.

\eqref{split} Suppose $\S\in\tors B$.
Since $(B\otimes_A-)\circ{}_A(-)=\id_{\mod B}$, we have ${}_A\S\subseteq\phi^+(\S)$.
Thus by the definition of $\overline{(-)}$, we have $\S\subseteq\overline{\phi^+(\S)}$.
On the other hand, since $(\overline{(-)},\phi^+)=(\phi_+,\phi^+)$ is an adjoint pair by
Proposition~\ref{relating the maps}\eqref{adj+}, we have $\overline{\phi^+(\S)}\subseteq\S$
by Proposition~\ref{adjoint}.
Thus $\overline{\phi^+(\S)}=\S$ holds. 

The adjoint pair $(\phi^-,\overline{(-)})=(\phi^-,\phi_-)$ gives $\overline{\phi^-(\S)}\supseteq\S$. We have \[\S^{\perp_B}\subseteq \FF := \{X \in \mod A \mid \Hom_A(B, X) \in \S^{\perp_B}\},\] and hence $\overline{\phi^-(\S)} = \mod B \cap {}^{\perp_A} \FF\subseteq {}^{\perp_B}(\S^{\perp_B}) = \S$.

\eqref{surj morph}
By Theorem~\ref{relating the maps}\eqref{meet semi homs}\eqref{join semi homs}, $\overline{(-)}$ is a morphism of complete lattices.
By \eqref{split}, $\overline{(-)}$ is surjective.

\eqref{maxmin} Suppose $\overline\TT=\S$.
Then in particular, $\phi_+(\TT)=\overline\TT\subseteq\S$, so since $(\phi_+,\phi^+)$ is an adjoint pair by Proposition~\ref{relating the maps}\eqref{adj+}, we have $\TT\subseteq\phi^+(\S)$.
Similarly, since $\S\subseteq\phi_-(\TT)$ and $(\phi^-,\phi_-)$ is an adjoint pair by Proposition~\ref{relating the maps}\eqref{adj-}, we have $\phi^-(\S)\subseteq\TT$.
\end{proof}

As we saw, when $\phi: A \to B$ is surjective, $\phi_- = \phi_+ = \overline{(-)}$ is automatically a morphism of complete lattice. We give an open problem about $\phi^+$ and $\phi^-$, which are much more difficult to understand.

\begin{problem}\label{phi upper +- question}
 Characterize the surjective morphisms $\phi: A \twoheadrightarrow B$ of $k$-algebras for which $\phi^+:\tors B\to\tors A$ and $\phi^-:\tors B\to\tors A$ are morphisms of complete lattices.
\end{problem}

We know that when $\phi: A \to B$ is surjective, then $\phi_- = \phi_+ = \overline{(-)}$ preserves functorial finiteness. A question  of interest is the following one.

\begin{problem} \label{ffpres}
 For a morphism $\phi: A \to B$ of $k$-algebras, for each of $\phi^+$, $\phi^-$, $\phi_+$ and $\phi_-$, characterize when they preserve functorial finiteness.
\end{problem}

We give an example of a non-$\tau$-tilting finite algebra such that
Problem~\ref{phi upper +- question} has a positive answer,
which also shows some difficulty to solve the problem in general. Moreover, this example shows that the answer to Problem \ref{ffpres} is not always positive.

\begin{example}\label{kronecker}
Let $k$ be an algebraically closed field and $Q_m$ the $m$-Kronecker quiver 
\[\xymatrix{2\ar@/^/[r]^{a_1}|{}="A" \ar@/_/[r]_{a_m}|{}="B"  &1 \ar@{..}"A";"B"}  \] for $m \geq 2$ and $A_m=kQ_m$.
The Hasse quiver of $\sttilt A_m\cong\ftors A_m$ is given by
\makebox[\textwidth]{%
$
\xymatrix@C=.5cm@R=.3cm{&&&P_1\ar[drrrr]&&&&\\
 A_m=P_1\oplus P_2\ar[dr]\ar[urrr] &&&&&&& 0 \\
&P_2\oplus P_3\ar[r]&P_3\oplus P_4\ar[r]&\cdots\ar[r]&I_2\oplus I_3\ar[r]&I_1\oplus I_2\ar[r]&I_1\ar[ur]}
$}
where $P_{2i+1}:=\tau^{-i}(A_m e_1)$ and $P_{2i+2}:=\tau^{-i}(A_m e_2)$ are preprojective modules, and $I_{2i+1}:=\tau^i(\D(e_2 A_m))$ and $I_{2i+2}:=\tau^i(\D(e_1 A_m))$ are preinjective modules where $\D = \Hom_k(-,k)$. See also Example \ref{kronecker2} for a more detailed description of $\Hasse (\tors A)$ when $m = 2$.

Let $B := A_2/(a_2)$. Then we have $\overline{\mod A_2}=\mod B$, $\overline{\Fac P_1}=\add P_1$ and $\overline{0}=0$.

Any torsion class $\TT$ corresponding to the preprojective $\tau$-tilting modules except $A_2$
satisfies $\overline{\TT}=\S:=\add(B e_2 \oplus S_2)$.
In this case, $\phi^+(\S)=\Fac(P_2\oplus P_3)$ belongs to $\ftors A_2$, but $\phi^-(\S) \notin \ftors A_2$. Indeed, $\phi^-(\S) = \T(X_{(1,0)})$ where $X_{(1,0)}$ is the regular module of dimension vector $(1,1)$, $a_1$ acting as $1$ and $a_2$ acting as $0$.

Similarly, any torsion class $\TT$ corresponding to the preinjective support $\tau$-tilting modules except $0$
satisfies $\overline{\TT}=\S':=\add S_2$.
In this case $\phi^-(\S')=\add S_2 \in \ftors A_2$, but $\phi^+(\S') \notin \ftors A_2$. Indeed, $\phi^+(\S')$ consists of $X \in \mod A_2$ such that $a_2 X = e_1 X$, that is $\phi^+(\S') = \T(\{X_{(\lambda,\mu)}\}_{\mu \neq 0})$, using the above notation.
\end{example}

Suppose that $\phi:A\to B$ is a surjective morphism of $k$-algebras. If $A$ is $\tau$-tilting finite, by Proposition \ref{surj nice}\eqref{surj morph}, $\overline{(-)}:\ftors A\to\ftors B$ is surjective as $\ftors A = \tors A$. However, if we drop the assumption that $A$ is $\tau$-tilting finite, then it is not necessarily surjective, as shown by the following example, developed by the second author with Yingying Zhang.

\begin{example}\label{surjective question}
 Keeping the notation of Example \ref{kronecker}, consider the two algebras $A = A_3$ and $B = A_2 = A_3/(a_3)$.
 Then $\overline{(-)}:\ftors A\to\ftors B$ is not surjective. Indeed, consider $\TT \in \ftors A$.
Then $\TT = \Fac T$ for some $T \in \sttilt A$. By immediate inspection of the Auslander-Reiten quiver of $\mod A$ in Example \ref{kronecker}, there are three possibilities, excluding the case $T = 0$ or $T = A$:
 \begin{itemize}
  \item $T = P_1 = S_1$. In this case, $\overline \TT = \add S_1$.
  \item $T = I_\ell \oplus I_{\ell+1}$ for $\ell \geq 0$ (with $I_0 = 0$). In this case, $\overline \TT = \TT \cap \mod B = \add S_2$. Indeed $\TT = \add (I_i)_{i \leq \ell}$ and, for $i > 0$, $I_i \in \mod B$ if and only if $a_3 I_i = 0$ if and only if $i = 1$ and hence $I_i = I_1 = S_2$.
  \item $T = P_\ell \oplus P_{\ell+1}$ for $\ell \geq 2$. In this case, $\overline \TT = \Fac P_2$ holds. Indeed $\TT$ contains all indecomposable $A$-modules except $P_i$ for $i < \ell$, and the result follows by a similar argument as above.
 \end{itemize}
 So the image of $\overline{(-)}$ consists of $\mod B$, $\Fac P_2$, $\add S_1$, $\add S_2$ and $0$.
\end{example}

\subsection{Algebraic lattice quotients}\label{general sec}

We are now interested in lattice quotients of the form $\tors A \twoheadrightarrow \tors (A/I)$. We recall that the congruence corresponding to such a lattice quotient is called \emph{algebraic}.

 We summarize some results of Section \ref{gen-res} in lattice-theoretical language in the following result.
\begin{theorem}\label{eta op}
Let $A$ be a finite-dimensional $k$-algebra.
\begin{enumerate}[\rm(a)]
\item For any $I\in\Ideals A$, the map $\TT\mapsto\TT\cap\mod(A/I)$ is a surjective morphism of complete lattices from $\tors A$ to $\tors(A/I)$.
\item The congruence $\Theta_I$ inducing $\tors A \twoheadrightarrow \tors(A/I)$ is an arrow-determined complete congruence.
\item The map $\eta_A:\Ideals A\to\Con(\tors A)$ sending $I$ to $\Theta_I$ is a morphism of complete join-semilattices.
\item The class of $\tau$-tilting finite algebras is closed under taking factor algebras.
\end{enumerate}
\end{theorem}

\begin{proof}
 (a) This is Proposition~\ref{surj nice}(d).

 (b) It is complete by (a). Then it is arrow-determined by Theorem~\ref{semidistributiveandbialg}(b) and Proposition \ref{arrdeteq}.

 (c) Let $\II \subseteq \Ideals A$. We will write $\sum\II$ for $\sum_{I \in \II} I$.
For $I \in \II$, we have $\mod(A/\sum \II)\subseteq\mod(A/I)$. So if $\TT, \UU \in \tors A$ satisfy $\TT \equiv_{\Theta_I} \UU$ (that is $\TT \cap \mod (A/I) = \UU \cap \mod (A/I)$), they also satisfy $\TT \cap \mod (A/\sum \II) = \UU \cap \mod (A/\sum \II)$, so $\TT \equiv_{\Theta_{\sum \II}} \UU$. We proved that $\Theta_I \le \Theta_{\sum \II}$ for any $I \in \II$. So $\Join_{I \in \II} \Theta_I \le \Theta_{\sum \II}$. In the rest, we prove the opposite inequality.

 For $I \in \Ideals A$ and $\TT \in \tors A$, let $\TT^{ I} := \smash{\piup_{\Theta_I}(\TT)}$ for simplicity. By Proposition \ref{surj nice}\eqref{maxmin}, we get
 \[\TT^{ I} = \{X \in \mod A \mid (A/I) \otimes_A X \in \TT \cap \mod (A/I) \} = \{X \in \mod A \mid X/IX \in \TT\}.\]

 Let now $I, J \in \Ideals A$. For $\TT \in \tors A$, we have
 \begin{align*}(\TT^{ I})^{ J} &= \{X \in \mod A \mid (X/JX)/(I (X/JX)) \in \TT\} \\ &= \{X \in \mod A \mid X/(I+J)X \in \TT\} = \TT^{ (I+J)}.\end{align*}
 Therefore, if $\TT, \UU \in \tors A$ satisfy $\TT \equiv_{\Theta_{I + J}} \UU$, we have $(\TT^{ I})^{ J} = (\UU^{ I})^{ J}$. So $\TT^{ I} \equiv_{\Theta_J} \UU^{ I}$. Finally, we get the sequence
  \[\TT \equiv_{\Theta_I} \TT^{ I} \equiv_{\Theta_J} \UU^{ I} \equiv_{\Theta_I} \UU,\]
 so $\TT \equiv_{\Theta_I \join \Theta_J} \UU$. We have proved that $\Theta_I \join \Theta_J = \Theta_{I+J}$.

 As $A$ is finite-dimensional, there exists $\II' \subseteq \II$ finite such that $\sum \II' = \sum \II$. So
 \[\Theta_{\sum \II} = \Theta_{\sum \II'} = \Join_{I \in \II'} \Theta_I \le \Join_{I \in \II} \Theta_I.\]

 (d) This is an immediate consequence of (a).
\end{proof}

Thanks to Theorem~\ref{eta op}, we have a morphism of complete join-semilattices $\eta_A: \Ideals A \to \Com(\tors A), I \mapsto \Theta_I$. As this map is usually not surjective, and as the case of lattice quotients coming from algebra quotients is of particular interest, we study the image $\AlgCon A$ of $\eta_A$. As a consequence of Theorem \ref{eta op}, we get:

\begin{theorem} \label{algc}
 The set $\AlgCon A$ of algebraic congruences is a complete join-sublattice of $\Con(\tors A)$, of $\Com(\tors A)$, and of $\Comb(\tors A)$. Hence it is a complete lattice.
\end{theorem}

\begin{proof}
 By Theorem \ref{eta op}(c), $\eta_A$ is a morphism of complete join-semilattices. Hence its image $\AlgCon A$ is a complete join-sublattice of $\Con(\tors A)$, and hence itself a complete lattice. Consider $\II \subseteq \Ideals A$. We know that $\Theta_{\sum \II}$ is the smallest congruence that is bigger than all $\Theta_I$ for $I \in \II$. Additionally, by Theorem \ref{eta op}(b), $\Theta_{\sum \II} \in \Comb(\tors A) \subseteq \Com(\tors A) \subseteq \Con(\tors A)$ so $\Theta_{\sum \II}$ is also the smallest complete congruence and the smallest arrow-determined complete congruence that is bigger than all $\Theta_I$ for $I \in \II$. So $\AlgCon A$ is also a complete join-sublattice of $\Com(\tors A)$ and $\Comb(\tors A)$.
\end{proof}

Recall that by Proposition \ref{combmeetsl}, $\Comb(\tors A)$ is a complete meet-sublattice of $\Com(\tors A)$ which is in turn a complete meet-sublattice of $\Con(\tors A)$.
In both case, it is clear that they are not complete join-sublattices, so the three statements of Theorem \ref{algc} are not just obtained by composition of morphisms of complete join-semilattices.

Note that $\eta_A$ is not necessarily a morphism of lattices:

\begin{example}
Let $A$ be the path algebra of the quiver \[\xymatrix@C1em{1&2\ar[l]_a\ar[r]^b&3},\] and let $I_1:= (a)$ and $I_2:= (b)$. Then $I_1\cap I_2=0$ holds. Notice that $\Hasse(\tors A)$ contains an arrow $\Fac P_2 \to \Fac (P_2/S_1 \oplus P_2/S_3)$, that is contracted by $\Theta_{I_1}$ and $\Theta_{I_2}$ hence by $\Theta_{I_1} \meet \Theta_{I_2}$. So we have $\Theta_{I_1\cap I_2} = \Theta_0 \neq\Theta_{I_1}\meet \Theta_{I_2}$.
This example also shows that $\AlgCon A$ is not a sublattice of $\Con(\tors A)$ since it is easy to check that $\Theta_{I_1}\meet \Theta_{I_2}$ is not an algebraic congruence.
\end{example}

We get the following important characterization of an algebraic congruence~$\Theta_I$ in terms of bricks. As $\mod (A/I)$ is a full subcategory of $\mod A$, we naturally identify $\brick (A/I)$ with the subset $\{S \in \brick A \mid IS = 0\}$ of $\brick A$.

 \begin{theorem} \label{annihilate contract0}
  Let $A$ be a finite-dimensional $k$-algebra and $I \in \Ideals A$. Then the following hold:
  \begin{enumerate}[\rm(a)]
   \item An arrow $q$ in $\Hasse(\tors A)$ is not contracted by $\Theta_I$ if and only if $S_q$ is in $\mod(A/I)$. Moreover, in this case, it has the same label in $\Hasse(\tors A)$ and $\Hasse(\tors (A/I))$.
   \item Consider two torsion classes $\UU \subseteq \TT$ in $\mod A$. We have $\TT \equiv_{\Theta_I} \UU$ if and only if, for every brick $S$ in $\TT \cap \UU^\perp$, $IS \neq 0$.
  \end{enumerate}
 \end{theorem}

 We start by a lemma.
 \begin{lemma} \label{anncl}
  Under the assumptions of \textup{Theorem \ref{annihilate contract0}(b)}, the bricks in $\overline \TT \cap \overline \UU^\perp$ are exactly the bricks of $\mod (A/I)$ that are in $\TT \cap \UU^\perp$.
 \end{lemma}

 \begin{proof}
  Recall that $\overline \TT = \TT \cap \mod(A/I)$ and $\overline \UU = \UU \cap \mod(A/I)$. It is immediate that $\overline \UU^\perp = \UU^\perp \cap \mod(A/I)$
and $\brick(A/I) = \brick A \cap \mod (A/I)$, so the result follows.
 \end{proof}

 \begin{proof}[Proof of Theorem~\ref{annihilate contract0}]
  (b) By definition, $\TT \equiv_{\Theta_I} \UU$ if and only if $\overline \TT = \overline \UU$. According to Theorem
\ref{difftors2}(a), this holds if and only if $\smash{\overline \TT \cap \overline \UU^\perp}$ contains no brick. So, by Lemma \ref{anncl}, $\TT \equiv_{\Theta_I} \UU$ if and only if no brick of $\TT \cap \UU^\perp$ is in $\mod(A/I)$, and the result follows.

  (a) Let $q: \TT\to\UU$ be an arrow in $\Hasse(\tors A)$. By definition, $S_q$ is the unique brick in $\TT \cap \UU^\perp$. Hence, by (b), $q$ is contracted, that is $\overline \TT = \overline \UU$, if and only if $I S_q \neq 0$, if and only if $S_q \notin \mod(A/I)$. If it is not the case, according to Theorem Lemma \ref{anncl}, $\smash{\overline \TT \cap \overline \UU^\perp}$ contains only the brick $S_q$, hence the arrow $\overline \TT \to \overline \UU$ is labelled by $S_q$.
 \end{proof}

Recall that the lattice $\Con L$ of congruences on a lattice $L$ has $\Phi\le\Theta$ if and only if for $x, y \in L$, if $x \equiv_\Phi y$ implies $x \equiv_\Theta y$.
When $\Phi$ and $\Theta$ are arrow-determined, $\Phi\le\Theta$ if and only if the set of Hasse arrows contracted by $\Phi$ is contained in the set of arrows contracted by $\Theta$.
As an immediate consequence of Theorems~\ref{eta op} and \ref{annihilate contract0}, we have the following characterization of $\AlgCon A$, the restriction of $\Con(\tors A)$ to algebraic congruences.

\begin{corollary}\label{contain refine}
Let $A$ be a finite-dimensional $k$-algebra.
Then $I, J \in \Ideals A$ satisfy $\Theta_I\le\Theta_J$ in $\AlgCon A$ if and only if $\brick(A/I)\supseteq\brick(A/J)$.
\end{corollary}

\begin{proof}
  First, if $\brick(A/I)\supseteq\brick(A/J)$, by Theorem \ref{annihilate contract0}(b), arrows contracted by $\Theta_I$ are also contracted by $\Theta_J$, hence we get $\Theta_I\le\Theta_J$ as $\Theta_I$ and $\Theta_J$ are arrow-determined by Theorem~\ref{eta op}(b).

  Suppose now that $\Theta_I\le\Theta_J$ and let $S \in \brick(A/J)$. By Theorem \ref{difftors2}(c), there exists an arrow $\TT \to \UU$ in $\Hasse(\tors A)$ labelled by $S$.
  By Theorem \ref{annihilate contract0}(b), $S \in \brick(A/J)$ implies that $\TT \not\equiv_{\Theta_J} \UU$, hence, as $\Theta_I\le\Theta_J$, $\TT \not\equiv_{\Theta_I} \UU$ so, again by Theorem \ref{annihilate contract0}(b), there is a brick in $\TT \cap \UU^\perp$ that is in $\mod (A/I)$. As $S$ is the only brick in $\TT \cap \UU^\perp$, $S \in \brick (A/I)$.
\end{proof}

We now relate $\AlgCon A$ and $\AlgCon (A/I)$ for an ideal $I$ of $A$.
\begin{proposition} \label{algconquot}
 Let $A$ be a finite-dimensional $k$-algebra and $I \in \Ideals A$. Then there exist two unique maps $\iota_I$ and $\varepsilon_I$ making the following diagram commutative:
 \[
  \xymatrix{
   \Ideals(A/I) \ar@{^{(}->}[r]^-{\phi^{-1}} \ar@{->>}[d]_{\eta_{A/I}} & \Ideals A \ar@{->>}[r]^-{J \mapsto J+I} \ar@{->>}[d]_{\eta_A} & \Ideals(A/I) \ar@{->>}[d]^{\eta_{A/I}} \\
   \AlgCon (A/I) \ar@{^{(}->}[r]_-{\iota_I}  &  \AlgCon A \ar@{->>}[r]_-{\varepsilon_I} & \AlgCon (A/I).
  }
 \]
 where $\phi: A \to A/I$ is the canonical surjection. Moreover,
 \begin{enumerate}[\rm(a)]
  \item $\varepsilon_I \circ \iota_I = \id_{\AlgCon(A/I)}$ ;
  \item $\iota_I \circ \varepsilon_I(\Theta) = \Theta \join \Theta_I$ for any $\Theta \in \AlgCon A$ ;
  \item $\Image \iota_I = [\Theta_I, \Theta_A]$ ($\Theta_A$ identifies all torsion classes) ;
  \item $\iota_I$ is a morphism of complete lattices ;
  \item $\varepsilon_I$ is a morphism of complete join-semilattices.
 \end{enumerate}
\end{proposition}

\begin{proof}
 Let $J$ be an ideal of $A/I$. The congruence $\eta_{A}(\phi^{-1}(J))$ corresponds to the surjective complete lattice morphism
 \[\tors A \twoheadrightarrow \tors (A/I) \twoheadrightarrow \tors ((A/I)/J) = \tors (A/\phi^{-1}(J))\]
 so it only depends on $\eta_{A/I}(J)$, so $\iota_I$ exists. As $\eta_{A/I}$ is surjective, $\iota_I$ is unique.

 Suppose that $J_1, J_2 \in \Ideals A$ satisfy $\eta_A(J_1) = \eta_A(J_2)$. By Corollary \ref{contain refine}, this is equivalent to $\brick (A/J_1) = \brick (A/J_2)$. This implies
 \begin{align*}\brick (A/(I+J_1)) &= \{S \in \brick (A/J_1) \mid I S = 0\} \\ &= \{S \in \brick (A/J_2) \mid I S = 0\} = \brick (A/(I+J_2)),\end{align*}
 so by Corollary \ref{contain refine} again, $\eta_{A/I}(J_1 +  I) = \eta_{A/I}(J_2 +  I)$ and $\varepsilon_I$ exists and is unique as before.

 (a) As the composition of the two maps of the upper row is the identity of $\Ideals(A/I)$, and $\eta_{A/I}$ is surjective, $\varepsilon_I \circ \iota_I = \id_{\AlgCon(A/I)}$.

 (b) For $J \in \Ideals A$, $\phi^{-1}(J + I) = J + I$, so, as $\eta_A$ is a morphism of complete join-semilattices, $\eta_A(\phi^{-1}(J + I)) = \eta_A(J) \join \eta_A(I) = \eta_A(J) \join \Theta_I$. On the other hand, using the commutative diagram, $\eta_A(\phi^{-1}(J + I)) = \iota_I(\epsilon_I(\eta_A(J)))$, so the assertion follows as $\eta_A$ is surjective.

 (c) This is a clear consequence of (b).

 (d) By (a) and (c), $\iota_I$ is an inclusion, with image a complete sublattice, hence $\iota_I$ is a morphism of complete lattices.

 (e) By Theorem \ref{eta op}(c), $\eta_A$ and $\eta_{A/I}$ are both morphisms of complete join-semilattices. Moreover, it is elementary that  $J \mapsto J + I$ is a morphism of complete join-semilattices. It follows easily that $\varepsilon_I$ is a morphism of complete join-semilattices.
\end{proof}

\begin{remark}
 In Proposition \ref{algconquot}, $\varepsilon_I$ is not a morphism of lattices in general. For example, consider the Kronecker quiver as in Example \ref{kronecker2}. Let $I = (a)$, $J = (b)$ and $J' = (a-b)$. As $\eta_A$ is a morphism of complete join-semilattices, we get easily $\varepsilon_I(\Theta_J) = \varepsilon_I(\Theta_{J'}) = \Theta_{(b)}$. On the other hand, $J S_{(1:0)} = 0$ and $J' S_{(1:1)} = 0$. It is immediate that the only ideal that annihilates both $S_{(1:0)}$ and $S_{(1:1)}$ is $0$, so, by Theorem \ref{annihilate contract0}, $\Theta_J \meet \Theta_{J'} = 0$. Finally, $\varepsilon_I(\Theta_J \meet \Theta_{J'}) = 0 \neq \Theta_{(b)} = \varepsilon_I(\Theta_J) \meet \varepsilon_I(\Theta_{J'})$.
\end{remark}

We get the following corollary of Theorem~\ref{annihilate contract0}.
\begin{corollary}\label{Istar}
Let $A$ be a finite-dimensional $k$-algebra and $I \in \Ideals A$. Then the following are equivalent:
\begin{enumerate}[\rm(i)]
\item $I \subseteq I_0 := \bigcap_{S \in \brick A} \ann S$ where $\ann S := \{a \in A \mid aS = 0\}$;
\item $\eta_A(I)$ is the trivial congruence;
\item The map $\TT\mapsto\TT\cap\mod(A/I)$ is an isomorphism from $\tors A$ to $\tors(A/I)$.
\item The maps $\iota_I$ and $\varepsilon_A$ of Proposition \ref{algconquot} are inverse of each other.
\end{enumerate}
Moreover, \textup{(i)}, \textup{(ii)}, \textup{(iii)}, \textup{(iv)} imply:
\begin{enumerate}[\rm(i)] \rs{4}
\item The lattices $\AlgCon(A/I)$ and $\AlgCon A$ are isomorphic.
\end{enumerate}
and \textup{(v)} implies \textup{(i)}, \textup{(ii)}, \textup{(iii)}, \textup{(iv)} if $A$ is $\tau$-tilting finite.

In particular, $I_0$ is the maximum of $\Ideals A$ satisfying each of these properties.
\end{corollary}

\begin{proof}
(i) $\Leftrightarrow$ (ii) It is an immediate consequence of Corollary~\ref{contain refine}.

(ii) $\Leftrightarrow$ (iii) It is true by definition of $\eta_A(I)$.

(ii) $\Leftrightarrow$ (iv) By Proposition \ref{algconquot}(a) and (b), $\iota_I$ and $\varepsilon_I$ are inverse of each other if and only if $\Theta_I = \eta_A(I)$ is trivial.

(iv) $\Rightarrow$ (v) It is trivial.

(v) $\Leftarrow$ (iv) If $A$ is $\tau$-tilting finite, hence $\# \tors A < \infty$ by Theorem \ref{tors is complete lattice}, we have $\# \AlgCon A \leq \# \Con A < \infty$. Therefore, if $\AlgCon(A/I) \cong \AlgCon A$, (iv) holds by Proposition \ref{algconquot}(a).
\end{proof}

We get the following corollary, extending a result of \cite{EJR}.

\begin{corollary} \label{EJRcoro}
Let $A$ be a finite-dimensional $k$-algebra and $Z$ the center of $A$.
Then for any $I \subset A \rad Z$, $\eta_A(I)$ is the trivial congruence.
\end{corollary}

\begin{proof}
Fix $a\in\rad Z$.
For any $A$-module $X$, we have an endomorphism $a:X\to X$ which is not an isomorphism. If $X$ is a brick, this has to be zero.
Thus any $S \in \brick A$ is annihilated by $a$, so $\eta_A(\rad Z)$ is trivial by Corollary \ref{Istar}.
\end{proof}

Corollary \ref{EJRcoro} immediately implies that if $I \subset A \rad Z$, the projection $\sttilt A \twoheadrightarrow \sttilt (A/I)$ is an isomorphism, which is the original result of \cite{EJR}.

By Theorem \ref{eta op} and Corollary \ref{Istar}, there is a surjective complete lattice morphism $\Ideals(A/I_0) \to \AlgCon A$. Notice that it is not necessarily an isomorphism:
\begin{example}
  Consider
\[A:=k \left(\xymatrix{ 1 \ar@(ul,dl)[]_u \ar@/_/[r]_x & 2 \ar@(dr,ur)[]_v \ar@/_/[l]_y } \right) / (yx, xy, u^2, v^2, xvy, vyu, yux, uxv).\]
Then $A$ has $10$ support $\tau$-tilting modules. We depicted $\Hasse(\tors A)$ and its brick labelling in Figure \ref{exnotcorr}.
\begin{figure}
\[\xymatrix@C=3.5cm@R=1.25cm@!=0cm{
 &\tb{\Pun}{\Pdeux} \ar[dl]|-{\cirl{\Sa}} \ar[dr]|-{\cirl{\Sb}} \\
 \tb{\Punm}{\Pdeux} \ar[dd]|-{\cirl{\Sduu}} & & \tb{\Pun}{\Pdeuxm} \ar[dd]|-{\cirl{\Sudd}} \\ \\
 \tb{\Punm}{\Sddu} \ar[dd]|-{\cirl{\Sdu}} & & \tb{\Suud}{\Pdeuxm} \ar[dd]|-{\cirl{\Sud}} \\ \\
\tb{\Sdd}{\Sddu} \ar[dd]|-{\cirl{\Sddu}} & & \tb{\Suud}{\Suu} \ar[dd]|-{\cirl{\Suud}} \\ \\
 \ta{\Sdd} \ar[dr]|-{\cirl{\Sb}} & & \ta{\Suu} \ar[dl]|-{\cirl{\Sa}} \\
 & 0}
\]
\caption{An example where $\Ideals(A/I_0) \not\cong \AlgCon A$} \label{exnotcorr}
\end{figure}
Moreover, it is easy to see that $I_0 = 0$, and, however, there is a family of ideals indexed by $\mathbb{P}^1$: $I_{(\lambda:\mu)} = (\lambda xv + \mu ux)$.
\end{example}

\section{The preprojective algebra and the weak order}\label{pre weak sec}
In this section we give background on the weak order on a Weyl group, on preprojective algebras, and on the connection between the weak order and preprojective algebras.

\subsection{Weak order on Weyl groups}\label{weyl sec}
Let $Q$ be a Dynkin quiver, that is, a quiver whose underlying graph is one of the following simply laced diagrams:
\[\xymatrix@C0.4cm@R0.3cm{
A_n&1\ar@{-}[r]&2\ar@{-}[r]&3\ar@{-}[r]&\ar@{..}[rrr]&&&\ar@{-}[r]&(n-2)\ar@{-}[r]&(n-1)\ar@{-}[r]&n ;\\
&&&&&&&&n\\
D_n&1\ar@{-}[r]&2\ar@{-}[r]&3\ar@{-}[r]&\ar@{..}[rrr]&&&\ar@{-}[r]&(n-2)\ar@{-}[r]\ar@{-}[u]&(n-1) ; \\
&&&4\\
E_6&1\ar@{-}[r]&2\ar@{-}[r]&3\ar@{-}[r]\ar@{-}[u]&5\ar@{-}[r]&6 ;\\
&&&&7\\
E_7&1\ar@{-}[r]&2\ar@{-}[r]&3\ar@{-}[r]&4\ar@{-}[r]\ar@{-}[u]&5\ar@{-}[r]&6 ;\\
&&&&&8\\
E_8&1\ar@{-}[r]&2\ar@{-}[r]&3\ar@{-}[r]&4\ar@{-}[r]&5\ar@{-}[r]\ar@{-}[u]&6\ar@{-}[r]&7.}
\]
The Dynkin quiver $Q$
determines a group called the \newword{Weyl group} $W$ of $Q$, which depends only on the underlying undirected graph of $Q$.
We label the vertices of $Q$ as above
and let $S := \set{s_1,\ldots,s_n}$.
The Weyl group $W$ of $Q$ is the group given by the presentation
\[W=\left\langle S \, \middle| \, \begin{array}{ll} s_i^2=1 & \text{for all }i=1,\ldots,n \\ s_i s_j s_i=s_j s_i s_j & \text{for all $i$ and $j$ adjacent in $Q$} \\ s_i s_j =s_j s_i & \text{for all $i$ and $j$ not adjacent in $Q$}  \end{array} \right\rangle.\]
The best known example of a Weyl group is the symmetric group $\sym_{n+1}$, which is the Weyl group associated to a quiver of type $A_n$.
The generators $s_1,\ldots,s_n$ are the simple transpositions $(1,2)$ through $(n,n+1)$.
We represent each element $\sigma$ of $\sym_{n+1}$ by its \newword{one-line notation} $\sigma(1)\cdots \sigma(n+1)$.
Background on the combinatorics of Coxeter groups can be found in \cite{BB}.

We will call an expression $s_{i_1}\cdots s_{i_k}$ a \newword{word} for $w$ if $w=s_{i_1}\cdots s_{i_k}$ holds in $W$.
The minimal length (number of letters) of a word for $w$ is called the \newword{length} of $w$ and denoted $\ell(w)$.
A word for $w$ having exactly $\ell(w)$ letters is called a \newword{reduced word} for $w$.

The \newword{(right) weak order} on $W$ is the partial order on $W$ setting $v\le w$ if and only if there exists a reduced word $s_{i_1}\cdots s_{i_k}$ for $w$ such that, for some $j\le k$, the word $s_{i_1}\cdots s_{i_j}$ is a reduced word for $v$.
Importantly for our purposes, the weak order on $W$ is a lattice
(see, for example, \cite[Theorem~3.2.1]{BB}). Arrows of $\Hasse W$ are of the form $ws \to w$ for $s \in S$ whenever $\ell(ws) > \ell(w)$.

As an example, we describe the weak order on permutations.
An \newword{inversion} of $\sigma \in \sym_{n+1}$ is a pair $(\sigma(i),\sigma(j))$ such that $1 \le i < j \le n+1$ and $\sigma(i) > \sigma(j)$.
The length of $\sigma$ is the number of inversions of $\sigma$.
The weak order on $\sym_{n+1}$ corresponds to containment of inversion sets.
Hasse arrows are $\tau \to \sigma$ were $\tau$ is obtained from $\sigma$ by swapping two adjacent entries $\sigma(i) < \sigma(i+1)$.
We illustrate the weak order on $\sym_4$ in Figure~\ref{weak S4}.

\begin{figure}
\[\xymatrix@R=.9cm@C=.5cm@!=0cm{
 & & & & & 4321\\
 & & & 4312 \ar@{<-}[urr] & & 4231 \ar@{<-}[u] & & 3421\ar@{<-}[ull] \\
 & 4132 \ar@{<-}[urr] & & 4213 \ar@{<-}[urr] & & 3412 \ar@{<-}[ull] \ar@{<-}[urr] & & 2431 \ar@{<-}[ull] & & 3241 \ar@{<-}[ull] \\
 1432 \ar@{<-}[ur] & & 4123 \ar@{<-}[ul] \ar@{<-}[ur] & & 2413 \ar@{<-}[ul] \ar@{<-}[urrr] & & 3142 \ar@{<-}[ul] & & 2341 \ar@{<-}[ul] \ar@{<-}[ur] & & 3214 \ar@{<-}[ul] \\
 & 1423 \ar@{<-}[ul] \ar@{<-}[ur] & & 1342 \ar@{<-}[ulll] \ar@{<-}[urrr] & & 2143 \ar@{<-}[ul] & & 3124 \ar@{<-}[ul] \ar@{<-}[urrr] & & 2314 \ar@{<-}[ul] \ar@{<-}[ur] \\
  & & & 1243 \ar@{<-}[ull] \ar@{<-}[urr] & & 1324 \ar@{<-}[ull] \ar@{<-}[urr] & & 2134 \ar@{<-}[ull] \ar@{<-}[urr] \\
& & & & & 1234 \ar@{<-}[ull] \ar@{<-}[u] \ar@{<-}[urr]
}
\]
\caption{The weak order on $\sym_4$}
\label{weak S4}
\end{figure}

We are interested in lattice congruences on and lattice quotients of the weak order.
Background on congruences and quotients of the weak order can be found in \cite{regions,regions2}.

 There is a hyperplane arrangement associated to $W$ which will play a role in
what follows.  Specifically, $\mathbb R^n$ can be equipped with a
positive-definite symmetric bilinear form such that each element $s_i$
acts as reflection in a hyperplane $H_i$.  It follows that any element of
the form $ws_iw^{-1}$ acts as reflection in some hyperplane.
It is less obvious, but still known, that every element of $W$ that acts as a reflection is conjugate to some $s_i$.
The collection of all these hyperplanes is called the \newword{reflection arrangement}.

The complement of the reflection arrangement is a union of open cones.  We
refer to the closure of each of these cones as a \newword{chamber}.  We
can view the collection of the chambers and their faces as a fan $\FF$, the
\newword{Coxeter fan}. Fix a chamber
$D$ in the Coxeter fan whose facets are given by
$H_1,\dots,H_n$.  We call this cone the \newword{dominant chamber}.
There is a natural bijection between the chambers of the Coxeter fan and the elements of $W$, given by sending $w$ to the chamber $wD$.

For any lattice congruence $\Theta$ on $W$, there is a corresponding coarsening of the Coxeter fan:
Since elements of $W$ correspond to chambers, each congruence class is a set of chambers.
The union of the chambers corresponding to a congruence class is itself a convex cone, and the set of such cones is a complete fan $\FF_\Theta$ that coarsens $\FF$.
(See \cite[Theorem~1.1]{con_app}.)
By definition, the maximal cones of $\FF_\Theta$ are in bijection with the elements of the quotient $W/\Theta$.
In fact, the arrows in the Hasse quiver of $W/\Theta$ correspond bijectively to the pairs of adjacent maximal cones in $\FF_\Theta$.
(This result can be obtained by concatenating \cite[Theorem~1.1]{con_app} and \cite[Proposition~3.3]{con_app} or by interpreting \cite[Proposition~9-8.6]{regions} in the special case of the weak order on $W$.)
As an immediate consequence, we have the following proposition, in which a fan is said to be \newword{simplicial} if for each of its maximal cones, the facet normals for the cone are linearly independent.

\begin{proposition}\label{hasse-simplicial}
Given a lattice congruence $\Theta$ on $W$, the quotient $W/\Theta$ is Hasse-regular if and only if the fan $\FF_\Theta$ is simplicial.
\end{proposition}

\subsection{Preprojective algebras and Weyl groups}\label{preproj-weyl} \label{alg Hasse-reg}
Let $Q=(Q_0,Q_1)$ be an acyclic quiver with set of vertices $Q_0$
and set of arrows $Q_1$.
We define a new quiver $\overline{Q}$ by adding a new arrow
$a^*:j\to i$ for each arrow $a:i\to j$ in $Q$.
The \newword{preprojective algebra} of $Q$ is defined as
\[\Pi=\Pi_Q:=k\overline{Q}\left/\middle(\sum_{a\in Q_1}(aa^*-a^*a)\right).\]
Then, up to isomorphism, $\Pi$ does not depend on the choice of orientation of the quiver $Q$.
It is well-known that $\Pi$ is finite-dimensional if and only if $Q$ is
a Dynkin quiver.

Now we assume that $Q$ is a Dynkin quiver, and let $W$ be the corresponding
Weyl group.
For a vertex $i\in Q_0$, we denote by $e_i$ the corresponding idempotent of $\Pi$.
We denote by $I_i$ the two-sided ideal of $\Pi$ generated by the idempotent $1-e_i$.
Then $I_i$ is a maximal left ideal and a maximal right ideal of $\Pi$ since $Q$ has no loops.
For each element $w\in W$, we take a reduced word $w=s_{i_1}\cdots s_{i_k}$
for $w$, and let
\[I_w:=I_{i_1}\cdots I_{i_k}.\]
The following result due to Mizuno is the starting point of this section.

\begin{theorem}\label{Mizuno's theorem}\
\begin{enumerate}[\rm(a)]
\item {\cite[Theorem III.1.9]{BIRS}} $I_w$ does not depend on the choice of the reduced word for $w$.
\item {\cite[Theorem 2.14]{M}} $\Pi$ is $\tau$-tilting finite, and we have bijections
\begin{equation}\label{Mizuno}
W\xrightarrow{\sim}\sttilt\Pi\xrightarrow{\sim}\tors\Pi
\end{equation}
given by $w\mapsto I_w\mapsto \Fac I_w$.
\item {\cite[Theorem 2.21]{M}} The bijections \eqref{Mizuno} give isomorphisms of lattices
\[(W,\le^{\op})\xrightarrow{\sim}
(\sttilt\Pi,\le)\xrightarrow{\sim}(\tors\Pi,\subseteq).\]
\end{enumerate}
\end{theorem}

Note that in \cite{M}, right modules are considered rather than
left modules, which has the consequence
that \cite{M} works with left weak order on $W$
rather than right weak order.

In type $A_3$, the weak order on $W=\sym_4$ is displayed in Figure~\ref{weak S4}.
The corresponding support $\tau$-tilting modules are shown in Figure~\ref{A3 fig}.
\begin{figure}
\[\begin{xy}
(0,0) *+{\tc{\abc}{\bacb}{\cba}}="1",
(-30,-14) *+{\tc{\bc}{\bacb}{\cba}}="2",
(0,-14) *+{\tc{\abc}{\acb}{\cba}}="3",
(30,-14) *+{\tc{\abc}{\bacb}{\ba}}="4",
(-50,-28) *+{\tc{\bc}{\cb}{\cba}}="5",
(-25,-28) *+{\tc{\Sc}{\acb}{\cba}}="6",
(0,-28) *+{\tc{\bc}{\bacb}{\ba}}="7",
(25,-28) *+{\tc{\abc}{\acb}{\Sa}}="8",
(50,-28) *+{\tc{\abc}{\ab}{\ba}}="9",
(-55,-42) *+{\tb{\bc}{\cb}}="10",
(-35,-42) *+{\tc{\Sc}{\cb}{\cba}}="11",
(-10,-42) *+{\tc{\Sc}{\acb}{\Sa}}="12",
(10,-42) *+{\tc{\bc}{\Sb}{\ba}}="13",
(35,-42) *+{\tc{\abc}{\ab}{\Sa}}="14",
(55,-42) *+{\tb{\ab}{\ba}}="15",
(-50,-56) *+{\tb{\Sc}{\cb}}="16",
(-25,-56) *+{\tb{\bc}{\Sb}}="17",
(-0,-56) *+{\tb{\Sa}{\Sc}}="18",
(25,-56) *+{\tb{\Sb}{\ba}}="19",
(50,-56) *+{\tb{\ab}{\Sa}}="20",
(-30,-70) *+{\ta{\Sc}}="21",
(0,-70) *+{\ta{\Sb}}="22",
(30,-70) *+{\ta{\Sa}}="23",
(0,-84) *+{0}="24",
\ar"1";"2",
\ar"1";"3",
\ar"1";"4",
\ar"2";"5",
\ar"2";"7",
\ar"3";"6",
\ar"3";"8",
\ar"4";"7",
\ar"4";"9",
\ar"5";"10",
\ar"5";"11",
\ar"6";"11",
\ar"6";"12",
\ar"7";"13",
\ar"8";"12",
\ar"8";"14",
\ar"9";"14",
\ar"9";"15",
\ar"10";"16",
\ar"10";"17",
\ar"11";"16",
\ar"12";"18",
\ar"13";"17",
\ar"13";"19",
\ar"14";"20",
\ar"15";"19",
\ar"15";"20",
\ar"16";"21",
\ar"17";"22",
\ar"18";"21",
\ar"18";"23",
\ar"19";"22",
\ar"20";"23",
\ar"21";"24",
\ar"22";"24",
\ar"23";"24",
\end{xy}\]
\caption{$\sttilt(\Pi)$ in type $A_3$}\label{A3 fig}
\end{figure}

Recall that a join-irreducible element is called a \newword{double join-irreducible element} if the unique element which it covers is either join-irreducible or the bottom element of the lattice.
Theorem~\ref{double ji} asserts that if $W$ is a finite Weyl group of simply-laced type and $\Theta$ is a lattice congruence on $W$, then the following conditions satisfy the implications $\mathrm{(i)}\Rightarrow\mathrm{(ii)}\Rightarrow\mathrm{(iii)}$:
\begin{enumerate}[\rm(i)]
\item $\Theta$ is an algebraic congruence.
\item $W/\Theta$ is Hasse-regular.
\item There is a set $J$ of double join-irreducible elements such that $\Theta$ is the smallest congruence contracting every element of $J$.
\end{enumerate}
The theorem also asserts that $\mathrm{(iii)}\Rightarrow\mathrm{(i)}$ when $W$ is of type $A$.
At the end of this subsection, we show that (iii) $\Rightarrow$ (ii) and
  (iii) $\Rightarrow$ (i) are not true for the preprojective algebra of type $D_4$.

We now prove Theorem~\ref{double ji}, except for the assertion that is specific to type~$A$, which is proved in Section~\ref{preprojA}.
By the definition, any algebraic lattice quotient of $\tors A$ is $\tors(A/I)$ for some $I\in\Ideals A$.
Thus the quotient is Hasse-regular by Corollary~\ref{Hasse-regular}.
We see that (i) implies (ii).

It is easy to construct non-Hasse-regular quotients of the weak order, which are therefore not algebraic quotients.
For example, in $\sym_4$, each one of the sets $\set{2413 \to 2143}$, $\set{3412 \to 3142}$, and $\set{2413 \to 2143,\,3412 \to 3142}$ is closed under polygonal forcing (see Figure~\ref{weak S4}), so by Proposition \ref{forcing=polygonal}, each defines a lattice congruence.
However, the corresponding quotients are not Hasse-regular
(each has one or more vertices of degree 4 in the Hasse quiver).

The following theorem shows that (ii) implies (iii).

\begin{theorem} \label{two}
Let $W$ be a simply-laced finite Weyl group, and suppose that $\Theta$ is a lattice congruence on $W$ such that $W/\Theta$ is Hasse-regular.
Then there exists a set $\S$ of double join-irreducible elements in $W$ such that $\Theta=\con(\S)$.
\end{theorem}

\begin{proof}
Let $x$ be a join-irreducible element that is maximal in the forcing order among join-irreducible elements contracted by $\Theta$. It suffices to show that $x$ is double join-irreducible.

For $w\in W$, write $D(w)$ for the \newword{right descents} of $w$, the set of simple reflections which can occur as the rightmost letter in a reduced word for $w$.
The set of elements covered by $w$ is $\set{ws_i:s_i\in D(w)}$.

Suppose $x$ is join-irreducible and let $x \to x_*$ be the unique arrow of $\Hasse W$ starting at $x$. 
Thus $x_*=xs_j$ where $s_j$ is the unique element of $D(x)$.
If $D(x_*)$ contains some element $s_i$, then $s_i$ and $s_j$ do not commute (otherwise, $s_i\in D(x)$).
Thus $i$ and $j$ are adjacent in the Dynkin diagram of $W$.
Furthermore, since $s_i\notin D(w)$, there is an arrow $x s_i \to x$ in $\Hasse W$.

Now suppose that $x_*$ is not join-irreducible, and not equal to $e$.
Then $D(x_*)$ contains at least two distinct simple reflections $s_i$ and $s_k$, with $i$ and $k$ each adjacent to $j$ in the Dynkin diagram.
The arrow $x \to x s_j$ is a side arrow in two distinct hexagons, namely the intervals $[xs_js_i,xs_i]$ and $[xs_js_k,xs_k]$, as shown below:
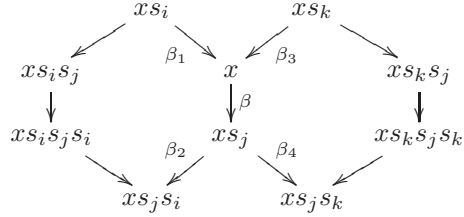
\[\xymatrix@C=5pt@R=11pt{
&xs_i\ar[dl]\ar[dr]_{\beta_1}&&xs_k\ar[dl]^{\beta_3}\ar[dr]&\\
xs_is_j\ar[d]&&x\ar[d]^\beta&&xs_ks_j\ar[d]\\
xs_is_js_i\ar[dr]&&xs_j\ar[dl]_{\beta_2}\ar[dr]^{\beta_4}&&xs_ks_js_k\ar[dl]\\
&xs_js_i&&xs_js_k&
}\]
None of the arrows up from $x$ or down from $xs_j$ in these hexagons are contracted by $\Theta$, because such a contraction would also force the contraction of $x \to xs_j$, contradicting our assumption that $x$ is maximal in forcing order among join-irreducible elements contracted by $\Theta$.
Since these arrows are not contracted, the cone $C$ in $\FF_\Theta$ corresponding to the $\Theta$-class of $x$ and $xs_j$ has walls that contain the walls of the Coxeter fan separating $xD$ from $xs_iD$ and from $xs_kD$ and separating $xs_jD$ from $xs_js_iD$ and $xs_js_kD$.
Thus the normal vectors to $C$ include the four vectors normal to these four walls.
Call these vectors $\beta_1$, $\beta_2$, $\beta_3$, and $\beta_4$, associated to walls in the Coxeter fan (and thus to arrows in the weak order) as indicated in the diagram above.
Also, write $\beta$ for the vector normal to the wall separating $xD$ from $xs_jD$ (which is not a wall of $C$).

All the Coxeter-fan walls associated to the hexagon $[xs_js_i,xs_i]$ contain a common codimension-2 face.
Thus in particular $\beta$ is in the linear span of $\beta_1$ and $\beta_2$.
Similarly, $\beta$ is in the linear span of $\beta_3$ and $\beta_4$.
We have found a nontrivial linear relation on the set $\set{\beta_1,\beta_2,\beta_3,\beta_4}$.
This is a subset of the set of normal vectors to walls of $C$, and we conclude that $C$ is not simplicial.
So Proposition~\ref{hasse-simplicial} implies that $W/\Theta$ is not Hasse-regular.
\end{proof}

We explain in the following example why we cannot have $\mathrm{(iii)}\Rightarrow\mathrm{(i)}$ or $\mathrm{(iii)} \Rightarrow \mathrm{(ii)}$ in other Dynkin types (see also Example \ref{cexiiii} for an easier counterexample to $\mathrm{(iii)}\Rightarrow\mathrm{(i)}$ for a different finite-dimensional algebra):
\begin{example} \label{cexd4}
 We consider the case $D_4$ indexed as in the beginning of the section:
 \[\overline{Q} = \boxinminipage
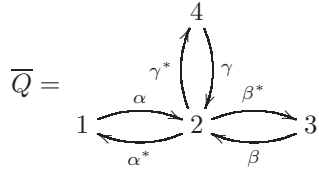
{\xymatrix@R=1.5cm@C=1.5cm@!=0cm{ & 4 \ar@/^/[d]^\gamma \\ 1 \ar@/^/[r]^\alpha & 2 \ar@/^/[l]^{\alpha^*} \ar@/^/[u]^{\gamma^*} \ar@/^/[r]^{\beta^*} &  3\ar@/^/[l]^\beta  }}\]

We consider the bricks
\[S = \boxinminipage{$\xymatrix@C=.15cm@R=.3cm@!=0cm@!=0cm{& 1 \\  & 2 \\ 3 & & 4}$} \quad \text{and} \quad
 S' = \boxinminipage{$\xymatrix@C=.15cm@R=.3cm@!=0cm@!=0cm{1 \\  & 2 \\ & & 3}$} \quad \text{and} \quad S'' = \boxinminipage{$\xymatrix@C=.15cm@R=.3cm@!=0cm@!=0cm{1 \\  & 2 \\ & & 4}$}. \]

Then we claim:
\begin{enumerate}[\rm (a)]

 \item $S$ does not force $S'$, and $S$ does not force $S''$.
 \item Any algebraic congruence contracting $S$ contracts at least one of $S'$ and $S''$.
 \item The smallest congruence contracting $S$ is not algebraic.
 \item Let $w = s_2 s_4 s_3 s_2 s_4 s_3 s_1$. Then $I_w$ is a double join-irreducible element and the corresponding brick is $S$.
\end{enumerate}

Together, these imply that (iii) $\not\Rightarrow$ (i), taking $J = \{S\}$.
Further, one can observe that the quotient by the
  smallest congruence 
  contracting $S$ is not
  Hasse-regular, proving that (iii) $\not\Rightarrow$ (ii). This, of course,
also constitutes a proof that (i) does not hold for this quotient, by Corollary \ref{Hasse-regular}. However, since this is a somewhat involved calculation, we prefer the more conceptual argument for the non-algebraicity of the quotient outlined above, and which we detail below.
\begin{proof}
  (a) 
  For $(\lambda:\mu) \in \mathbf{P}^1(k)$, let $I_{(\lambda:\mu)} := (\lambda \alpha \beta^* + \mu \alpha \gamma^*) \subseteq \Pi$. Then $I_{(0:1)}S \neq 0$ and $I_{(0:1)} S' = 0$. Thus $\smash{\Theta_{I_{(0:1)}}}$ contracts $S$ but does not contract $S'$. So $S$ does not force $S'$. In the same way, using $I_{(1:0)}$, $S$ does not force $S''$.  Alternatively, the fact that $S$ forces neither $S'$ nor $S''$ follows immediately from Theorem \ref{theorem forcing}.

 (b) Since $S' \oplus S'' \in \Fac S$ and $S \in \Sub(S' \oplus S'')$, we have $\ann(S)=\ann(S'\oplus S'')=\ann(S')\cap \ann(S'')$.
 Let $I$ be an ideal of $\Pi$ such that $\Theta_I$ contracts $S$. Then $I \not \subseteq \ann(S)$ by Theorem \ref{annihilate contract0}(a). Thus, at least one of $I \not\subseteq \ann(S')$ and $I \not\subseteq \ann(S'')$ holds. Again by Theorem \ref{annihilate contract0}(a), $\Theta_I$ contracts at least one of $S'$ and $S''$.

 (c) By (a), the smallest congruence contracting $S$ contracts neither $S'$ nor $S''$, so it is not algebraic by (b).

 (d) Let $w_0$ be the longest element of $W$. As all reduced expressions of $w w_0$, which are $s_2 s_4 s_3 s_2 s_1$ and $s_2 s_3 s_4 s_2 s_1$, terminate by $s_2 s_1$, $w w_0$ is double join-irreducible in $W$. So, by Theorem \ref{Mizuno's theorem}, and because $u \mapsto u w_0$ is an anti-automorphism of $W$, we get that $I_w$ is join-irreducible in $\sttilt \Pi$. We easily compute
 \[I_w = \boxinminipage{$\xymatrix@C=.15cm@R=.3cm@!=0cm@!=0cm{1}$} \oplus \boxinminipage{$\xymatrix@C=.15cm@R=.3cm@!=0cm@!=0cm{&1 \\ & 2 \\ 4 & & 3 & & 1 \\ & & 2}$}  \oplus \boxinminipage{$\xymatrix@C=.15cm@R=.3cm@!=0cm@!=0cm{1 \\ & 2 \\ & & 3}$} \oplus \boxinminipage{$\xymatrix@C=.15cm@R=.3cm@!=0cm@!=0cm{1 \\ & 2 \\ & & 4}$}\]
 and the brick labelling the arrow $q$ starting at $I_w$ is $S$ by Proposition \ref{theorem label2}.
\end{proof}
\end{example}

\subsection{The preprojective algebra and the weak order in type \texorpdfstring{$A$}{A}}  \label{preprojA}
We continue our discussion of the preprojective algebra of type $A$ in Section 5 in \cite{IRRT}.
The goal of this section is to provide more combinatorial and algebraic details about the type $A$ case.
Some related results can be found in the recent preprint \cite{Kase}.
Throughout this section, let $\Pi$ be the preprojective algebra of type $A_n$, that is given by the quiver
\[\overline{Q} := \xymatrix@R=.3em@C=3em{1\ar@<2pt>[r]^{x_1}&2\ar@<2pt>[r]^{x_2}\ar@<2pt>[l]^{y_2}&3\ar@<2pt>[r]^{x_3}\ar@<2pt>[l]^{y_3}&\ar@{.}[r]\ar@<2pt>[l]^{y_4}&\ar@<2pt>[r]^-{x_{n-2}}&n-1\ar@<2pt>[r]^-{x_{n-1}}\ar@<2pt>[l]^-{y_{n-1}}&n\ar@<2pt>[l]^-{y_n}}\]
with relations $x_1y_2=0$, $x_iy_{i+1}=y_{i}x_{i-1}$ for $2\le i\le n-1$ and $y_nx_{n-1}=0$.
We identify the corresponding Weyl group $W$, $\sttilt\Pi$ and $\tors\Pi$ by bijections \eqref{Mizuno}.

We denote by $\SS$ the set of \emph{non-revisiting walks} on the double quiver $\overline{Q}$.
By definition, these are walks in $\overline{Q}$ which follow a sequence of arrows either with or against the direction of the arrow, and which do not visit any vertex more than once.
We  identify a walk and its reverse walk (for instance $x_2 y_4^{-1}$ is identified with $y_4 x_2^{-1}$). We also include length $0$ walks starting and ending at the same vertex.

We give a definition of string modules that fits this context. For a more general version, see \cite{WW}. For each $S \in \SS$, there is a \emph{string module} $X_S \in \mod \overline{\Pi}$ satisfying
\begin{itemize}
 \item for $i \in \overline{Q}_0$, $e_i X_S = k$ if $S$ contains $i$ and $e_i X_S = 0$ otherwise;
 \item for $q \in \overline{Q}_1$, $q$ acts as $\id_k$ if $S$ contains $q$ and acts as $0$ otherwise.
\end{itemize}

The set $\UUU$ of \emph{non-revisiting paths} defined in the introduction is the subset of $\SS$ corresponding to uniserial modules.

The main result of this section is the following.

\begin{theorem} \label{brmodpi}\
 \begin{enumerate}[\rm(a)]
  \item Bricks of $\mod \Pi$ are exactly the string modules.
  \item For two bricks $S$ and $S'$, we have $S \forces S'$ if and only if $S$ is a subfactor of $S'$.
 \end{enumerate}
\end{theorem}

We give an easy example.

\begin{example} \label{exforc}
    The Hasse quiver of $(\brick \Pi, \forces)$ for $n=3$ is:
   
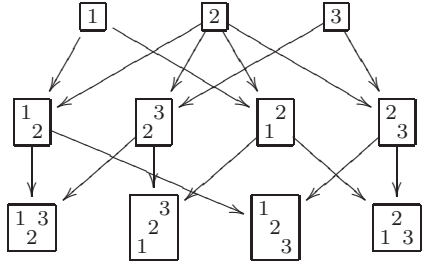
\[
   \xymatrix@R=1.4cm@C=0.8cm@!=0cm{
   &\ta{\Sa}\ar@{->}[dl]\ar@{->}[drrr]&&\ta{\Sb}\ar@{->}[dlll]\ar@{->}[dl]\ar@{->}[dr]\ar@{->}[drrr]&&\ta{\Sc}\ar@{->}[dlll]\ar@{->}[dr]\\
   \ta{\ab}\ar@{->}[d]\ar@{->}[drrrr]&&\ta{\cb}\ar@{->}[dll]\ar@{->}[d]&&\ta{\ba}\ar@{->}[dll]\ar@{->}[drr]&&\ta{\bc}\ar@{->}[dll]\ar@{->}[d]\\
   \ta{\acb}&&\ta{\cba}&&\ta{\abc}&&\ta{\bac}
   }
  \]
\end{example}

Recall that $I_{\cyc}$ is the ideal of $\Pi$ generated by all 2-cycles and define
\[\overline{\Pi}:=\Pi/I_{\cyc}.\]
 \begin{proposition} \label{pipibar}\
  \begin{enumerate}[\rm (a)]
   \item We have $I_\cyc = \bigcap_{S \in \brick \Pi} \ann S$.
   \item There is an isomorphism of lattices \[\overline{(-)}: \tors \Pi \to \tors \overline{\Pi}, \TT \mapsto \TT \cap \mod \overline{\Pi}.\]
   \item We have $\brick \Pi = \brick \overline{\Pi}$.
  \end{enumerate}
 \end{proposition}

For the proof of the proposition, we need an elementary lemma about ideals of~$\Pi$:
\begin{lemma}\label{ideals of Pi}
For $I \in \Ideals \Pi$,
$I=\spanv_k (I\cap\UUU) \oplus(I\cap I_{\cyc})$ as $k$-vector spaces.
\end{lemma}

\begin{proof}
 As we have $\Pi=\spanv_k \UUU \oplus I_{\cyc}$, it suffices to prove $I\subseteq \spanv_k(I\cap\UUU)+I_{\cyc}$.
 Let $i, j \in Q_0$.
 Using relations for $\Pi$, we have $e_i \Pi e_j = p e_j \Pi e_j$ where $p \in \UUU$ is the shortest path from $i$ to $j$ in $\overline{Q}$. As $e_j \Pi e_j$ is a local ring with maximal ideal $e_j I_\cyc e_j$, we get that either $e_i I e_j = e_i \Pi e_j$ or $e_i I e_j \subseteq p e_j I_\cyc e_j \subseteq e_i I_\cyc e_j$. Therefore $e_i I e_j \subseteq  \spanv_k( I \cap \{p\} ) \oplus e_i I_\cyc e_j$. Thus, $I \subseteq \spanv_k (I \cap \UUU) + I_\cyc$ holds.
\end{proof}

 \begin{proof}[Proof of Proposition~\ref{pipibar}]
  (a) Any $x \in \UUU$ is outside of the annihilator of the corresponding uniserial module. As any uniserial module is multiplicity free, hence a brick, we deduce that $\UUU \cap \bigcap_{S \in \brick \Pi} \ann S = \emptyset$. Hence by Lemma \ref{ideals of Pi}, \[\bigcap_{S \in \brick \Pi} \ann S \subseteq I_\cyc.\] Let $\omega := x_1 y_2 + x_2 y_3 + \cdots + x_{n-1} y_n$. This is clearly a generator of $I_\cyc$ which is central in $\Pi$. Hence by Corollaries \ref{Istar} and \ref{EJRcoro}, $I_\cyc \subseteq \bigcap_{S \in \brick \Pi} \ann S$.

  (b) and (c) are immediate consequences of (a) and Corollary \ref{Istar}.
 \end{proof}
 Because of this proposition, from the point of view of this paper, we can restrict our study to $\overline{\Pi}$.

 We recall that an algebra presented by a quiver and relations $kQ / I$ is \emph{gentle} if
 \begin{itemize}
  \item every $x \in Q_0$ has at most two incoming and at most two outgoing arrows;
  \item the ideal $I$ is generated by paths of length $2$;
  \item for any $q \in Q_1$, there is at most one $q' \in Q_1$ with $qq' \notin I$;
  \item for any $q \in Q_1$, there is at most one $q' \in Q_1$ with $q'q \notin I$;
  \item if $q, q', q'' \in Q_1$ with $t(q') = t(q'') = s(q)$ and $q' \neq q''$, we have $q'q \notin I$ or $q''q \notin I$;
  \item if $q, q', q'' \in Q_1$ with $s(q') = s(q'') = t(q)$ and $q' \neq q''$, we have $qq' \notin I$ or $qq'' \notin I$.
 \end{itemize}

Then we get the following.

  \begin{proposition} \label{Anbrick}\
   \begin{enumerate}[\rm(a)]
    \item The algebra $\overline{\Pi}$ is gentle.
    \item There is a commutative diagram of bijections:
     \[\xymatrix{
       & & \itrigid \Pi \ar[r]_\sim^-{\text{\ref{brick-rigid}(a)}} \ar[d]^\wr_{\overline{\Pi} \otimes_\Pi -} & \brick \Pi \ar@{=}[d] \\
        \SS \ar[r]_-\sim^-{X_{-}} & \ind \overline{\Pi} \ar@{=}[r] & \itrigid \overline{\Pi} \ar@{=}[r] & \brick \overline{\Pi}
      }\]
   \end{enumerate}
 \end{proposition}

 \begin{proof}
   (a) As, clearly, $\overline{\Pi} = k \overline{Q} / I_\cyc$, this is immediate from the definition.

   (b) As $\overline{\Pi}$ is gentle, it is special biserial. Therefore, by \cite{WW}, an indecomposable $\overline{\Pi}$-module $X$ is a string module or a band module. It is an easy verification that $\SS$ is the set of strings, hence string modules are exactly the $X_S$'s. It is also easy that there are no bands. Thus we obtain the bijection from $\SS$ to $\ind \overline{\Pi}$.
Using this, it is immediate that
$\ind \overline{\Pi} = \brick \overline{\Pi}$ holds.
In particular, the bijection of Proposition \ref{brick-rigid}(a) becomes an equality $\itrigid \overline{\Pi} = \brick \overline{\Pi}$. We have $\brick \overline{\Pi} = \brick \Pi$ by Proposition \ref{pipibar}(c).

Let $T \in \itrigid \Pi$. By Proposition \ref{tau-rigid gives tau-rigid}(b), $\Fac T \cap \mod {\overline{\Pi}} = \Fac (\overline{\Pi} \otimes_\Pi T)$, so  by Theorem \ref{annihilate contract0}(a), the labels of the arrow starting at $\Fac T \in \cjirr(\tors \Pi)$ and the arrow starting at $\Fac (\overline{\Pi} \otimes_\Pi T) \in \cjirr(\tors \overline{\Pi})$ coincide. Thus the diagram commutes.
  \end{proof}

  \begin{proof}[Proof of Theorem~\ref{brmodpi}]
   (a) It follows from Proposition~\ref{Anbrick}.

   (b) This is shown in Corollary~\ref{elemforc}.
  \end{proof}

 \begin{proposition}\label{A injective}
 Let $\Pi$ be a preprojective
 algebra of type $A$ and let $\overline\Pi$ be as above.
 \begin{enumerate}[\rm (a)]
  \item We have an isomorphism of lattices $\Ideals \UUU \to \Ideals \overline{\Pi}$ sending $\S$ to $\spanv_k \S$.
  \item The map $\eta_{\overline\Pi}:\Ideals\overline\Pi\to\Con W$ is injective.
  \item The morphism $\eta_{\overline \Pi}$ is an isomorphism of lattices $\Ideals\overline\Pi\cong\AlgCon \Pi$. 
 \end{enumerate}
 \end{proposition}
 \begin{proof}
  (a) It is an immediate consequence of Lemma~\ref{ideals of Pi}.

  (b) For $\S \in \Ideals \UUU$, by Theorem \ref{annihilate contract0}(a), bricks contracted by $\eta_{\overline\Pi}(\spanv_k \S)$ are those that are  not annihilated by $\spanv_k \S$. In particular, $\S$ corresponds to the set of uniserial bricks that are contracted by $\eta_{\overline\Pi}(\spanv_k \S)$. Using (a), it implies that $I \in \Ideals \overline{\Pi}$ is determined by $\eta_{\overline\Pi}(I)$.

  (c) By definition $\AlgCon \Pi$ is the image of $\eta_{\overline{\Pi}}$. So we have a bijection by (b). It is order-preserving, so it has to be an isomorphism of lattices.
 \end{proof}

We now prove Theorem~\ref{AlgQuot A}.

\begin{proof}[Proof of Theorem~\ref{AlgQuot A}]
(a) This is Proposition \ref{pipibar}(a).

 (b) This is Proposition \ref{A injective}.

 (c) By Propositions \ref{pipibar}, \ref{algconquot} and Corollary \ref{Istar}, $\eta_\Pi(I) = \eta_\Pi(J)$ if and only if $\eta_{\overline{\Pi}}(I + I_\cyc) = \eta_{\overline{\Pi}}(J + I_\cyc)$. By Proposition \ref{A injective}(b), this happens if and only if $I + I_\cyc = J + I_\cyc$. By Lemma \ref{ideals of Pi}, we have $I + I_\cyc = \spanv_k ((I + I_\cyc) \cap \UUU) \oplus I_\cyc$ and $J + I_\cyc = \spanv_k ((J + I_\cyc) \cap \UUU) \oplus I_\cyc$. Therefore, $\eta_\Pi(I) = \eta_\Pi(J)$ if and only if $(I + I_\cyc) \cap \UUU = (J + I_\cyc) \cap \UUU$.

 To conclude the proof, it suffices to prove that $(I + I_\cyc) \cap \UUU = I \cap \UUU$ (and, by symmetry, $(J + I_\cyc) \cap \UUU = J \cap \UUU$). 
By Lemma \ref{ideals of Pi}, we have $\Pi=\spanv_k(\UUU)\oplus I_\cyc$ and $I+I_\cyc=\spanv_k(I\cap\UUU)\oplus I_\cyc$. Thus the desired equality holds.
\end{proof}

We conclude by proving the part of Theorem \ref{double ji} that is specific for type $A$. Before that, we give an explicit description of double join irreducible elements of $W = \sym_{n+1}$. For $1 \le i \le j \le n$, we define $d_{i, j} := s_i s_{i+1} \cdots s_j$ and for $1 \le j \le i \le n$, we define $d_{i, j} := s_i s_{i-1} \cdots s_j$ where the $s_i$ are the standard generators of $\sym_{n+1}$.

\begin{proposition}\label{A double ji}
 Double join-irreducible elements of $\sym_{n+1}$ are exactly elements $d_{i, j}$ for $1 \le i, j \le n$.
\end{proposition}

\begin{proof}
Let $\sigma \in \sym_{n+1}$ be double join-irreducible. Let $s_j$ be the rightmost simple reflection in a reduced word for $\sigma$.
Since $\sigma$ is join-irreducible, this simple reflection is unique.
If $\sigma\ne s_j$, then since $\sigma$ is double join-irreducible, the rightmost simple reflection of $\sigma s_j$ must also be unique.
This simple reflection cannot commute with $s_j$, or else there would be two possible rightmost simple reflections for $\sigma$.
Thus, this reflection is $s_{j-1}$ or $s_{j+1}$.
By symmetry, we can assume that it is $s_{j-1}$.

If $\sigma \ne s_{j-1}s_j$, then consider the previous simple reflection in a reduced word for $\sigma$. As before, because $\sigma s_j$ is join-irreducible, it does not commute with $s_{j-1}$, hence it is $s_{j-2}$ or $s_j$. If it was $s_j$, as $s_j s_{j-1} s_j = s_{j-1} s_j s_{j-1}$, it would contradict the uniqueness of $s_j$. So the only possibility is $s_{j-2}$. In particular, it is unique, hence $\sigma s_j s_{j-1}$ is also join-irreducible, so $\sigma s_j$ is double join-irreducible. Then, by an immediate induction, we get $\sigma = d_{i, j}$ for some $i \le j$.
\end{proof}

Using Proposition \ref{A double ji}, we associate to each double join-irreducible element $d_{i, j}$ of $W$ the non-revisiting path $u_{d_{i,j}} \in \UUU$ starting at $i$ and ending at $j$. Via the lattice isomorphism $W \cong \tors \overline{\Pi}$ that sends $w$ to $\overline \Pi \otimes_\Pi I_{w w_0}$, the unique arrow pointing from $d_{i, j}$ in $\Hasse W$ is labelled by $X_{u_{d_{i,j}}}$.

\begin{proposition} \label{realdji}
 Let $J$ be a set of double join-irreducible elements of $W$. We consider the ideal $I = (u_\sigma)_{\sigma \in J}$ of $\overline{\Pi}$. Then we have $\eta_{\overline{\Pi}}(I) = \con J$.
\end{proposition}

\begin{proof}
 For $\sigma \in J$, the uniserial module $X_{u_\sigma}$ is not annihilated by $u_\sigma$, so it is not annihilated
by $I$. Therefore, by Theorem \ref{annihilate contract0}(a), $\sigma$ is contracted by $\eta_{\overline{\Pi}}(I)$. Consider now $S \in \brick \overline{\Pi}$ that is contracted by $\eta_{\overline{\Pi}}(I)$. By Theorem \ref{annihilate contract0}(a) again, we have $I S \neq 0$. So there exists $\sigma \in J$ such that $u_\sigma S \neq 0$. As $u_\sigma$ is a non-revisiting path, it implies that $X_{u_\sigma}$ is a subfactor of $S$, so by Theorem \ref{brmodpi}(b), we get that $S$ is contracted by $\con J$.
\end{proof}

\begin{proof}[Proof of Theorem \ref{double ji} \textup{(iii)} $\Rightarrow$ \textup{(i)}]
 It follows from Proposition \ref{realdji}.
\end{proof}

\subsection{Combinatorial realizations}

We now discuss the combinatorics of algebraic congruences and quotients in type $A_n$.
Specifically, we describe which arrows are contracted by a given algebraic congruence, and describe the quotient explicitly as a subposet of the weak order.

Recall that $s_\ell$ is the transposition $(\ell, \ell+1)$ and that the arrows in $\Hasse \sym_{n+1}$ are $\sigma \to \tau$ such that $\sigma=\tau s_\ell$ for $\ell$ with $\tau(\ell) < \tau(\ell+1)$.
 It is immediate that $\sigma \in \sym_{n+1}$ is join-irreducible if and only if \begin{equation} \sigma(1) < \sigma(2) < \cdots < \sigma(\ell) > \sigma(\ell+1) < \sigma(\ell+2) <\cdots < \sigma(n+1) \label{dscjirr} \end{equation} for some $\ell \in \{1, 2, \dots, n\}$.

 The following observation is an easy consequence of \cite[Section 6.1]{IRRT}. We fix~$\Pi$ as in Section~\ref{preprojA}.
 Until the end of this subsection, in order to get an isomorphism of partially ordered sets between $\sym_{n+1}$ and $\sttilt \Pi$ (see Theorem~\ref{Mizuno's theorem}), we identify $\sigma \in \sym_{n+1}$ with $I_{\sigma w_0} \in \sttilt \Pi$, where $w_0$ is the longest element in $\sym_{n+1}$ (\emph{i.e.} $w_0(i) = n+2-i$).
 \begin{proposition}[Corollary of \cite{IRRT}] \label{bricksym}
  Let $\sigma$ be join-irreducible and $\ell$ as before.
  Then the arrow starting at $\sigma$ in $\Hasse \sym_{n+1}$, that is $\sigma \to \sigma s_\ell$, is labelled by the following brick in $\mod \Pi$, depicted in composition series notation:
  \[ \scriptsize
   \left(\boxinminipage{\xymatrix@C=.4cm@R=.5cm{
    \sigma(\ell+1) \ar[r] & \sigma(\ell+1) + 1 \ar@{..>}[r] & \sigma(\ell+2)-1 \\
    \sigma(\ell+2) \ar[urr] \ar[r] & \sigma(\ell+2) + 1  \ar@{..>}[r] & \sigma(\ell+3)-1 \\
    \sigma(\ell_M) \ar@{..>}[urr] \ar[r] & \sigma(\ell_M) + 1 \ar@{..>}[r] & \ell_M-1
   }}\right) = \left(
   \boxinminipage{\xymatrix@C=.4cm@R=.5cm{
    \sigma(\ell_m)-1 \ar@{..>}[drr] \ar[r] & \sigma(\ell_m) - 2 \ar@{..>}[r] & \ell_m \\
    \sigma(\ell-1) - 1 \ar[drr] \ar[r] & \sigma(\ell-1) - 2 \ar@{..>}[r] & \sigma(\ell-2) \\
    \sigma(\ell) - 1 \ar[r] & \sigma(\ell) - 2 \ar@{..>}[r]  & \sigma(\ell-1)
   }}\right)
  \]
  where $\ell_M$ is the biggest index satisfying $\sigma(\ell_M) < \ell_M$ and $\ell_m$ is the smallest satisfying $\sigma(\ell_m) > \ell_m$. Notice that $\ell_M = \sigma(\ell)$ and $\ell_m = \sigma(\ell+1)$.
 \end{proposition}

To reformulate Proposition \ref{bricksym}, and justify the equality of the two string modules, the label of $\sigma \to \sigma s_\ell$ corresponds to the non-revisiting walk supported by vertices $\ell_m = \sigma(\ell+1), \ell_m + 1, \dots, \ell_M-1 = \sigma(\ell) - 1$, traveling through the arrow $x_{i-1} = (i-1 \to i)$ if $i \in \sigma(\{1, 2, \dots, \ell\})$ and through the arrow $y_i = (i-1 \leftarrow i)$ if $i \in \sigma(\{\ell+1, \ell+2, \dots, n+1\})$.

Just before Proposition \ref{A double ji}, we defined double join-irreducible permutations $d_{i, j} \in \sym_{n+1}$ for $1 \le i, j \le n$.
Given $d_{i, j}$ and a permutation $\sigma$, define a \newword{$d_{i,j}$-pattern} in $\sigma$ to be a pair $\sigma(\ell) \sigma(\ell+1)$ with $\sigma(\ell)>\sigma(\ell+1)$ such that
 \[ [i+1, j] \subseteq \sigma([1, \ell-1]) \text{ if } i \le j \quad \text{and} \quad [j+1, i] \subseteq \sigma([\ell+2, n+1]) \text{ if } i \ge j.\]
We say that $\sigma$ \emph{avoids} $d_{i,j}$ if it contains no $d_{i,j}$-pattern.
The following is a special case of \cite[Corollary~4.6]{arcs} and Proposition \ref{realpidown}.

\begin{theorem} \label{type A bottoms}
Let $D$ be a set of double join-irreducible elements of $\sym_{n+1}$ and let $\Theta_D$ be the smallest congruence on $\sym_{n+1}$ that contracts the elements of $D$.
Then the quotient $\sym_{n+1}/\Theta_D$ is isomorphic to the subposet of $\sym_{n+1}$ induced by the permutations $\sigma$ that avoid $d_{i,j}$ for all $d_{i, j}\in D$.
\end{theorem}

We can also say explicitly which arrows of $\Hasse \sym_{n+1}$ are contracted by $\Theta_D$.
The following theorem is a consequence of Theorem~\ref{type A bottoms} and \cite[Theorem~2.4]{arcs}.

\begin{theorem} \label{type A contract}
Let $D$ be a set of double join-irreducible elements of $\sym_{n+1}$ and let $\Theta_D$ be the smallest congruence on $\sym_{n+1}$ that contracts the elements of $D$.
Then an arrow $\sigma \to \tau$ with $\sigma=\tau s_\ell$ is contracted by $\Theta_D$ if and only if $\sigma(\ell) \sigma(\ell+1)$ is a $d_{i,j}$-pattern.
\end{theorem}

\section{Cambrian and biCambrian lattices} \label{camblat}
In this section, we use the results of this paper to re-derive the known connection between hereditary algebras of Dynkin type and Cambrian lattices.
We also give an algebraic/lattice-theoretic proof of another known fact, namely that each Cambrian lattice is a sublattice of the weak order.
Both of our proofs bypass the combinatorics of sortable elements, which is needed in the previously known proofs. We also analyze the biCambrian congruence of Barnard and Reading \cite{BR} and show that it is algebraic.

\subsection{A representation-theoretic interpretation of Cambrian lattices}\label{camb subsec}
Let $Q$ be a simply-laced Dynkin quiver and let $W$ be the corresponding Weyl group.
A \newword{Coxeter element} of $W$ is an element $c$ obtained as the product in any order of the generators $S=\set{s_1,\ldots, s_n}$.
The quiver $Q$ defines a Coxeter element given by an expression $c=s_{i_1}s_{i_2}\cdots s_{i_n}$ such that if there is an arrow $i\leftarrow j$ in $Q$ then $s_i$ appears before $s_j$ in the expression $s_{i_1}s_{i_2}\cdots s_{i_n}$.
There may be several expressions having this property, but they all define the same Coxeter element of $W$ because if $i$ and $j$ are not related by an arrow of $Q$, the generators $s_i$ and $s_j$ commute.
Conversely a Coxeter element $c$ uniquely determines an orientation of the Dynkin diagram such that an edge $i$ --- $j$ is oriented $i\leftarrow j$ if $s_i$ precedes $s_j$ in some (equivalently, every) reduced word for $c$.

We use $Q$ (or equivalently $c$) to define a lattice congruence $\Theta_c$ on $W$ called the \newword{$c$-Cambrian congruence}.  
We consider the set $\E_c := \{s_j s_i \to s_j  \mid i \leftarrow j \in Q_1\}$ of arrows of $\Hasse W$ and the congruence $\Theta_c := \con \E_c$.
The full set of arrows contracted by $\Theta_c$ can be computed using polygon forcing as in Section~\ref{lat prelim}.

The Cambrian congruence $\Theta_c$ is illustrated in the left picture of Figure~\ref{camb cong fig} for $W=\sym_4$ and $c=s_2s_1s_3$.
(The edges contracted by $\Theta_c$ are doubled.)
Thus $\Theta_c$ is the smallest congruence on $\sym_4$ contracting the arrows $2314 \to 2134$ and $1423 \to 1243$.
\begin{figure}
\[\xymatrix@R=.9cm@C=.5cm@!=0cm{
 & & & & & 4321\\
 & & & 4312 \ar@{<-}[urr] & & 4231 \ar@{<-}[u] & & 3421\ar@{<-}[ull] \\
 & 4132 \ar@{<-}[urr] & & 4213 \ar@{<=}[urr] & & 3412 \ar@{<-}[ull] \ar@{<-}[urr] & & 2431 \ar@{<=}[ull] & & 3241 \ar@{<-}[ull] \\
 1432 \ar@{<=}[ur] & & 4123 \ar@{<-}[ul] \ar@{<-}[ur] & & 2413 \ar@{<=}[ul] \ar@{<=}[urrr] & & 3142 \ar@{<-}[ul] & & 2341 \ar@{<-}[ul] \ar@{<-}[ur] & & 3214 \ar@{<=}[ul] \\
 & 1423 \ar@{<-}[ul] \ar@{<=}[ur] & & 1342 \ar@{<-}[ulll] \ar@{<-}[urrr] & & 2143 \ar@{<=}[ul] & & 3124 \ar@{<-}[ul] \ar@{<-}[urrr] & & 2314 \ar@{<=}[ul] \ar@{<-}[ur] \\
 & & & 1243 \ar@{<=}[ull] \ar@{<-}[urr] & & 1324 \ar@{<-}[ull] \ar@{<-}[urr] & & 2134 \ar@{<-}[ull] \ar@{<=}[urr] \\
 & & & & & 1234 \ar@{<-}[ull] \ar@{<-}[u] \ar@{<-}[urr]
} \quad
\xymatrix@R=.9cm@C=.5cm@!=0cm{
 & & & & & 4321\\
 & & & 4312 \ar@{<-}[urr] & &  & & 3421\ar@{<-}[ull] \\
 & & & & & & 3412 \ar@{<-}[ulll] \ar@{<-}[ur] & &  &   \\
 1432 \ar@{<-}[uurrr] & & & & & & 3142 \ar@{<-}[u] & &  & & 3214 \ar@{<-}[uulll] \\
 & & & 1342 \ar@{<-}[ulll] \ar@{<-}[urrr] & & 2143 \ar@{<-}[uuuu] & & 3124 \ar@{<-}[ul] \ar@{<-}[urrr] & &  \\
 & & & 1243 \ar@{<-}[uulll] \ar@{<-}[urr] & & 1324 \ar@{<-}[ull] \ar@{<-}[urr] & & 2134 \ar@{<-}[ull] \ar@{<-}[uurrr] \\
 & & & & & 1234 \ar@{<-}[ull] \ar@{<-}[u] \ar@{<-}[urr]
}
\]
\caption{A Cambrian congruence and Cambrian lattice}
\label{camb cong fig}
\end{figure}

The quotient $W/\Theta_c$ is called the \newword{$c$-Cambrian lattice}.
The Cambrian lattice $W/ \Theta_c$ for $W=\sym_4$ and $c=s_2s_1s_3$ is drawn on the right of Figure~\ref{camb cong fig}.
As a special case of Proposition \ref{realpidown}, the Cambrian lattice is isomorphic to the subposet $\smash{\pidown^c\, W}$ of $W$ consisting of bottom elements of $\Theta_c$-classes.
These bottom elements were characterized combinatorially in \cite[Theorems 1.1, 1.4]{sort_camb} as the \emph{$c$-sortable elements}.
By definition, an element of $W$ is $c$-sortable if it admits a reduced expression $u_1 u_2 \dots u_\ell$ where, for each $i = 1, \dots, \ell-1$, $u_{i+1}$ is a subword of $u_i$ (\emph{e.g.} $s_2 s_3 s_5$ is a subword of $s_1 s_2 s_3 s_4 s_5$) and $u_1$ is a subword of a reduced expression $u$ for $c$. 

The connection between torsion classes and Cambrian lattices was established in \cite{IngTho}:

\begin{theorem}\label{IT-th} Let $Q$ be a quiver of simply-laced Dynkin type, and
$c$ the corresponding Coxeter element.  Then $\tors kQ$
is isomorphic to the $c$-Cambrian lattice.
\end{theorem}

This theorem was proved by showing that $\tors kQ$ is isomorphic to the sublattice of $W$ consisting of the $c$-sortable elements.
We will now give a direct representation-theoretical argument in Theorem~\ref{kQ camb} using the lattice-theoretic definition of the $c$-Cambrian lattice rather than the combinatorial realization via $c$-sortable elements.
Let $\Pi = \Pi_Q$ be a preprojective algebra and $I$ be the ideal $(a^* \mid a \in Q_1)$ of $\Pi$. Then, we identify $kQ$ with $\Pi/I$ and consider the canonical projection \[\phi: \Pi \to \Pi/I = k Q.\]
\begin{theorem} \label{kQ camb}
The congruences $\Theta_c$ and $\Theta_I$ of $W \cong \tors \Pi$ coincide. Thus, there is a lattice isomorphism $W/\Theta_c \cong \tors kQ$  making the following square commute:
\[
\xymatrix@R=0.5cm{
W \ar[r] \ar[d] & W/\Theta_c \ar@{.>}[d] \\
\tors \Pi \ar[r]^-{\overline{(-)}}& \tors kQ.
}\]
\end{theorem}

We prepare now for the proof of Theorem \ref{kQ camb}.

\begin{lemma}[{Corollary of \cite[Proposition 6.4]{H}, \cite[Corollary 3.19]{J}}] \label{reduchered}
 Let $A$ be a finite-dimensional hereditary algebra. If $(M, P) \in \trigidpair A$ then $\WW(M, P)$ is equivalent to $\mod H$ where $H$ is a finite-dimensional hereditary algebra.
\end{lemma}

\begin{proof}
 We have $\WW(M, P) = {}^\perp (\tau M) \cap P^\perp \cap M^\perp$. First, up to replacing $A$ by $A/(e)$, where $e$ is the idempotent that corresponds to the projective $P$, we can suppose that $P = 0$. Then, as $A$ is hereditary, by Auslander-Reiten duality, we have \[\WW(M, 0) = \{X \in \mod A \mid \Ext^1_A(M, X) = \Hom_A(M, X) = 0\}.\] We have $\Ext^1_A(M, M) = 0$, so if $M$ is indecomposable, again because $A$ is hereditary, by Kac's Theorem, we get $\End_A(M) \cong k$. Hence, the result follows \cite[Proposition 6.4]{H} if $M$ is indecomposable.

 If $M$ is not indecomposable, the result is proven by induction on the number of indecomposable direct summands of $M$, using that rigid objects of $\mod A$ are rigid in $\WW(M', 0)$ for an indecomposable direct summand $M'$ of $M$.
\end{proof}

\begin{lemma} \label{prepcamb1}
 Let $Q$ be a finite union of Dynkin quivers. Let $\{S_1, S_2\}$ be a semibrick of $\mod k Q$ such that $\Ext^1_{k Q}(S_1, S_2) \neq 0$ and $\dim S_1 + \dim S_2 \ge 3$. Then one of the following holds in $\mod k Q$:
 \begin{itemize}
  \item There is a semibrick $\{S_1', S_1'', S_2\}$ and an exact sequence $0 \to S_1' \to S_1 \to S_1'' \to 0$.
  \item There is a semibrick $\{S_1, S_2', S_2''\}$ and an exact sequence $0 \to S_2' \to S_2 \to S_2'' \to 0$.
 \end{itemize}
\end{lemma}

\begin{proof}
 First of all, if $\# Q_0 \leq 2$, then $Q$ is of type $A_1 \times A_1$ or $A_2$ and there is no semibrick $\{S_1, S_2\}$ with $\dim S_1 + \dim S_2 \ge 3$. We start with the case $\# Q_0 = 3$. As $\Ext^1_{k Q}(S_1, S_2) \neq 0$, there is a non-split extension $0 \to S_2 \to S \to S_1 \to 0$. By Lemma \ref{lemma extbricks}, $S$ is a brick with $\dim S=\dim S_1+\dim S_2 \geq 3$. Thus, Q has to be of type $A_3$, and $\{\dim S_1, \dim S_2\}=\{1,2\}$. Thus the simple $kQ$-modules form the desired semibrick.

 \begin{figure}
  \[
   \xymatrix@R=1.1cm@C=2cm@!=0cm{
    & & & (M_0^+, P_0) \ar@{..}@`{p+(40,85), p+(40,50), p+(40,5)}[dddddddd] \ar@{..}@`{p+(-50,85), p+(-50,50), p+(-50,5)}[dddddddd] & & \\
    & & (M, P) \ar[dl]|{\bcirl{S_1}} \ar[dr]|{\bcirl{S_2}} \\
    & (M_1, P_1) \ar@{..}[d] & & (M_2, P_2) \ar@{..}[d] \\
    & \ar[dr]|{\bcirl{S_2}} & & \ar[dl]|{\bcirl{S_1}} \\
    & & (M', P') \ar[d] \\
    & & (M'', P'') & (M''', P''') \ar[dl]|{\bcirl{S_1}} \ar[rd]|{\bcirl{S_2}} \\
    & & \ar@{..}[d] & & \ar@{..}[d] \\
    & & (M_2', P_2') \ar[dr]|{\bcirl{S_2}} & & (M_1', P_1') \ar[dl]|{\bcirl{S_1}}  \\
    & & & (M_0^-, P_0^-)
   }
  \]
  \caption{Hasse quiver of $\WW(M_0, P_0)$} \label{illustproofcamb}
 \end{figure}

 Let us return to the general case. We illustrate the following reasoning in Figure \ref{illustproofcamb}. By Proposition \ref{brick-rigid}(b), there exists $(M, P) \in \ttiltpair k Q$ such that $S_1 \oplus S_2 = M/\rad_{\End_{kQ}(M)} M$. Equivalently, $\Fac M$ is the smallest torsion class $\T(S_1, S_2)$ containing $S_1$ and $S_2$. In particular, there are exactly two arrows $q_1: (M, P) \to (M_1, P_1)$ and $q_2: (M, P) \to (M_2, P_2)$ starting at $(M,P)$ in $\Hasse(\ttiltpair k Q)$, $q_1$ being labelled by $S_1$ and $q_2$ by $S_2$. We consider the polygon $[(M', P'), (M, P)]$ where $(M', P') = (M_1, P_1) \meet (M_2, P_2)$.

 As at least one of $S_1$ and $S_2$ is not simple and labels an arrow pointing toward $(M', P')$, by Proposition \ref{topbottomlatsimpleforce}, we get that $(M', P') \neq (0, kQ)$. So there exists an arrow $(M', P') \to (M'', P'')$ in $\Hasse(\tors k Q)$. Let $(M_0,P_0)$ be the biggest common direct summands of $(M,P)$, $(M',P')$ and $(M'',P'')$.  By mutation theory, $(M_0, P_0)$ has exactly $\# Q_0 - 3$ non-isomorphic indecomposable direct summands.

 Let $(M_0^+, P_0)$ be the Bongartz completion of $(M_0, P_0)$ and $(M_0^-, P_0^-)$ be its co-Bongartz completion. The interval $[(M_0^-, P_0^-), (M_0^+, P_0)]$ is a $3$-polytope as defined in Definition \ref{lpolyt}. By Theorem \ref{theorem filtorth}(a), $\WW := \WW(M_0, P_0)$ is a wide subcategory of $\mod kQ$ and by Theorem \ref{theorem filtorth}(e), the labels of arrows of $\Hasse [(M_0^-, P_0^-), (M_0^+, P_0)]$ are exactly the bricks that are contained in $\WW$.
 The set $\S$ of simple objects of $\WW$ is a semibrick with $\# \S = 3$. We will prove that $\S$ is the desired semibrick.

 Suppose first that $\{S_1, S_2\} \subseteq \S$. We have a polygon $[(M_0^-,P_0^-),(M''',P''')]$ containing two arrows ending at $(M_0^-,P_0^-)$ labelled by $S_1$ and $S_2$ and two arrows starting at $(M''',P''')$ labelled by $S_1$ and $S_2$. Since $\Fac M''' \supset T(S_1,S_2)= \Fac M$, $(M,P)$ belongs to the polygon $[(M_0^-,P_0^-),(M''',P''')]$, and hence $(M,P)=(M''',P''')$ and $(M',P')=(M_0^-,P_0^-)$ hold. This is a contradiction since $(M',P')$ is not the minimum element of $[(M_0^-,P_0^-),(M_0^+,P_0)]$. So $\{S_1, S_2\} \not\subseteq \S$.

 By Lemma \ref{reduchered}, there is an equivalence $\psi: \WW \cong \mod k Q'$ for a quiver $Q'$. Moreover, as $\WW \subseteq \mod k Q$ has finitely many isomorphism classes of indecomposable objects, by Gabriel's theorem, $Q'$ is a union of Dynkin quivers with $\# Q'_0 = 3$. As $\{S_1, S_2\} \not\subseteq \S$, it means that $\psi(S_1)$ and $\psi(S_2)$ are not both simple in $\mod Q'$. So $\dim \psi(S_1) + \dim \psi(S_2) \geq 3$ and the result has already been proven in $\WW \cong \mod k Q'$.
\end{proof}

Recall that, as $kQ$ is hereditary, the Auslander-Reiten translation $\tau: \mod kQ \to \mod kQ$ is a functor. Then, we recall the following alternative description of $\mod \Pi$:

\begin{definition}[\cite{ringel}]
 We define the category $(\mod k Q)(1, \tau)$ in the following way: an object of $(\mod kQ)(1, \tau)$ is a pair $(M, \alpha)$ where $M \in \mod k Q$ and $\alpha : M \to \tau M$ is a morphism. A morphism from $(M, \alpha)$ to $(N, \beta)$ is a morphism $f: M \to N$ in $\mod kQ$ satisfying $\beta \circ f = (\tau f) \circ \alpha$.
\end{definition}

\begin{theorem}[{\cite[Theorem B]{ringel}}] \label{thmdescmodpi}
 There is an equivalence of categories between $\mod \Pi$ and $(\mod k Q)(1, \tau)$ such that, via this equivalence,
 \begin{enumerate}[\rm (a)]
  \item The restriction $\mod \Pi \to \mod kQ$ along $kQ\hookrightarrow\Pi$ is given by $(M, \alpha) \mapsto M$;
  \item The restriction $\mod kQ \to \mod \Pi$ along $\Pi\twoheadrightarrow kQ$ is given by $M \mapsto (M, 0)$.
 \end{enumerate}
\end{theorem}

\begin{lemma} \label{prepcamb2}
 Let $S \in \brick \Pi$ such that $S \notin \mod k Q$ and $\dim S \ge 3$. Then there exists a semibrick $\{S_1, S_2\}$ of $\mod \Pi$ and a short exact sequence $0 \to S_1 \to S \to S_2 \to 0$ such that at least one of $S_1$ and $S_2$ is not in $\mod k Q$.
\end{lemma}

\begin{proof}
 By \cite[Theorem 1.2]{IRRT}, all bricks in $\mod \Pi$ are stones. Then they are clearly $k$-stones. So, by Proposition \ref{splitbrick}, there exist a semibrick $\{S'_1, S'_2\}$ in $\mod \Pi$ and a short exact sequence
 \[\chi: 0 \to S'_1 \to S \to S'_2 \to 0\]
 such that $\dim \Ext^1_\Pi(S_2', S_1') = 1$.

 If at least one of $S'_1$ and $S'_2$ is not in $\mod kQ$, we have our conclusion, so we suppose that $S'_1, S'_2 \in \mod k Q$. As $\dim \Ext^1_\Pi(S'_2, S'_1) = 1$ and $S \notin \mod kQ$, we have $\Ext^1_{k Q} (S'_2, S'_1) = 0$ and $\chi$ splits as an exact sequence of $kQ$-modules. So, via the equivalence of Theorem \ref{thmdescmodpi}, $\chi$ can be rewritten as
 \[\chi: 0 \to (S'_1, 0) \xto{u} \left(S'_1 \oplus S'_2, \alpha\right) \xto{v} (S'_2, 0) \to 0.\]
  As $u$ and $v$ are morphisms, we have
 \[\alpha = \begin{bmatrix} 0 & \beta \\ 0 & 0\end{bmatrix},\]
 where $\beta$ is a morphism from $S'_2$ to $\tau S'_1$. As $\chi$ does not split, $\beta \neq 0$. Hence by Auslander-Reiten duality for hereditary algebras, \[\dim \Ext^1_{k Q} (S_1', S_2') = \dim \Hom_{k Q}(S_2', \tau S'_1) \ge 1.\]

 So we can apply Lemma \ref{prepcamb1}. By symmetry, we suppose that we are in the first case: There is a semibrick $\{S''_1, S''_2, S'_2\}$ and a short exact sequence
 \[\xi: 0 \to S''_1 \xto{f} S'_1 \xto{g} S''_2 \to 0\]
 in $\mod k Q$. Applying $\Hom_{k Q}(S'_2, -)$ to $\xi$ gives the exact sequence
 \[0 = \Hom_{k Q}(S'_2, S''_2) \to \Ext^1_{k Q} (S'_2, S''_1) \to \Ext^1_{k Q} (S'_2, S'_1) = 0 \to \Ext^1_{k Q} (S'_2, S''_2) \to 0,\]
 so $\Ext^1_{k Q} (S'_2, S''_1) = 0 = \Ext^1_{k Q} (S'_2, S''_2)$. Let us consider two possibilities, depending on the image $g \chi \in \Ext^1_\Pi(S_2', S''_2)$ of $\chi$.
 \begin{itemize}
  \item If $g\chi \neq 0$. In this case, we get the following Cartesian diagram where the last row does not split:
  \[\xymatrix{
   & S''_1 \ar@{=}[r] \ar@{^{(}->}[d]_f & S''_1 \ar@{^{(}->}[d] \\
   \chi: & S'_1 \ar@{->>}[d]_g \ar@{^{(}->}[r] & S \ar@{->>}[d] \ar@{->>}[r] & S'_2 \ar@{=}[d] \\
   g\chi: & S''_2 \ar@{^{(}->}[r] & S_2 \ar@{->>}[r] & S'_2.
  }\]
 As $\{S''_1, S''_2, S'_2\}$ is a semibrick, Lemma \ref{lemma extbricks} implies that $S_2$ is a brick. We also deduce that $\{S_1'', S_2\}$ is a semibrick. As $\Ext^1_{k Q} (S'_2, S''_2) = 0$, we have $S_2 \notin \mod k Q$, so the middle vertical sequence satisfies our requirements.
  \item If $g \chi = 0$. In this case, we get the following Cartesian diagram:
  \[\xymatrix{
   \chi': & S''_1 \ar@{^{(}->}[r] \ar@{^{(}->}[d]_f & S_1 \ar@{->>}[r] \ar@{^{(}->}[d] & S'_2 \ar@{=}[d] \\
   \chi = f\chi': & S'_1 \ar@{->>}[d]_g \ar@{^{(}->}[r] & S \ar@{->>}[d] \ar@{->>}[r] & S'_2  \\
    & S''_2 \ar@{=}[r] & S''_2
  }\]
 As before, $\{S_1, S_2''\}$ is a semibrick. As $\Ext^1_{k Q} (S'_2, S''_1) = 0$, we get $S_1 \notin \mod k Q$, so the middle vertical sequence satisfies our requirements.  \qedhere
 \end{itemize}
\end{proof}

Then we can prove Theorem \ref{kQ camb}:

\begin{proof}[Proof of Theorem \ref{kQ camb}]
 For an arrow $i \to j$ of $\overline{Q}$, we denote by $X_{i, j}$ the indecomposable $\Pi$-module of length $2$ with top $\top A e_i$ and socle $\top A e_j$. By definition, $\Theta_c = \con E$ where $E := \{X_{j, i} \mid (j \to i) \notin Q_1\}$. For $X \in E$, we have $I X \neq 0$, hence by Theorem~\ref{annihilate contract0}(a), $X$ is contracted by $\Theta_{I}$, so $\Theta_{I} \ge \Theta_c$.

 By Theorem \ref{annihilate contract0}(a), bricks $S$ contracted by $\Theta_{I}$ are exactly the ones satisfying $S \notin \mod k Q$. So to prove that $\Theta_{I} \le \Theta_c$, it suffices to prove that such a brick $S$ is contracted by $\Theta_c$. We argue by induction on $\dim S$. If $\dim S = 2$, then $S \in E$, so $S$ is contracted by $\Theta_c$. Otherwise, $\dim S \ge 3$ and by Lemma \ref{prepcamb2}, there is a short exact sequence $0 \to S_1 \to S \to S_2 \to 0$ such that $\{S_1, S_2\}$ is a semibrick of $\Pi$ that is not in $\mod k Q$. So, by the induction hypothesis, $\Theta_c$ contracts $S_1$ or $S_2$. By Theorem \ref{theorem forcing}, both $S_1$ and $S_2$ force $S$, so $\Theta_c$ contracts $S$.
\end{proof}

Recall from Section~\ref{lat prelim} that for a general congruence on a finite lattice $L$, the set of bottom elements of congruence classes are a join-sublattice of $L$, but need not be a sublattice of $L$.
The bottom elements can fail to be a sublattice even when $L$ is $W$ and even when the congruence is algebraic.
As an example, one can consider the algebraic congruence generated by contracting the double join-irreducible element $s_1s_2s_3$ in $\sym_4$.
However, the $c$-Cambrian congruence is an exception:
the following is \cite[Theorem~1.2]{sort_camb}.

\begin{theorem}\label{camb sublat}
For any Coxeter element $c$ of $W$, the set $\smash{\pidown^c\, W}$, which consists of $c$-sortable elements, is a sublattice of $W$.
\end{theorem}

We now give a new, representation-theoretical proof of Theorem~\ref{camb sublat}.

As before, we consider the projection $\phi: \Pi \twoheadrightarrow k Q$. We also consider the natural inclusion $i: kQ \hookrightarrow \Pi$. It gives a fully faithful functor $\mod k Q \hookrightarrow \mod \Pi$ that we denote implicitly or by $M \mapsto {}_\Pi M$ if necessary, and a faithful functor $\mod \Pi \to \mod kQ$ that we denote by $X \mapsto {}_{kQ} X$. We start with a lemma.

\begin{lemma} \label{filtPi}
 The following hold:
 \begin{enumerate}[\rm(a)]
  \item Let $X, Y \in \mod \Pi$. Then there is an exact sequence
   \[0 \to \Hom_\Pi (X, Y) \to \Hom_{kQ}({}_{kQ} X, {}_{kQ} Y) \xto{u} \Hom_{kQ}({}_{kQ} X, \tau({}_{kQ} Y))\]
   where the first map is the canonical inclusion and $u(f) = (\tau f) \circ \alpha - \beta \circ f$ where $X = (M, \alpha)$ and $Y = (N, \beta)$ via the equivalence of Theorem \ref{thmdescmodpi}.
  \item Let $X \in \mod \Pi$. There exists a filtration
  \[0 = X_0 \subsetneq X_1 \subsetneq X_2 \subsetneq \cdots \subsetneq X_{\ell-1} \subsetneq X_\ell = X\]
 of $X$ by $\Pi$-submodules such that ${}_\Pi({}_{kQ} X) \cong \bigoplus_{i=1}^\ell X_{i} / X_{i-1}$.
 \end{enumerate}

\end{lemma}

\begin{proof}
 (a) This is an immediate consequence of Theorem \ref{thmdescmodpi}.

 (b) We prove the statement by induction on $\dim X$. Consider the indecomposable direct summand $N$ of ${}_{kQ} X$ that is leftmost in the Auslander-Reiten quiver of $\mod kQ$. Then, $\Hom_{kQ}({}_{kQ} X, \tau N) = 0$, so by (a), the canonical projection of $kQ$-modules $\pi: X \twoheadrightarrow N$ is a morphism of $\Pi$-modules. By the induction hypothesis, $\Kernel \pi$ has a filtration of the desired form, which is easily extended to $X$.
\end{proof}

\begin{proof}[Proof of Theorem~\ref{camb sublat}]
 It is proven in Theorem~\ref{relating the maps}\eqref{meet semi homs}\eqref{join semi homs} that $\phi^-$ is a morphism of join-semilattices and $i_-$ is a morphism of meet-semilattices. We also know from Proposition~\ref{surj nice}(e) that the image of $\phi^-$ is $\smash{\pidown^c\, W}$.

 We conclude by proving that $\phi^- = i_-$ so that $\phi^-$ is a morphism of lattices and $\smash{\pidown^c\,W}$ is a sublattice of $W$.
 Let $\TT \in \tors kQ$.
 By Proposition~\ref{surj nice}\eqref{maxmin}, $\phi^-(\TT)$ is minimal such that $\overline{\phi^-(\TT)} = \TT$, hence $\phi^-(\TT)$ is the minimal torsion class in $\mod \Pi$ containing $\TT$. By definition, $i_-(\TT)$ consists of all $\Pi$-modules that are in $\TT$ as $kQ$-modules. Hence $\phi^-(\TT) \subseteq i_-(\TT)$ clearly. Moreover, by Lemma~\ref{filtPi}, any $X \in i_-(\TT)$ is filtered by modules in $\TT$, hence is in $\phi^-(\TT)$. It concludes the proof.
\end{proof}

\subsection{The bipartite biCambrian congruence}
Let $W$ be a finite Coxeter group.
The \newword{bipartite biCambrian congruence} on $W$, defined in \cite{BR}, is the lattice congruence $\Theta_{\bic}=\Theta_c\wedge \Theta_{c^{-1}}$, where $\Theta_c$ is the
Cambrian congruence from Section~\ref{camb subsec} and $c$ is a bipartite Coxeter element. 
We will prove \cite[Conjecture 2.11]{BR}, which asserts that $W/\Theta_\bic$ is Hasse-regular.

\begin{theorem}   
Suppose $W$ is a simply-laced finite Coxeter group and $\Pi$ is the associated preprojective algebra.
Identifying $W$ with $\tors\Pi$ as before, $\Theta_\bic$ coincides with $\Theta_I$, where $I$ is the ideal in $\Pi$ generated by all paths of length 
two.  
\end{theorem}

\begin{proof}
The condition that $c$ is bipartite means that the corresponding
orientation of the Dynkin diagram $Q$ has only sinks and sources.  As we have
showed, the bricks contracted by $\Theta_c$ are the bricks which are not
representations of $Q$, while the bricks contracted by $\Theta_{c^{-1}}$ are
those which are not representations of $Q^{\op}$.
Consider a
path $p$ of length two in the doubled quiver.  It necessarily uses
one arrow from $Q$ and one arrow from $Q^{\op}$.  Therefore, for $S$ a brick, if $pS\ne 0$,
then $S$ is neither a representation of $Q$ nor a representation of $Q^{\op}$.
Thus, $S$ is contracted by both $\Theta_{c}$ and $\Theta_{c^{-1}}$, and
thus is contracted by $\Theta_\bic=\Theta_c\wedge \Theta_{c^{-1}}$.

On the other hand, for a brick $S$, the following properties are equivalent:
\begin{itemize}
\item $S$ is not contracted by $\Theta_I$,
\item $IS=0$,
\item The Loewy length of $S$ is at most 2,
\item $S$ is a representation of $Q$ or of $Q^{op}$.
\end{itemize}
Thus, if $S$ is not contracted by $\Theta_I$, then $S$ is a representation of
$Q$ or of $Q^{\op}$, and therefore is not contracted by $\Theta_c$ or
by $\Theta_{c^{-1}}$ respectively, and thus is not contracted by $\Theta_\bic$.
\end{proof}

The following corollary is now immediate from Corollary~\ref{Hasse-regular}.  
\begin{corollary} \label{bicat cor}
If $W$ is a simply-laced finite Coxeter group, then $W/\Theta_\bic$ is Hasse-regular.
\end{corollary}
Corollary~\ref{bicat cor} is the simply-laced case of \cite[Conjecture 2.11]{BR}.
The general case follows by a folding argument, as explained in the type-B case in \cite[Section~3.6]{BR}.

\begin{remark}\label{bic rem}
As pointed out in \cite{BR}, when $c$ is not bipartite, $\Theta_c\wedge \Theta_{c^{-1}}$ is less well-behaved.  
From our point of view, the point is that, for more general $c$, the congruence $\Theta_c\wedge \Theta_{c^{-1}}$ need not be algebraic. 
Indeed, in type $A_3$, for $c=s_1s_2s_3$, it is apparent in \cite[Figure 4]{BR} that the fan associated to $\Theta_c\wedge\Theta_{c^{-1}}$ is not simplicial, and thus the quotient of $W$ modulo this congruence is not Hasse-regular.
\end{remark}

\section*{Acknowledgments}
We thank the Mathematical Sciences Research Institute, and the organizers of the Workshop ``Cluster Algebras in Combinatorics, Algebra, and Geometry'' where this collaboration began.
We also thank the Mathematisches Forschungsinstitut Oberwolfach and the organizers of the workshops ``Cluster Algebras and Related Topics'' and ``Representation Theory of Quivers and Finite Dimensional Algebras'' for helping this collaboration to continue; we likewise thank the Centre International de Rencontres Math\'ematiques and the organizers of ``ARTA VI'' and ``Cluster algebras twenty years on''. 
We thank the Institut Mittag-Leffler, and the organizers of the program on ``Representation Theory'', which provided support and excellent working
conditions.
We thank David Speyer, Gordana Todorov, and Friedrich Wehrung for helpful conversations.
We thank Universit\"at Bielefeld for their hospitality. 
We thank the referee for their helpful comments.
We thank Lorenzo Molena for pointing out an error in an earlier version of the proof of Proposition~\ref{arrdeteq}.


\begin{thebibliography}{IRRT}
\bibitem[AAC]{AAC} T. Adachi, T. Aihara, A. Chan, \emph{Classification of two-term tilting complexes over Brauer graph algebras}, 
Math. Z. 290 (2018), no. 1-2, 1--36. 
\bibitem[AIR]{AIR} T. Adachi, O. Iyama, I. Reiten, \emph{$\tau$-tilting theory}, Compos. Math. 150 (2014), no. 3, 415--452.

\bibitem[AGT]{AGT}
K. V. Adaricheva, V. A. Gorbunov, V. I. Tumanov,
\emph{Join-semidistributive lattices and convex geometries}
Adv. Math. 173 (2003), no. 1, 1--49. 

\bibitem[AN]{AN} K. Adaricheva, J. B. Nation, \emph{Classes of semidistributive lattices}, Chapter 3 of Lattice theory: special topics and applications, vol. 2, ed. G. Gr\"{a}tzer and F. Wehrung, Birkh\"{a}user, Cham, Switzerland 2016.

\bibitem[A]{A} S. Asai, \emph{Semibricks}, Int. Math. Res. Not. IMRN 2020, no. 16, 4993--5054. 
\bibitem[ARS]{ARS} M. Auslander, I. Reiten, S. O. Smal\o, \emph{Representation theory of Artin algebras}, Cambridge Studies in Advanced Mathematics, 36. Cambridge University Press, Cambridge, 1997.
\bibitem[AS]{AS} M. Auslander, S. O. Smal\o, \emph{Almost split sequences in subcategories}, J. Algebra 69 (1981), no. 2, 426--454.
\bibitem[BCZ]{BCZ} E. Barnard, A. T. Carroll, S. Zhu, \emph{Minimal inclusions of torsion classes}, Algebr. Comb. 2 (2019), no. 5, 879--901. 
\bibitem[BR]{BR} E. Barnard, N. Reading, \emph{Coxeter biCatalan combinatorics}, J. Algebraic Combin. 47 (2018), no. 2, 241--300. 
Comment. Math. Helv. 60 (1985), no. 3, 392--399.
\bibitem[B]{Birkhoff}
G. Birkhoff,
\textit{Lattice theory.}
Corrected reprint of the 1967 third edition. American Mathematical Society Colloquium Publications 25.
American Mathematical Society, Providence, RI, 1979.
\bibitem[BB]{BB} A. Bj\"orner, F. Brenti, \emph{Combinatorics of Coxeter groups}, Graduate Texts in Mathematics, 231. Springer, New York, 2005.
\bibitem[BIRS]{BIRS} A. B. Buan, O. Iyama, I. Reiten, J. Scott, \emph{Cluster structures for 2-Calabi-Yau categories and unipotent groups}, Compos. Math. 145 (2009), no. 4, 1035--1079.
\bibitem[CH]{CH}
I. Chajda and R. Hala\v{s},
\emph{Sectionally pseudocomplemented lattices and semilattices},
Advances in Algebra, 282--290, World Sci. Publ., River Edge, NJ, 2003. 

\bibitem[CD]{CD}
P. Crawley and R. P. Dilworth,
\emph{Algebraic theory of lattices},
Prentice-Hall, Englewood Cliffs, NJ, 1973. 

\bibitem[DIJ]{DIJ} L. Demonet, O. Iyama, G. Jasso, \emph{$\tau$-tilting finite algebras, bricks, and $g$-vectors}, Int. Math. Res. Not. IMRN 2019, no. 3, 852--892.
\bibitem[DF]{DF} H. Derksen, J. Fei, \emph{General presentations of algebras},
Adv. Math. 278 (2015), 210--237.
\bibitem[EJR]{EJR} F. Eisele, G. Janssens, T. Raedschelders, \emph{A reduction theorem for $\tau$-rigid modules}, Math. Z. 290 (2018), no. 3-4, 1377--1413.

\bibitem[FZ]{FZ1}
S. Fomin, A. Zelevinsky, \emph{Cluster algebras. I. Foundations}, J. Amer. Math. Soc. 15 (2002), no. 2, 497--529.

\bibitem[GM]{GM} A. Garver, T. McConville,
\emph{Lattice Properties of Oriented Exchange Graphs and Torsion Classes}, Algebr. Represent. Theory 22 (2019), no. 1, 43--78. 
\bibitem[G]{LTF}
G. Gr\"{a}tzer,
\textit{Lattice theory: foundation.}
Birkh\"{a}user/Springer Basel AG, Basel, 2011.
\bibitem[H]{H} D. Happel, \emph{Triangulated categories in the representation theory of finite-dimensional algebras},
London Mathematical Society Lecture Note Series, 119. Cambridge University Press, Cambridge, 1988. x+208 pp. 

\bibitem[HM]{HM}
H. P. Hoang and T. M\"{u}tze,
\emph{Combinatorial generation via permutation languages. II. Lattice congruences},
to appear in Israel J. Math.

\bibitem[IT]{IngTho}
C. Ingalls and H. Thomas,
\textit{Noncrossing partitions and representations of quivers.}
Compos. Math. 145 (2009), no. 6, 1533--1562.
\bibitem[IRRT]{IRRT} O. Iyama, N. Reading, I. Reiten, H. Thomas,
\emph{Lattice structure of Weyl groups via representation theory of preprojective algebras}, Compos. Math. 154 (2018), no. 6, 1269--1305. 
\bibitem[IRTT]{IRTT} O. Iyama, I. Reiten, H. Thomas, G. Todorov, \textit{Lattice structure of torsion classes for path algebras}, Bull. LMS 47 (2015),
no. 4, 639--650.
\bibitem[IZ]{IZ} O. Iyama, X. Zhang, \emph{Classifying $\tau$-tilting modules over the Auslander algebra of $K[x]/(x^n)$}, J. Math. Soc. Japan 72 (2020), no. 3, 731--764.
\bibitem[J]{J} G. Jasso, \textit{Reduction of $\tau$-tilting modules and torsion pairs}, Int. Math. Res. Not. IMRN 2015, no. 16, 7190--7237.
\bibitem[K]{Kase}
R. Kase,
\textit{Weak orders on symmetric groups and posets of support $\tau$-tilting modules},  Internat. J. Algebra Comput. 27 (2017), no. 5, 501--546.
\bibitem[KL]{KL} K. Keimel, J. Lawson, \emph{Continuous lattices}, Chapter 1 of Lattice theory: special topics and applications, vol. 1, ed. G. Gr\"{a}tzer and F. Wehrung, Birkh\"{a}user, Cham, Switzerland 2014.
\bibitem[M]{M} Y. Mizuno, \emph{Classifying $\tau$-tilting modules over preprojective algebras of Dynkin type}, Math. Z. 277 (2014), no. 3-4, 665--690.
\bibitem[N]{NationNotes} J. B. Nation, \emph{Revised Notes on Lattice Theory}, lecture notes available at \href{http://www.math.hawaii.edu/~jb/books.html}{http://www.math.hawaii.edu/$\sim$jb/books.html}.
\bibitem[PPP]{PPP} Y. Palu, V. Pilaud, P.G. Plamondon, \emph{Non-kissing complexes and $\tau$-tilting for gentle algebras}, 
Mem. Amer. Math. Soc. 274 (2021), no. 1343.
\bibitem[R1]{con_app}
N.~Reading,
\textit{Lattice congruences, fans and Hopf algebras.}
J. Combin. Theory Ser. A 110 (2005) no.~2, 237--273.
\bibitem[R2]{sort_camb}
N.~Reading,
\textit{Sortable elements and Cambrian lattices.}
Algebra Universalis 56 (2007), no. 3-4, 411--437.
\bibitem[R3]{arcs}
N.~Reading,
\textit{Noncrossing arc diagrams and canonical join representations.}
SIAM J. Discrete Math. 29 (2015), no. 2, 736--750.

\bibitem[R4]{regions} N. Reading, \emph{Lattice theory of the poset of regions}, Chapter 9 of Lattice theory: special topics and applications, vol. 2, ed. G. Gr\"{a}tzer and F. Wehrung, Birkh\"{a}user, Cham, Switzerland 2016.
\bibitem[R5]{regions2} N. Reading, \emph{Finite Coxeter groups and the weak erder}, Chapter 10 of Lattice theory: special topics and applications, vol. 2, ed. G. Gr\"{a}tzer and F. Wehrung, Birkh\"{a}user, Cham, Switzerland 2016.

\bibitem[Ri1]{ringelk} C.M. Ringel, \emph{Representations of K-species and bimodules}, J. Algebra, 41 (1976), 269--302.

\bibitem[Ri2]{ringel} C.M. Ringel, \emph{The preprojective algebra of a quiver}, Algebras and modules, II (Geiranger, 1996), 467--480, CMS Conf. Proc., 24, Amer. Math. Soc., Providence, RI, 1998.

\bibitem[Si]{Si} L. Silver, \emph{Noncommutative localizations and applications}, J. Algebra, 7 (1967), 44--76.

\bibitem[Ste]{Ste} B. Stenstr\"om, \emph{Rings of quotients}, Die Grundlehren der Mathematischen Wissenschaften, Band 217. Springer-Verlag, New York-Heidelberg, 1975. viii+309 pp. 

\bibitem[Sto]{Sto} H.H. Storrer, \emph{Epimorphic extensions of non-commutative rings}, Comment. Math. Helv., 48 (1973), 72--86.

\bibitem[WW]{WW} B. Wald, J. Waschb\"usch, \emph{Tame biserial algebras}, J. Algebra  95  (1985), no. 2, 480--500.
\end{thebibliography}
\end{document}